\documentclass{article}

  \usepackage[main, final]{neurips_2026}

\usepackage[utf8]{inputenc} 
\usepackage[T1]{fontenc}    
\usepackage{hyperref}       
\usepackage{url}            
\usepackage{booktabs}       
\usepackage{amsfonts}       
\usepackage{nicefrac}       
\usepackage{microtype}      
\usepackage{xcolor}         
\usepackage{float}

\usepackage{graphicx}
\usepackage{subcaption}

\usepackage{amsmath}
\usepackage{amssymb}
\usepackage{mathtools}
\usepackage{amsthm}
\usepackage{multirow}
\newcommand{\rev}[1]{#1}
\usepackage[capitalize,noabbrev]{cleveref}

\theoremstyle{plain}
\newtheorem{theorem}{Theorem}[section]
\newtheorem{proposition}[theorem]{Proposition}
\newtheorem{lemma}[theorem]{Lemma}
\newtheorem{corollary}[theorem]{Corollary}
\theoremstyle{definition}
\newtheorem{definition}[theorem]{Definition}

\theoremstyle{remark}
\newtheorem{remark}[theorem]{Remark}

\usepackage[disable,textsize=tiny]{todonotes}

\newcommand{\E}{\mathbb{E}}

\newcommand{\Var}{\text{Var}}

\title{Random Neural Network Expressivity for Non-Linear Partial Differential Equations}

\author{
 Muhammed Ali Mehmood \\ Department of Mathematics \\ Imperial College London \\ UK \\   \texttt{muhammed.mehmood21@imperial.ac.uk}
\And
Lukas Gonon \\ School of Computer Science \\ University of St. Gallen \\ Switzerland \\
\texttt{lukas.gonon@unisg.ch}}

\begin{document}

\maketitle

\begin{abstract}
Neural networks with randomly generated hidden weights (RaNNs) have been extensively studied, both as a standalone learning method and as an initialization for fully trainable deep learning methods. In this work, we study RaNN expressivity for learning solutions to non-linear partial differential equations (PDEs). Despite their widespread use in practical applications, a rigorous theoretical understanding of the approximation properties of RaNNs in this context remains limited.
Here, we derive error bounds for RaNN approximations to time-dependent Sobolev functions and obtain a dimension-free approximation rate $\frac{1}{2}$ for sufficiently regular functions. 
We apply our results to two important classes of non-linear PDEs: Porous Medium Equations and Compressible Navier-Stokes Equations, \rev{whose solutions have the required Sobolev regularity under sufficiently smooth initial data, and obtain bounds on the PINN residual of RaNN approximations to these solutions. }
\end{abstract}

\section{Introduction}
Partial Differential Equations (PDEs) are foundational to our understanding of the natural world, with applications across all areas of science and engineering. Many complex phenomena are modelled by non-linear PDEs (e.g.\ Navier-Stokes, Schr\"odinger, Porous medium equations), which exhibit disorderly behaviour that renders them intractable to classical analytic/numerical approaches. Therefore, it is crucial to develop numerical methods for solving non-linear PDEs efficiently. In the past years, a variety of deep learning methods for solving PDEs have been introduced and analysed. Neural networks with randomly generated hidden weights (RaNNs) play an important role in many of these methods; either as standalone learning method or as initialization for fully trainable deep neural networks. In both cases, when employing these methods, a precise understanding of the approximation error is crucial for controlling the overall error.  

In this paper, we are concerned with the expressivity of RaNNs for learning solutions of non-linear PDEs. To tackle this problem, we derive approximation error bounds for time-dependent Sobolev functions, which encompass the solution spaces for many important non-linear PDEs. Our obtained bounds show that sufficiently regular Sobolev functions can be approximated by RaNNs at a dimension-independent rate of $\frac{1}{2}$. 
We then apply our results to two important classes of non-linear PDEs: Porous Medium Equations (PME) and Compressible Navier-Stokes Equations, whose solutions have the required Sobolev regularity under sufficiently smooth initial data. We prove that with high probability, there exist RaNNs whose \rev{PDE residual} for Physics-Informed Neural Networks (PINNs) \cite{raissi2019physics} is bounded by a constant multiple of $N^{-1}$ for these PDEs, where the constant is given explicitly in terms of the Sobolev norms of the solution, the sup-norms of the network and its derivatives and the problem parameters (Corollaries \ref{cor:PME} and \ref{cor:NS}). In particular, our results provide quantitative approximation guarantees for RaNN-based PINNs for learning non-linear PDEs, as have been extensively studied in computational experiments (cf.\ the references below).  We complement our theoretical analysis by numerical experiments for the two benchmark PDEs\rev{. These experiments fit the read-out weights by regression onto a reference solution, which isolates the approximation error bounded by Theorem~\ref{thm1} from optimisation and collocation errors. Experiments with Gaussian hidden weights indicate that the rate may also hold beyond the weight distribution assumed in the main result of Theorem~\ref{thm1}. Note that our results concern expressivity, and we make no claim of practical superiority of RaNNs over fully trained networks.}


\subsection{Related Works}
In recent years, a variety of deep learning-based methods for solving PDEs have been introduced, 
addressing the challenges of classical mesh-based methods such as finite difference methods. 
Seminal works include \cite{SirignanoSpiliopoulos2017}, \cite{EHanJentzen2017CMStat} \cite{raissi2019physics}, \cite{EYu2017}. We refer, e.g., to the survey articles \cite{beck2020overview, Germainetal2021, MR4457972, Gononetal2024} for an extensive overview and further references on deep learning methods for PDEs and their theoretical foundations. 

PINNs \citep{raissi2019physics} constitute a flexible and widely applicable deep learning-based approach for solving PDEs. PINNs reframe the problem as training a neural network to solve the PDE, by minimising a loss function that encodes the PDE residual along with the boundary/initial conditions.
While this approach has been demonstrated to be highly effective in many settings \cite{MR4412280,HuShukla2024}, for non-linear PDEs the loss landscape may become exceptionally complex. 
This has motivated the use of RaNN-based PINN methods, for which several recent studies have carried out extensive empirical experiments, see, e.g., \cite{dwivedi2020physics,shang2023randomized, shang2024randomized, sun2024local,wang2024extreme,ying2024accurate,doi:10.1177/10812865251362165,datar2024fast,chen2022bridgingtraditionalmachinelearningbased,Nelsen2020}.  

RaNNs \cite{huang2006extreme,rahimi2007random,rahimi2008uniform} are neural networks with randomly generated hidden weights. RaNNs have been used both as standalone learning methods and as means for studying the effects of random initialization for neural networks trained using gradient-based optimization \cite{braun2024convergence,CRR2018}. Generalization properties of random feature models have been studied in \cite{RudiRosasco2017,MM19,2023ErrorBoundsForLearningWithVectorValuedRandomFeatures,2023ATheoreticalAnalysisOfTheTestErrorOfFiniteRankKernelRidgeRegression}. In addition to RaNN-based PINNs, many other RaNNs-based methods have been developed for solving PDEs \cite{Nelsen2020,gonon2023random,jacquier2023random,neufeld2025full}. More broadly, RaNNs and related random feature models have demonstrated state-of-the-art performance and speed across various tasks \cite{2023SWIMSamplingWeightsNeuralNetworks, 2023Hydra, gattiglio2024randnet, 2024RandomRepresentationsOutperformOnlineContinuallyLearnedRepresentations,zozoulenko2025random}. 

The reduced number of trainable parameters of RaNNs in comparison to fully trainable models results in a simpler training phase with a reduced computational cost, potentially at the expense of lower expressivity. Therefore, a precise theoretical understanding of RaNN approximation capabilities is crucial. 
Quantitative approximation properties of RaNNs for functions in the associated reproducing kernel Hilbert space have been studied in \cite{rahimi2008uniform,JMLR:v18:14-546,sun2018approximation}. For smoothness-based function classes, RaNN approximation error bounds were derived in \cite{gonon2023approximation, gonon2023random} using Barron-type representations and further extended in \cite{neufeld2023universal,de2025approximation}. In the context of PDEs, \cite{gonon2023random} obtains a full RaNN learning error analysis free from the curse of dimensionality  for a class of linear PDEs. In all these results, the random weight distribution is fixed (e.g.\ a uniform, normal or Student-t distribution). In contrast, RaNNs also appear as a means of proof for deriving deterministic approximation bounds \cite{barron2002universal,Barron1994ApproximationAE,BarronKlusowski2018,SiegelXu2020}, with RaNN weight distributions depending on the function to be approximated.

While the approximation results in  \cite{gonon2023approximation, gonon2023random} and \cite{neufeld2023universal,de2025approximation} allow to control the RaNN approximation errors in uniform, mean-squared or Sobolev-norms, respectively, in the context of non-linear PDEs these results would either require strict information on the solution (e.g.\ finiteness in Barron-ridgelet norms and decay on the Fourier transform of $u$) which is typically not known, or the bounds would be applicable only for approximating PDE solutions at a fixed point in time. However, for time-dependent non-linear PDEs, solutions often have significantly different behaviour in time versus space (e.g.\ solutions to Navier-Stokes or semi-linear heat equations). In contrast, our results allow to handle time-dependent functions in mixed Sobolev spaces, as arise in the context of non-linear PDEs. This is especially important because neural networks with randomised architectures have proven to be very effective, even for complex tasks such as solving non-linear PDEs \cite{datar2024fast}.

\subsection{Contributions}
In this paper, we provide RaNN approximation error bounds tailored to the context of time-dependent, non-linear PDEs. We will denote the width of a RaNN by $N$. Our paper makes the following contributions:  
\begin{enumerate} 
    \item \textbf{RaNN approximation bound with dimension-independent rate $N^{-1/2}$:} We derive RaNN approximation error bounds for sufficiently regular time-dependent Sobolev functions (Theorem~\ref{thm1}). These functions encompass the solution spaces for many important non-linear PDEs. Our unbiased RaNN estimator approximates functions in mixed Sobolev norms $H^p_t H^q_x$ at the rate $N^{-1/2}$ independently of dimension, while only requiring minimal extra regularity in time and space.
    \item \textbf{Implications for non-linear PDEs:} We showcase the implications of our bounds on two important classes of non-linear PDEs: Porous Medium Equations (PME) and Compressible Navier-Stokes Equations. 
    \rev{In particular, our bound supplies the approximation error term in a PINN error decomposition, and yields explicit high-probability bounds on the PINN residual of RaNNs for PME and compressible Navier-Stokes. In numerical experiments that isolate the approximation error, the measured error decays at a rate consistent with $N^{-1/2}$, in agreement with the bounds.}
\end{enumerate}
To prove these results, we obtain a specific ridgelet-based representation for $L^2$ functions (Proposition \ref{prop:u-rep}) and a higher-order Plancherel-type estimate (Lemma~\ref{lemma:parseval}) that connects Sobolev regularity of $u$ with its ridgelet transform.
Our obtained bounds may serve as building block for generalization error analyses of RaNNs for PINN-based learning as obtained for deterministic networks in \cite{MR4793683,Mishraestimatepde2022,MR4565573,ar2409.17938}.

\section{Preliminaries}

\subsection{Random Neural Networks}
A random neural network (RaNN) is a fully connected neural network with one hidden layer in which the weights are randomly sampled, leaving only the output weights trainable. 
In this paper, we are interested in studying solutions to PDEs, which typically treat time as a separate dimension. Therefore we will 
consider time-dependent random neural networks of the following form.
\begin{definition} \label{def:rann}
A time-dependent random neural network of width $N$ is a function $u^{\tau, \mathbf{a}, \mathbf{b}}_W : \mathbb{R} \times \mathbb{R}^d \to \mathbb{R}$ with
 \begin{equation} \label{RANN}
     u_W^{\tau,\mathbf{a}, \mathbf{b}}(t,x) = \sum_{i=1}^N W_i \sigma (\tau_i t + \mathbf{a}_i \cdot x - b_i),
 \end{equation}
where $\sigma : \mathbb{R} \to \mathbb{R}$ is an activation function, $\mathbf{a}_1, ..., \mathbf{a}_N$ are $\mathbb{R}^d$-valued i.i.d.\ random variables, $ b_1, ..., b_N $ are i.i.d.\ random variables in $\mathbb{R}$, $ \tau_1, ..., \tau_N $ are i.i.d.\ random variables in $\mathbb{R}$  and $W_1, ..., W_N$ are trainable weights.\footnote{Formally, $W_i = g_i(\tau_1,\ldots,\tau_N, \mathbf{a}_1, ..., \mathbf{a}_N,b_1,\ldots,b_N)$ for measurable functions $g_i$. }
\end{definition}
A RaNN is used as learning system by optimizing the weights $W_1, ..., W_N$ of $u_W^{\tau,\mathbf{a}, \mathbf{b}}$ with respect to a given loss function. 

\subsection{Physics-informed machine learning}
Physics-informed neural networks (PINNs) have been introduced in \cite{raissi2019physics} as an unsupervised learning method for solving partial differential equations. 
PINNs approximate the solution $u$ to a PDE $\mathcal{L}[u] = 0$  by a neural network $u_\theta$, with $\theta$ representing the trainable parameters. The solution is approximated by minimising a loss function $\mathcal{J}[u_\theta]$ encoding the structure of the PDE --- including initial/boundary conditions --- at collocation points $\{t_p^i, x_p^i\}_{i=1}^{M}$ in the interior domain $(0,T) \times D$, as well as $\{ x_{ic}^i\}_{i=1}^{M}$ , $\{t_{bc}^i\}_{i=1}^{M}$ on the slices $\{t=0\} \times D$ and $(0,T) \times \partial D$ on which the initial data / boundary conditions are defined, respectively.  
For example, if we have a PDE on the 1D domain $(0,T) \times (a,b)$ given by $\mathcal{L}[u] = 0$ with initial condition $u(0,\cdot) = u_0$ and boundary conditions $u(t,a) = u(t,b)$, the PINN loss function is  given by
\begin{equation*}
\begin{aligned}
    \mathcal{J}[u_\theta] := \frac{1}{M} \sum_{i=1}^M |\mathcal{L}[u_\theta]|^2 (t_p^i, x_p^i) &+ \frac{1}{M} \sum_{i=1}^M |u_\theta(0, x_{ic}^i) - u_0(x_{ic}^i)|^2 \\[1ex] &+ \frac{1}{M} \sum_{i=1}^M |u_\theta(t_{bc}^i, b) - u_\theta(t_{bc}^i, a) |^2.
\end{aligned}
\end{equation*}
The parameters $\theta$ are then iteratively updated using a stochastic optimization algorithm. 

\subsection{Ridgelet transform} 
The ridgelet transform $\mathcal{R}_{\psi}u$ of $u : \mathbb{R} \times \mathbb{R}^d \to \mathbb{R}$ with respect to $\psi : \mathbb{R} \to \mathbb{R}$ is given by
\begin{equation} \label{ridge-defn}
       \mathcal{R}_\psi u (\tau, \mathbf{a}, b) := \int_{\mathbb{R}^{d+1}} u(t,\mathbf{x})\psi(\tau t +\mathbf{a}\cdot \mathbf{x}-b)\|(\tau, \mathbf{a})\|^s~dtd\mathbf{x},
    \end{equation} 
    where $\tau \in \mathbb{R}, \mathbf{x} \in \mathbb{R}^d, b \in \mathbb{R}.$
    The factor $\|(\tau, \mathbf{a})\|^s$ appears for convenience in the literature. We will take $s=0$ in this paper. Then the dual ridgelet transform $R_\eta^\dagger T$ of $T:\mathbb{R}^{d+2} \to \mathbb{R}$ with respect to $\eta : \mathbb{R} \to \mathbb{R}$ is defined as \begin{equation} \begin{aligned}
        &\mathcal{R}_\eta^\dagger T(t,\mathbf{x}) := \int_{\mathbb{R}^{d+2}} T(\tau, \mathbf{a},b) \eta(\tau t +\mathbf{a}\cdot \mathbf{x}-b)\|(\tau, \mathbf{a})\|^{-s} d\tau d\mathbf{a} db.
        \end{aligned}
    \end{equation}
The ridgelet transform originates from harmonic analysis in the work of Candès and Donoho \cite{candes1998ridgelets,candes1999ridgelets}. Ridgelets can be described as wavelets in the Radon domain. More precisely, the ridgelet transform is obtained by applying a one-dimensional wavelet transform to the Radon transform of a function. This construction is particularly suited to representing functions with singularities concentrated along hyperplanes, in contrast to classical wavelets which are better suited to point singularities.
Ridgelets are also closely related to ridge functions of the form $x \mapsto \varphi(a \cdot x - b)$ \cite{donoho2001ridge}, which can be interpreted as a single layer neural network. This connection underlies a second area of work linking ridgelet transforms to integral representations of neural networks and their approximation properties. We refer to \cite{murata1996integral, candes1999harmonic, ito1991representation} for early works that developed this research direction. Recent works \cite{unser2023ridges, ongie2019function, sonoda2017neural} have further exploited the connection between ridgelet formulations and neural networks to study approximation rates / reconstruction properties for a broad class of activation functions.
We refer to \cite{sonoda2017neural} for a more comprehensive overview of the ridgelet transform and its properties.
In this paper, we extend the ridgelet-based approximation framework to time-dependent Sobolev spaces, which are the natural function spaces for solutions to evolutionary PDEs. We also show how approximation bounds obtained via ridgelet techniques can be propagated to bounds on the PDE residual, which is the quantity minimised in PINN-type methods.

\subsection{Notation and structure of the paper}
We adopt the shorthand notation $X_t Y_x$ for the Bochner space $X(\mathbb{R}; Y(\mathbb{R}^d))$ (or $X(0,T; Y(D))$, depending on the context). For example, we may write $H^p_t H^q_x$  for the space $H^p(\mathbb{R}; H^q(\mathbb{R}^d))$. We write $\widehat{f}$ to denote the Fourier transform of $f$, with the convention $\widehat{f}(\omega) = \int e^{-i\omega x} f(x)~dx$. We also denote by $\|\cdot\|$ the $\ell^2$ norm. In Section \ref{sec:main}, we prove a representation formula for general time-dependent Sobolev functions, before proving a key inequality that connects the regularity of the ridgelet transform to the original function. We use this to derive our main result Theorem \ref{thm1}. In Section \ref{sec:PDE} we then apply Theorem~\ref{thm1} to RaNN-PINN approximators of two benchmark PDEs. Finally, some numerical experiments are carried out which validate the $N^{-1/2}$ convergence rate. 
\section{RaNN approximations of time-dependent Sobolev functions} \label{sec:main}
Our first step is to obtain an integral representation for $u$ based on the ridgelet transform. This representation will be used later to derive RaNN approximation error bounds for $u$.
\subsection{Obtaining an integral representation}
We first introduce the Lizorkin distribution space $\mathcal{S}_0'(\mathbb{R})$, which is the dual space of $\mathcal{S}_0(\mathbb{R})$; the space of Lizorkin functions. $\mathcal{S}_0(\mathbb{R})$ is a closed subspace of the space of Schwartz functions $\mathcal{S}(\mathbb{R})$ and contains the functions $f \in \mathcal{S}(\mathbb{R})$ such that all moments vanish, i.e. $\mathcal{S}_0(\mathbb{R}^d) = \{ f \in \mathcal{S} : \int_{\mathbb{R}^d} \mathbf{x}^\alpha f(\mathbf{x})~d\mathbf{x}=0 \text{ for any } \alpha \in \mathbb{N}_0^d\}$. The Lizorkin distribution space $\mathcal{S}_0'(\mathbb{R})$ itself includes many common activations such as $\tanh$, $\mathrm{sigmoid}$ and $\mathrm{ReLU}$. We refer to \cite{sonoda2017neural} for a more detailed description of the $\mathcal{S}_0'(\mathbb{R})$ space. In this paper, we will consider activation functions $\sigma : \mathbb{R} \to \mathbb{R}$ that belong to the following subspace of $\mathcal{S}_0'(\mathbb{R})$ for some $k \ge 0$.
\begin{definition} \label{def:Tk}
   Let $k \in \mathbb{N}_0$. We say $\sigma \in \mathcal{S}_0'(\mathbb{R})$ belongs to $\mathcal{T}_k$ if (i) there exists $C_1 > 0$ such that
    \begin{equation}
        \sum_{j=0}^k |\sigma^{(j)}(x)| \le C_1 \quad \forall x \in \mathbb{R},
    \end{equation}
 and (ii) there exists $\delta>0$ and $\beta \in \mathbb{N}_0$ such that $\zeta^\beta \widehat{\sigma}(\zeta) \in C(-\delta, \delta)$ and for any $\alpha \in \mathbb{N}_0$\begin{equation}
        J_\sigma :=    \int_\mathbb{R} \zeta^{2\alpha+\beta} \widehat{\sigma}(\zeta)e^{-\zeta^2/2}~d\zeta  \ne 0. 
    \end{equation}
    
\end{definition}

 Examples of admissible activation functions for which our results hold are $\tanh, \cos$ and $\mathrm{sigmoid}$ (each in $\mathcal{T}_k$ for all $k \ge 0$). This is shown in Remark \ref{rmk:tanh} in the appendix. For $\psi \in \mathcal{S}(\mathbb{R})$ and $m \in \mathbb{N}$ we also define the following (possibly infinite) admissibility constant 
  \begin{equation} \label{psi-condition}
    A_{\psi,m} :=\int_{-1}^1 |\widehat{\psi}(\omega)|^2 |\omega|^{-m} ~d\omega.
    \end{equation} 

We now mention the following integral representation formula, which directly follows from results obtained by \cite{sonoda2017neural} using the theory of ridgelet transforms.
\begin{proposition} \label{prop:u-rep}
    Let $m \ge 0$ and $\sigma \in \mathcal{T}_k$ for some $k \ge 0$. A function $u \in L^1(\mathbb{R}; L^1(\mathbb{R}^d))$ \rev{with $\widehat{u} \in L^1(\mathbb{R}^{d+1})$} can be expressed as
     \begin{equation} \label{u-representation}
        u(t,\mathbf{x}) = \int_\mathbb{R} \int_{\mathbb{R}^{d}} \int_\mathbb{R} (\mathcal{R}_\psi u) (\tau, \mathbf{a}, b) \sigma (\tau t+ \mathbf{a}\cdot \mathbf{x}-b) ~d\tau d\mathbf{a}db,
    \end{equation}
    where $\mathcal{R}_\psi u $ is the ridgelet transform of $u$ with respect to a Schwartz function $\psi \in \mathcal{S}(\mathbb{R})$ such that 
    \begin{equation} \label{psi-growth}
        |\widehat{\psi}(\omega)| \le C|\omega|^{m}  \qquad \forall ~|\omega|<1,
    \end{equation} for some $C>0$ independent of $u$,
    and therefore $A_{\psi, m}<+\infty$. We also have that $(\sigma, \psi)$ is an admissible pair in the sense of \cite{sonoda2017neural}, meaning that the following constant is finite and non-zero:
    \begin{equation}
K_{\psi, \sigma} := \rev{(2\pi)^{d}} \int_\mathbb{R} \frac{\rev{\overline{\hat{\psi}(\zeta)}} \hat{\sigma}(\zeta)}{|\zeta|^{\rev{d+1}}} ~d\zeta.
    \end{equation}
\end{proposition}
The construction of an appropriate $\psi$ is given in the proof of Proposition \ref{prop:u-rep} in Appendix \ref{sec:proofPROP}. From now on, for any activation $\sigma \in \mathcal{T}_k$ and function $u \in L^1(\mathbb{R}; L^1(\mathbb{R}^d))$ we will denote by $\mathcal{R}_\psi u$ the ridgelet transform of $u$ with respect to $\psi \in \mathcal{S}(\mathbb{R})$ constructed according to Proposition \ref{prop:u-rep}.

\subsection{Parseval relation for the ridgelet transform}
 In order to derive RaNN approximation error bounds, we will need a Parseval-type result which connects the regularity of $\mathcal{R}_\psi u$ in parameter space with the regularity of $u$ in cartesian space. For convenience, we give an outline of the proof in Appendix \ref{sec:sketchproofparseval} and the full proof in Appendix \ref{sec:proofparseval}.
\begin{lemma} \label{lemma:parseval}
Suppose $u \in H^{p+1}(\mathbb{R}; H^{q+1}(\mathbb{R}^d))$ for some $p,q \ge 0$, and that $u$ is compactly supported in time and space on $[-2T,2T] \times [-2R,2R]^d$ for some $T,R \ge0$. Then there exists $\psi \in \mathcal{S}(\mathbb{R})$ with
\begin{equation} \label{parseval}
\begin{aligned}
      \int_{\mathbb{R}^{d+2}} |\mathcal{R}_\psi u(\tau, \mathbf{a}, b)|^2 (1+|\tau|^{2})^p(1+\|\mathbf{a}\|^{2})^q
     \cdot (1+b^2) ~d\tau d\mathbf{a} db  \le \mathcal{L}_{\psi}\|u\|_{H^{p+1}(\mathbb{R} ; H^{q+1}(\mathbb{R}^d))}^2,
\end{aligned}
\end{equation}
where, for $M = (2p+2q+d+3)/2$ \rev{and $K_{T,R} := (2T+\lceil p\rceil+1)^2+(2R\sqrt{d}+\lceil q\rceil+1)^2$},
 \begin{equation} \label{Lpsi} \begin{aligned}
    \mathcal{L}_{\psi} := (4\pi + \|\psi\|_{L^2(\mathbb{R})}^2)\rev{(1+4K_{T,R})} +\rev{2(M+1)^2\left(4\pi + \|\psi '\|_{L^2(\mathbb{R})}^2\right)}.
    \end{aligned}
\end{equation}
\end{lemma}
\subsection{Random neural network approximation error bounds} \label{sec:MC}
In this subsection we provide our main result for approximating a function $u \in H^p(0,T; H^q(D))$ using RaNNs. The full proof is deferred to Appendix \ref{sec:proofTHM}.
\begin{theorem} \label{thm1}
    Fix a bounded subset $(0,T) \times D \subset \mathbb{R}_+ \times \mathbb{R}^d$ where $D$ is Lipschitz and let $u \in H^{p+s_1}(0,T; H^{q+s_2}(D))$ for $p, q \ge 0, s_1 >3/2$ and $s_2 >(d+2)/2$. Furthermore, let $\sigma \in \mathcal{T}_{p+q}$. There exist weights $\{W_i\}_{i=1}^N$ such that the following random neural network is an unbiased estimator of $u$:
    \begin{equation*}
        u_N(t,\mathbf{x}) = \sum_{i=1}^N W_i \sigma( \tau_i t + \mathbf{a}_i \cdot\mathbf{x} - b_i)
    \end{equation*} for $ (\tau_i, \mathbf{a}_i, b_i) \sim \pi$, where\begin{equation} \begin{aligned} \label{pi}
        \pi(\tau,\mathbf{a}, b) = \frac{1}{C_\pi} (1 + \tau^2)^{-\lambda_{\tau}}(1+ \|\mathbf{a}\|^2)^{-\lambda_{a}}(1 + b^2)^{-1}, \end{aligned}
    \end{equation} where $ \lambda_\tau=s_1-1, \lambda_a = s_2 - 1$ and $C_\pi$ is the normalisation constant
    \begin{equation} \begin{aligned}
        &C_\pi :=  \int_\mathbb{R}\frac{1}{(1+\tau^2)^{\lambda_\tau}}~d\tau  \int_{\mathbb{R}^d}\frac{1}{(1+\|\mathbf{a}\|^2)^{\lambda_a}} d\mathbf{a}   \int_\mathbb{R}\frac{1}{(1+b^2)}~db.
        \end{aligned}
       \end{equation}
    The neural network $u_N$ satisfies
       \begin{equation} \label{thm-ineq} \begin{aligned}
      \mathbb{E}_\Theta &\left( \| u-u_N \|_{H^{p}(0,T; H^q(D))}^2 \right) \\[1ex] &\le \frac{C_{\Omega} C_\pi \|\sigma^{(p+q)}\|_\infty^2 T|D|\rev{(p+1)\binom{q+d}{d}}
 \mathcal{L}_\psi}{N}  \|u\|_{H^{p+s_1}(0,T; H^{q+s_2}(D))}^2,
      \end{aligned}
    \end{equation}
    where $C_{\Omega}$ is a constant dependent on $p$, $q$, $d$ and the domain, \rev{$\|\sigma^{(p+q)}\|_\infty := \max_{0 \le j \le p+q}\|\sigma^{(j)}\|_\infty$}  and $\mathcal{L}_\psi$ is given by \eqref{Lpsi}.
\end{theorem}
\begin{remark} \label{rmk:smoothness}
Let us note that the space of admissible functions does depend on $d$ through the requirement $s_2 >(d+2)/2$. Nonetheless, for a given function which satisfies the conditions of the theorem, the rate of decay $N^{-1/2}$ with respect to the network width $N$ is independent of $d$. This is an improvement upon the work of \cite{de2025approximation} (see Theorem 3.9) which required $u \in H^{s}(\mathbb{R}^d)$ for $s \ge (d+9)/2$ for estimates in $H^1(\mathbb{R}^d)$ and $H^2(\mathbb{R}^d)$, with slower rates. Moreover, we obtain estimates in a time-dependent norm, whereas previous results only obtained error rates for solutions at a fixed time (see also Proposition 4.24 of \cite{neufeld2023universal}).
\end{remark}
\begin{remark}[Growth of the coefficient with respect to dimension]
    Obtaining a more explicit bound (in terms of $d$) on the coefficient 
    \begin{equation}
        C_{\Omega} C_\pi \|\sigma^{(p+q)}\|_\infty^2 \rev{T|D|(p+1)\binom{q+d}{d}} \mathcal{L}_\psi
    \end{equation}
    appearing in \eqref{thm-ineq} is difficult for general $u$ and arbitrary $p,q$. Nonetheless, we show in Remark \ref{rmk-growth} that \eqref{thm-ineq} holds with a constant that grows at most polynomially in $d$ in the case where $\sigma \in \mathcal{T}_k$ and $u \in H^{p}(\mathbb{R};H^{q}(\mathbb{R}^d))$ (i.e. $u$ is already defined on the full space, support in $[-2T, 2T] \times [-2R, 2R]^d$, $|D| \le 1$).
\end{remark}

\section{Applications to non-linear PDEs} \label{sec:PDE}
We look at two representative non-linear PDEs (the Porous Medium Equation and compressible Navier-Stokes) in order to understand how one can obtain specific asymptotic bounds on the residual loss and approximation error using the structure of the PDE. 

\subsection{Porous Medium Equation}
The porous medium equation (PME) is an important example of a non-linear parabolic PDE that models the flow of gases through porous media. In dimension $d \in \mathbb{N}$ and for $m >0$, a function $u: \mathbb{R}_+ \times \mathbb{R}^d \to \mathbb{R}$ is said to solve the porous medium equation if 
\begin{equation} \label{PME}
\left\{
\begin{aligned}
&\partial_t u - \Delta( u^m) = 0, \text{ on } \mathbb{R}_+ \times \mathbb{R}^d. \\
&u(0,\cdot) = u_0,
\end{aligned}
\right.
\end{equation}
In the case where $u_0$ is bounded below by a positive constant and $u_0 - \bar{c} \in H^k(\mathbb{R}^d)$ for some constant $\bar{c} > 0$ and $k \in \mathbb{N}$, the following result is classically known.

\begin{itemize}
    \item  \cite{vazquez2007porous}:  Consider initial data $u_0$ satisfying
    \begin{equation} \label{u0-PME}
        u_0 - \bar{c} \in H^k(\mathbb{R}^d) \text{ for some constant } \bar{c} > 0, ~~k \in \mathbb{N}, ~~ 0 < c \le u_0 \le C,
    \end{equation}
    Then
     there exists a classical solution $u$ to \eqref{PME} with $c \le u(t,x) \le C$ and
    \begin{equation*}
        u - \bar{c}\in C^\infty((0,\infty) \times \mathbb{R}^d) \cap C([0,\infty); H^k(\mathbb{R}^d)).
    \end{equation*}
\end{itemize}
In practice, we simulate solutions on a bounded domain $(0,T) \times D$. In this case, the smoothness of the solution implies that $u \in H^k((t_0,T) \times D)$ for any $k \ge 0$ and any $0 < t_0 < T$. We can deduce an approximation result on solutions to PME using Theorem \ref{thm1}. Before we state the result, let us note that the classical loss function one would use when training a physics-informed neural network to approximate solution $u$ is
a discretisation of the following metric:
\begin{equation} \label{pme-JPDE}
    \mathcal{J}_{PDE}(u_N) = \int_{(t_0,T) \times D} |\partial_t u_N - \Delta (u^m_N)|^2~dtdx.
\end{equation}

\begin{corollary} \label{cor:PME}
    Suppose $u$ solves \eqref{PME} with initial data $u_0$ satisfying \eqref{u0-PME}. Then there exists a random neural network $u_N(t,x)$ such that on the domain $(t_0,T) \times D$ for any $0 <t_0 < T$:
    \begin{enumerate}
        \item For any $p, q \ge 0$, \begin{equation} \begin{aligned} \label{PME-bound} 
        \mathbb{E}_\Theta(\|u-u_N\|_{H^p_t H^q_x}^2) &\le \frac{C_\Omega C_\pi \|\sigma^{(p+q)}\|_\infty^2 T|D|\rev{(p+1)\binom{q+d}{d}} \mathcal{L}_\psi}{N}\|u\|_{H^{p+s_1}_t H^{q+s_2}_x}^2 \\[1ex] &=: \frac{\mathcal{M}_{\psi}}{N} \|u\|_{H^{p+s_1}_t H^{q+s_2}_x}^2,
        \end{aligned}
    \end{equation}    for any $s_1>3/2,s_2>(d+2)/2$.
    \item For any $\delta \in (0,1)$, with probability $1-\delta$ over the network parameters, the PDE residual \eqref{pme-JPDE} can be bounded as:
      \begin{equation} \label{pme-residual-1} \begin{aligned}
      &\mathcal{J}_{\text{PDE}}(u_N) \le  \frac{C_m(1+L^2+C_{emb}^2\|u\|_{L^\infty_{t}H^{2+k}_x}^2) \mathcal{M}_\psi }{N \delta} \rev{\|u\|_{H^{3}_t H^{3+k}_x}^2},
         \end{aligned}
    \end{equation}for $k >d/2$, where $L := \| u_N\|_{L^\infty_{t,x}}+\|\nabla u_N\|_{L^\infty_{t,x}}+\|\Delta u_N\|_{L^\infty_{t,x}} $ is finite since $u_N$ is smooth. Here, $C_{emb}$ is the constant arising from the Sobolev embedding $H^{k}_x \hookrightarrow L^\infty_{x}$ and $C_m$ is a polynomial in the PME parameter $m$.
    \end{enumerate}
    \rev{In other words, for every $\delta \in (0,1)$ and every width $N$, with probability at least $1-\delta$ over the drawn hidden weights there exist read-out weights for which $\mathcal{J}_{\text{PDE}}(u_N) \le \rev{C_L}/(N\delta)$. Here the constant \rev{$C_L$} depends on $m$ (through $C_m$), on $d$ and $D$ (through $C_{emb}$ and $\mathcal{M}_\psi$), on $T$, on the Sobolev norms of $u$ appearing in \eqref{pme-residual-1}, and on the network through $L$. In particular, the PDE residual decays at rate $1/N$ along any sequence of networks $u_N$ for which $L$ remains bounded.}
\end{corollary}
\rev{This is an existence statement; it does not assert that a particular training algorithm finds these networks.}

\subsection{Compressible Navier-Stokes Equations}

We now look at a more delicate example, which is a system of equations that does not enjoy the instantaneous regularisation property of the PME. The compressible Navier-Stokes equations in dimension $d$ are given by
\begin{equation} \label{NSE}
\left\{
\begin{aligned}
&\partial_t \rho + \text{div}(\rho \mathbf{ u }) = 0, \\
&\partial_t (\rho  \mathbf{ u }) + \text{div} (\rho  \mathbf{ u } \otimes  \mathbf{ u }) - \nabla (\mu(\rho) \text{div}~ \mathbf{ u }) + \nabla p(\rho)=  0,
\end{aligned}
\right.
\end{equation}
on $(0,T) \times D$, where
\begin{equation}
    p(\rho) = \rho^\gamma, ~~\gamma>0 \text{ and } \mu(\rho) = \rho^\alpha, ~\alpha>0.
\end{equation}
The solution is a pair $(\rho, \mathbf{u})$ where the density $\rho : \mathbb{R}_+ \times \mathbb{R}^d \to \mathbb{R}$ is the scalar density and $\mathbf{u} : \mathbb{R}_+ \times \mathbb{R}^d \to \mathbb{R}^d$ is the vector-valued velocity. We look at the one-dimensional setting $D = (0,1)$ with periodic boundary conditions, where global-in-time classical solutions exist under mild assumptions. 
In this case, the following global well-posedness result applies.

\begin{itemize}
    \item Theorem 1.5, \cite{constantin2020compressible}. Given initial data $(\rho_0, u_0)$ satisfying \begin{equation} \label{NSE-ID}
        (\rho_0, u_0) \in H^k(D),~ k\ge3, ~~ 0 < \delta \le \rho_0 \le C,
    \end{equation}
     then there exists a unique solution $(\rho, u)$ on $(0,T)$ to \eqref{NSE} with initial data $(\rho_0, u_0)$ such that
    \begin{equation} \label{nse-class} \begin{aligned}
        &\rho \in C(0,T; H^k(D)), ~~ u \in C(0,T; H^k(D)) \cap L^2(0,T; H^{k+1}(D)).
        \end{aligned}
    \end{equation}
\end{itemize}
We will take the pressureless case $p=0$ and constant viscosity $\mu(\rho) = \mu$ for simplicity, although our computations can be easily extended to handle a more general setting where $p, \mu$ are smooth and convex (e.g. $p(\rho) = \rho^\gamma, \mu(\rho) = \rho^\alpha,$ for $ \alpha,\gamma>0$). In this case, the PINN residual loss is 
the discretisation of the following loss functions:
\begin{equation}
    \begin{aligned}
        &\mathcal{J}_{PDE}^1(\mathbf{v}_N) := \int_{(0,T)\times D} |\partial_t (\rho_N) + \partial_x(\rho_N u_N)|^2~dxdt, \\[1ex]
        &\mathcal{J}_{PDE}^2(\mathbf{v}_N) := \int_{(0,T)\times D} |\partial_t (\rho_N u_N) + \partial_x(\rho_N u_N^2) - \mu\partial_x^2u_N|^2~dxdt.
    \end{aligned}
\end{equation}

To apply Theorem \ref{thm1}, we first define the product norm $ \|(f,g)\|_{H^p_tH^q_x}^2 := \|f\|_{H^p_t H^q_x}^2 + \|g\|^2_{H^p_t H^q_x}.$
We then have the following result whose proof is given in Appendix \ref{sec:proofNS}.
\begin{corollary}\label{cor:NS}
    Suppose $(\rho, u)$ is a solution to \eqref{NSE} in dimension $d=1$ generated by initial data satisfying $(\rho_0, u_0) \in H^k(D),~ \rev{k\ge9}, ~ 0 < \delta \le \rho_0 \le C.$ Then there exists a random neural network $\mathbf{v}_N = (\rho_N, u_N)$ such that
    \begin{itemize}
        \item For any $q \ge 0$,
        \begin{equation}
        \begin{aligned}
             &\mathbb{E}_\Theta(\|(\rho, u) - \mathbf{v}_N\|_{L^2_t H^q_x}^2) \le \frac{\mathcal{M}_\psi}{N}(\|\rho\|_{H^{s_1}_t H^{q+s_2}_x}^2+ \|u\|_{H^{s_1}_t H^{q+s_2}_x}^2 ),
        \end{aligned}
    \end{equation} for $s_1, s_2>3/2,$ where $\mathcal{M}_\psi$ is the coefficient from \eqref{thm-ineq}, later defined as $\mathcal{M}_\psi$ in \eqref{nse-Mpsi}.
    \item For any $\delta \in (0,1)$ with probability $1-\delta$ over the network parameters, the PDE residual loss $\mathcal{J}_{PDE}:= \mathcal{J}_{PDE}^1 + \mathcal{J}_{PDE}^2$ can be bounded as:
        \begin{equation} \label{NSE-residuals} 
        \begin{aligned}
             &\mathcal{J}_{PDE}(\mathbf{v}_N) \le  \left(\frac{ C_L \mathcal{M}_\psi}{N\delta} \right) ( \|\rho\|^2_{H^{3}_t H^{3}_x} + (2\mu +1)\|u\|^2_{H^{3}_t H^{4}_x}),
        \end{aligned}
        \end{equation}
        where $ L := \| (\rho_N, u_N)\|_{W^{1,\infty}_{t,x}}^2 +  \| (\rho, u)\|_{W^{1,\infty}_{t,x}}^2$ is finite since $\rho_N, u_N, \rho, u$ are sufficiently smooth, and \begin{equation*}
       \rev{      C_L := \left[16L^2 +128L+2\mu+2+ 64(L+1)\|u\|_{W^{1,\infty}_{t,x}}^2 \right]. }
        \end{equation*}
    
    \end{itemize}
\begin{remark}
    The choice \rev{$k \ge 9$} ensures that the norms on the right-hand side of \eqref{NSE-residuals} are finite. With this regularity, one can also show that the boundary/initial condition residuals can be controlled similarly to $\mathcal{J}_{PDE}$, using trace inequalities.
\end{remark}
\end{corollary}
\rev{The use of these bounds inside solvers for the non-linear equations, via direct minimisation of the PDE residual or via Picard linearisation, is discussed in Appendix~\ref{apdx:solvers}.}

\section{Numerical illustrations} \label{sec:experiments-main}
In this section, we provide numerical experiments to validate the obtained bounds by studying the effect of network width $N$ on the error in practice.\footnote{Code: \url{https://github.com/M6LI/RaNN-Expressivity-NonLinearPDEs}.} A secondary goal of these experiments is to study whether the convergence rates extend beyond the heavy-tailed distribution considered in Theorem~\ref{thm1} to normally-distributed weights. We consider the Porous Medium Equation (PME) in dimensions $d=1-5$ and the compressible Navier-Stokes equations in $d=1$. 
\subsection{Experiments for PME}
We consider the PME with $m=2$. An exact, self-similar solution to the PME is the Barenblatt-Kompaneets-Zeldovich solution
\begin{equation} \label{barenblatt}
    u(t,x) = \frac{1}{t^\alpha} \left( b - \frac{m-1}{2m} \beta \frac{\|x\|^2}{t^{2\beta}} \right)^{\frac{1}{m-1}}_+,
\end{equation}
where $\|\cdot\|$ is the $\ell^2$ norm, $(\cdot)_+$ is the positive part and
$\alpha = \frac{d}{d(m-1)+2}$ and $\beta = \frac{1}{d(m-1)+2}$.
This solution is compactly supported but not differentiable at the edges of the support, which causes difficulty for numerical schemes. 
We now investigate the effect of the network width $N$ on the relative error between the network and the solution $u$. This allows us to validate the convergence rate obtained in our theoretical results.
Theorem \ref{thm1} shows that for given $u : (0,T) \times \mathbb{R}^d \to \mathbb{R}$ in some Sobolev space, a RaNN $ \hat{u}_N$ of width $N$ is able to achieve relative error $\mathcal{R}(u, \hat{u}) = \|\hat{u} - u\|_{L^2_{t,x}}/ \|u\|_{L^2_{t,x}}$ bounded by $C N^{-1/2}$, where $C$ depends on $d$, $\Omega$ and $u$.

To validate this convergence rate in practice, we train a RaNN to approximate the PME solution \eqref{barenblatt} in dimensions $d=1,...,5$. In each dimension, we take a set of widths $N \in \{N_1, ..., N_{k} \}$, train the network for each width and plot the relative error of the final network against the true solution. For each dimension $d$ and width $N$, we sample $M = 10N$ points ($M\gg N$ to ensure the problem remains well-posed) with a mixed strategy; $50\%$ of the points are sampled uniformly on $(0,1)^d$ and $50\%$ are sampled uniformly on $[0.2, 0.8]^d$, which is a box focused on the initial support of the solution. Then we find weights $\mathbf{W} =\{ W_i \}_{i=1}^N$ minimising the Ridge regression loss and evaluate the relative error between RaNN approximation and true solution.
The cases $d=4,5$ are included in Figure \ref{fig:RaNN-PME-d45}, while the cases $d=1,2,3$ can be found in Appendix \ref{sec:experiments}, Figure \ref{fig:errors-PME}. The error curves are plotted on log-log scales in Figure \ref{fig:errors-PME-log}, Appendix \ref{sec:experiments}. 
\begin{figure}
    \centering
    \begin{subfigure}{0.4\textwidth}
        \includegraphics[width=\linewidth]{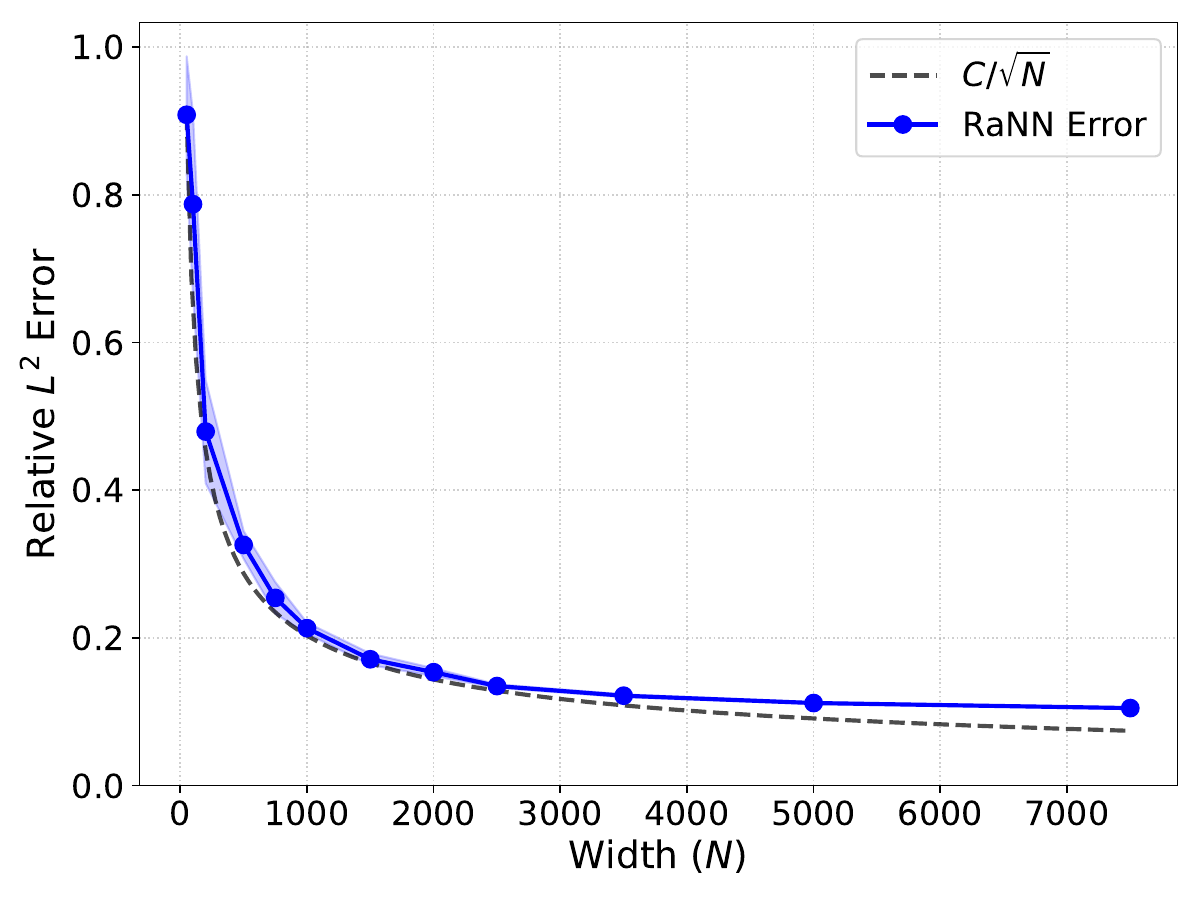}
        \caption{$d = 4$}
    \end{subfigure}
    \begin{subfigure}{0.4\textwidth}
        \includegraphics[width=\linewidth]{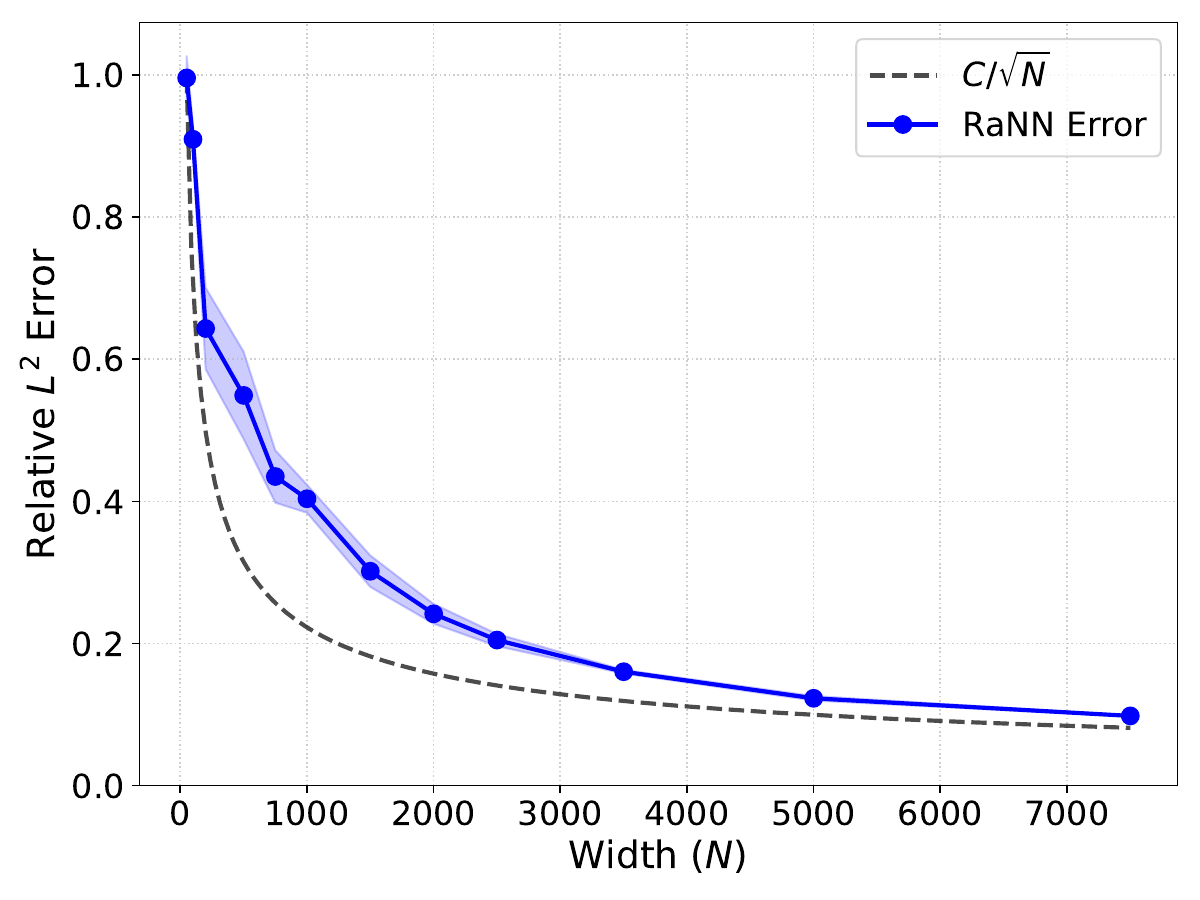}
        \caption{$d = 5$}
    \end{subfigure}
    \caption{Approximation error of RaNNs solving PME in dimensions $d = 4,5$ with varying widths. The shaded band indicates the region within one standard deviation of the mean relative $L^2$ error.}
    \label{fig:RaNN-PME-d45}
\end{figure}
\rev{\begin{remark}[Choice of weight distribution]
    The experiments above use normally-distributed hidden weights, in order to test whether the convergence rate extends beyond the heavy-tailed distribution $\pi$ of Theorem~\ref{thm1} to a distribution that is more common in practice. We have also repeated these experiments with weights drawn from $\pi$ itself, and with a (heavy-tailed) Student-t distribution, and obtained similar rates in both cases; these results are reported in Appendix~\ref{apdx-PME}, together with a discussion of why the proof of Theorem~\ref{thm1} relies on the heavy-tailed form of $\pi$ (Remark~\ref{rmk:sampling}).
\end{remark}}

\rev{ \begin{remark}[Training on the PDE residual]
In Appendix~\ref{apdx:residual} we also train RaNNs only by minimising the PINN loss, without using the solution in the interior, for the PME in $d=1,2$ on the interior of the support, where the hypotheses of Corollary~\ref{cor:PME} hold. We compare with fully trainable PINNs with the same number $N$ of trainable parameters, trained with the same optimiser. The RMS PDE residual of the RaNN decays at least as fast as $N^{-1/2}$, with fitted slopes $-0.76 \pm 0.07$ ($d=1$) and $-1.47 \pm 0.29$ ($d=2$), which is consistent with Corollary~\ref{cor:PME}. The fully trainable network attains smaller residuals, by a factor of $5.7$ ($d=1$) and $1.7$ ($d=2$) at $N=1280$, but its measured rates vary considerably more across random initialisations. 
\end{remark}}

\subsection{Experiments for compressible Navier-Stokes} \label{subsec:NSE}
We now turn to the one-dimensional compressible Navier-Stokes system, given by \eqref{NSE}. 
As baseline, we consider the travelling shock-wave solutions considered by \cite{dalibard2020existence}, where the pressure $p_\epsilon(v) = \epsilon / (v-1)^\gamma$ for $\gamma>0$ is taken, where $v=1/\rho$. The travelling wave solutions can be obtained by taking the ansatz $(v,u)(t,x) = (\mathfrak{v}, \mathfrak{u})(x- st)$, where $s$ is the shock speed. This reduces the PDE to an ODE for $\mathfrak{v}$. The velocity $\mathfrak{u}$ can then be obtained from the relationship $\mathfrak{u} = s(v_+ - \mathfrak{v})$, which follows from the conservation of mass.

For our experiment, we consider the domain $(0,T) \times (-5,5)$ with $T=1.0$, $\mu =10,\epsilon = 10^{-3}, \gamma = 2$. 
We then compute RaNN approximations $(\mathfrak{v}_N, \mathfrak{u}_N)$  for different widths $N$ and measure the relative error to  the baseline solution $(\mathfrak{v}, \mathfrak{u})$. The results can be seen in Figure \ref{fig:NS}, which shows that the errors are close to the $C/\sqrt{N}$ curve, in support of the upper bound of Theorem \ref{thm1}. Further experimental details and a visual depiction of the travelling-wave solution can be found in Appendix \ref{apdx-NSE}.
\begin{figure}
    \centering
    \includegraphics[width=0.45\linewidth]{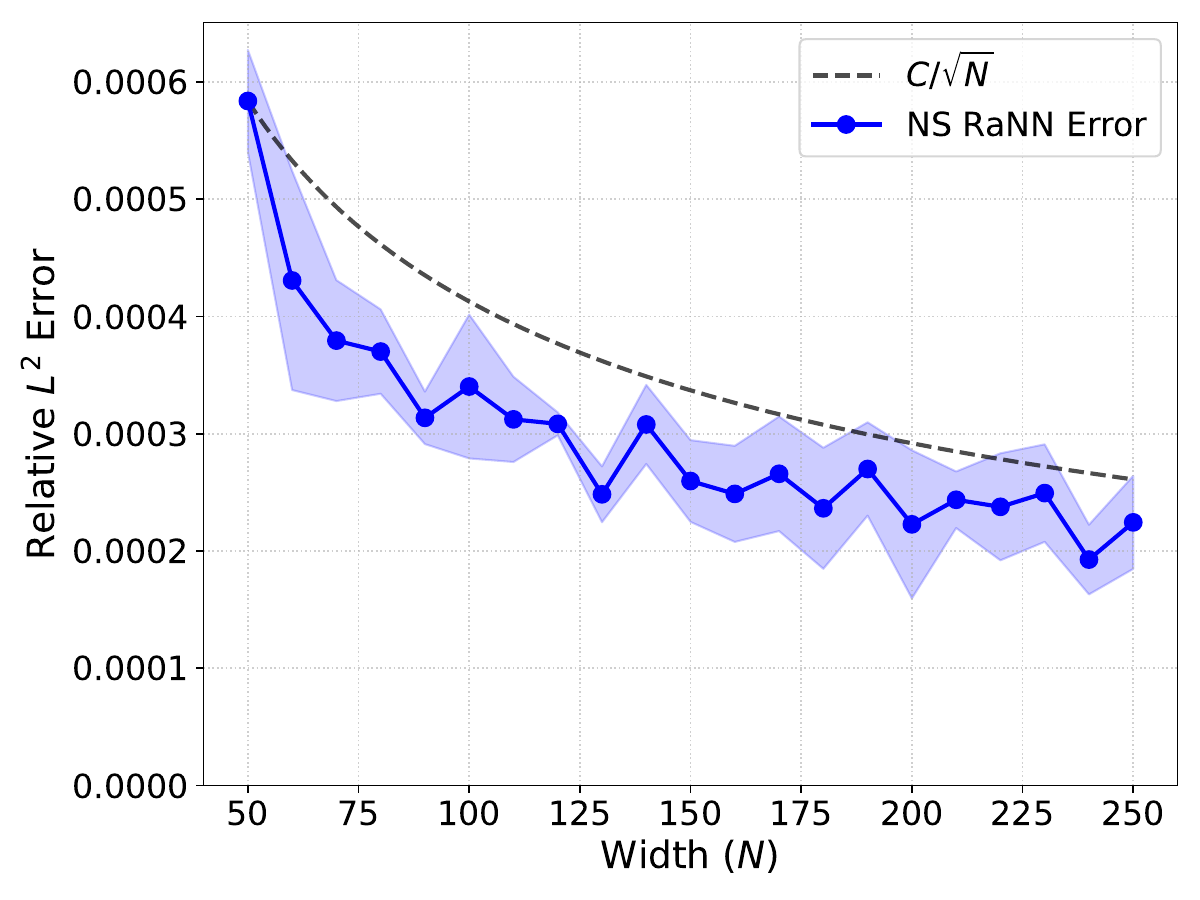}
    \caption{Approximation error of RaNNs solving the compressible NS system with varying widths. The shaded band indicates the region within one standard deviation of the mean relative $L^2$ error.}
    \label{fig:NS}
\end{figure}
To explain the choice of dimensions, we note that high-dimensional studies of randomised neural network approximations have been carried out for linear PDEs (see e.g. \cite{gonon2023random}), but for non-linear equations high quality analytical reference solutions are rare: travelling waves for compressible Navier-Stokes are essentially one-dimensional, and for the Barenblatt profile the support spreads more slowly as $d$ grows, which amplifies sampling errors.


\section{Conclusion}
In this work, we have shown that neural networks with randomly generated hidden weights (RaNNs) \rev{approximate sufficiently regular functions in time-dependent Sobolev spaces at the dimension-free rate $N^{-1/2}$. This requires $u \in H^{p+s_1}_t H^{q+s_2}_x$ with $s_1 > 3/2$ and $s_2 > (d+2)/2$. An important future direction would be to explore whether these assumptions can be relaxed; in Remark~\ref{rmk:scope} we further discuss when these assumptions hold, which regimes they exclude, and provide a preliminary discussion on how they might be relaxed.}  

Future work may additionally investigate whether our results can be extended to solutions which are less smooth than those considered here. Moreover, our results assume a heavy-tailed weight distribution, while numerical experiments indicate that the same rate also holds for Gaussian weights. It would also be interesting to see whether the constant of proportionality (appearing in \eqref{thm-ineq}) can be improved, either using the theory of ridgelet transforms or an alternative representation. 
It will also be important to explore whether RaNNs are prone to the same spectral bias issues that PINNs face, especially for complex PDEs such as compressible Navier-Stokes  in the turbulent regime. 

\begin{ack}
The work of M.A. Mehmood was supported by the Roth Scholarship funded by the Department of Mathematics, Imperial College London.
\end{ack}
\bibliographystyle{plain}
\bibliography{neurips}

\appendix
\section{Proofs}
\subsection{Proof of Proposition \ref{prop:u-rep}}
\label{sec:proofPROP}
\begin{proof} Fix $m \ge 0$ and an arbitrary $\psi \in \mathcal{S}(\mathbb{R})$. Using the definition of the ridgelet transform from \cite{sonoda2017neural}, we have for any $s\ge0$, 
    \begin{equation} 
       \mathcal{R}_\psi u (\tau, \mathbf{a}, b) := \int_{\mathbb{R}^{d+1}} u(t,\mathbf{x})\psi(\tau t +\mathbf{a}\cdot \mathbf{x}-b)\|(\tau, \mathbf{a})\|^s~dtd\mathbf{x}.
    \end{equation} Note that the dual ridgelet transform $R_\eta^\dagger T$ of $T:\mathbb{R}^{d+2} \to \mathbb{R}$ with respect to $\eta : \mathbb{R} \to \mathbb{R}$ is defined as \begin{equation}
        \mathcal{R}_\eta^\dagger T(t,\mathbf{x}) := \int_{\mathbb{R}^{d+2}} T(\tau, \mathbf{a},b) \eta(\tau t +\mathbf{a}\cdot \mathbf{x}-b)\|(\tau, \mathbf{a})\|^{-s} d\tau d\mathbf{a} db
    \end{equation} We will take $s=0$ in the above definitions.
    Recall that $\sigma \in \mathcal{S}_0'(\mathbb{R})$ is a fixed activation. Theorem 5.6 of \cite{sonoda2017neural} says that if we can find a function $\psi$ so that $K_{\psi, \sigma} \in (0, \infty)$
    then the reconstruction formula $u(t,\mathbf{x}) =  \frac{1}{K_{\psi, \sigma}}R_\sigma^{\dagger}R_\psi u(t,\mathbf{x})$ holds. Therefore, to finish the proof we need to find $\psi$ so that $K_{\psi, \sigma} = 1$ and $   |\widehat{\psi}(\omega)| \le C|\omega|^{m}$ for $|\omega|<1$ and some $C>0$ (condition \eqref{psi-growth}). We will make use of Corollary 5.5 of \cite{sonoda2017neural}, which says that if $\zeta^\beta \widehat{\sigma}(\zeta) \in C(-\delta,\delta)$ for some $\delta>0, \beta\in \mathbb{N}$ and $\psi_0 \in \mathcal{S}(\mathbb{R})$ is such that \begin{equation}
        J_\sigma := \int_\mathbb{R} \zeta^\beta \overline{\widehat{\psi_0}(\zeta)} \widehat{\sigma}(\zeta) ~d\zeta \ne 0,
    \end{equation} then 
    \begin{equation*}
        \psi = \Lambda^{d+1} \psi_0^{(\beta)} 
    \end{equation*}is admissible with $\sigma$, where for $m \in \mathbb{N}$, $\Lambda^{m}$ is the backprojection filter that satisfies \begin{equation} \label{back-filter}
        \widehat{\Lambda^m F}(\mathbf{u}, \omega) = i^m|\omega|^m \widehat{F}(\mathbf{u}, \omega).
    \end{equation} 
  We  let $G(z) := \exp(-z^2/2)$ be the standardised gaussian and take 
    \begin{equation}
        \psi_0(z) := \frac{d^{2n}}{dz^{2n}} G(z),
    \end{equation} where $n$ is an arbitrary positive integer. Then since we assume $\sigma \in \mathcal{T}_k$ (see Definition \ref{def:Tk}), we have $\zeta^\beta \widehat{\sigma}(\zeta) \in C(-\delta, \delta)$ for some $\delta>0$ and 
    \begin{equation}
        \begin{aligned}
            J_\sigma = \int_\mathbb{R} \zeta^\beta \widehat{G^{(2n)}}(\zeta)~ \widehat{\sigma}(\zeta)~d\zeta &=  \sqrt{2\pi}(-1)^{n} \int_\mathbb{R} \zeta^{2n+\beta}G(\zeta) ~\widehat{\sigma}(\zeta)~d\zeta \ne 0,
        \end{aligned}
    \end{equation} 
    where we have used $\widehat{G}(\zeta) = \sqrt{2\pi} G(\zeta)$.
    It is important to observe that we can take $n \in \mathbb{Z}^+$ to be as large as we like in this construction, since this is required by the second point in Definition \ref{def:Tk} (of $\mathcal{T}_k$).
    
    Thus, using Corollary 5.5 of \cite{sonoda2017neural} we can say that $\psi = \Lambda^{d+1} \psi_0^{(\beta)}$ and $\sigma$ are jointly admissible. We will normalise $\psi$ and therefore we can assume $K_{\psi, \sigma} = 1$.  Applying the reconstruction formula (Theorem 5.6 of \cite{sonoda2017neural}) gives
    \begin{equation}
        \begin{aligned}
            u(t,\mathbf{x}) &= R_\sigma^{\dagger}R_\psi u(t,\mathbf{x}) \\[1ex] &=  \int_{\mathbb{R}}\int_{\mathbb{R}^d} \int_{\mathbb{R}} R_\psi u(\tau, \mathbf{a}, b) \sigma (\tau t+\mathbf{a}\cdot \mathbf{x}-b) dbd\mathbf{a}d\tau.
        \end{aligned}
    \end{equation} Lastly, we need to verify that $A_{\psi, m} <+\infty$. Using the property of the backfilter from \eqref{back-filter}, we have \begin{equation} \label{psi-construction} \begin{aligned} 
        \widehat{\psi}(\omega) &= i^{d+1}|\omega|^{d+1} \widehat{\psi_0'} = i^{d+2}\omega|\omega|^{d+1} \widehat{\psi_0} \\[1ex] &= \sqrt{2\pi}~i^{d+2n+2}\omega|\omega|^{d+2n+1} G(\omega).
        \end{aligned}
    \end{equation} Taking the absolute value gives $ |\widehat{\psi}(\omega)| = \sqrt{2\pi}|\omega|^{d+2n+2} G(\omega)$. Then choosing $n$ so that $d+2n+2 > m$ (e.g. $n> m-d-2 $), we have $|\widehat{\psi}(\omega)| \le \sqrt{2\pi}|\omega|^m$ for $|\omega| < 1$.
\end{proof}
\begin{remark}[Proving admissibility of $\tanh, \cos $ and sigmoid] \label{rmk:tanh} We show here that each of these activations belong to $\mathcal{T}_k$ for $k\ge0$.
    Firstly, they each belong to $\mathcal{S}_0'(\mathbb{R})$; this is explicitly mentioned in Section 6.1 of \cite{sonoda2017neural}. In the case of $\sigma = \tanh$, we have $\widehat{\tanh}(\zeta) = -i\pi/ \sinh(\pi\zeta/2)$ so $\zeta \widehat{\sigma}(\zeta)$ is continuous around $0$, and
    \begin{equation}
        J_\sigma = \sqrt{2\pi}(-1)^{n+1} i\pi\int_\mathbb{R}   \frac{\zeta^{2n+1}G(\zeta)}{\sinh(\frac{\pi \zeta}{2})}~d\zeta \ne 0.
    \end{equation} 
    Furthermore, $\tanh$  and each of its derivatives are bounded uniformly. Therefore $\tanh \in \mathcal{T}_k$ for any $k\ge0$. 
    For $\sigma = \cos$, $\widehat{\cos}(\zeta) = \pi\left(\delta(\zeta-1) + \delta(\zeta+1)\right)$, which is continuous (in fact, zero) in a neighbourhood of $0$, so we may take $\beta = 0$, and \begin{equation}
    J_\sigma = \sqrt{2\pi}(-1)^{n} \pi \left[ G(1)+ G(-1)\right] \ne 0.
\end{equation} Therefore it also belongs to $\mathcal{T}_k$ for all $k\ge0$.

    For $\sigma = \text{sigmoid}$, the argument is similar; $\widehat{\sigma} = -i \pi \text{csch}(\pi \zeta ) + \pi \delta(\zeta)$, so $\zeta \widehat{\sigma}(\zeta)$ is continuous around the origin and $J_\sigma \ne 0$. Each of its derivatives are also uniformly bounded.
\end{remark}

\subsection{Sketching the proof of Lemma \ref{lemma:parseval}} \label{sec:sketchproofparseval}
For convenience, we start with a sketch of Lemma \ref{lemma:parseval}. The full proof is provided in \ref{sec:proofparseval} below. 
\begin{proof}[Proof (Outline)] 
\textbf{Step 1: Plancherel in $b$. } We apply the Plancherel formula in $b$ to get
 \begin{equation} 
    \begin{aligned}
    I &= \int_{\mathbb{R}^{d+2}} | \widehat{\mathcal{R}_\psi u}(\tau, \mathbf{a}, \omega)|^2 (1+|\tau|^{2})^p(1+\|\mathbf{a}\|^{2})^q~d\tau d\mathbf{a} d\omega \\[1ex] &+ \int_{\mathbb{R}^{d+2}} | \partial_\omega \widehat{\mathcal{R}_\psi u}(\tau, \mathbf{a}, \omega)|^2 (1+|\tau|^{2})^p(1+\|\mathbf{a}\|^{2})^q~d\tau d\mathbf{a} d\omega 
    \end{aligned}
    \end{equation}

Then we show using the definition of the Fourier Transform that $\widehat{R_\psi u}(\omega) = \widehat{u}(\tau \omega, \mathbf{a}\omega) \widehat{\psi}(-\omega)$, so $I$ can be split up as 
\begin{equation}
    \begin{aligned}
        I &\le \int_\mathbb{R} \int_{\mathbb{R}^{d}} \int_\mathbb{R}  |  \widehat{u}(\tau \omega, \mathbf{a}\omega) \widehat{\psi}(-\omega) |^2  (1+|\tau|^{2})^p(1+\|\mathbf{a}\|^{2})^q  ~d\omega  d\mathbf{a} d\tau \\[1ex] &+  \int_\mathbb{R} \int_{\mathbb{R}^{d}} \int_\mathbb{R}  |\partial_\omega \left( \widehat{u}(\tau \omega, \mathbf{a}\omega) \widehat{\psi}(-\omega) \right)|^2  (1+|\tau|^{2})^p(1+\|\mathbf{a}\|^{2})^q  ~d\omega d\mathbf{a} d\tau =: I_1 + I_2
    \end{aligned}
\end{equation}
The second term $I_2$ is the main obstacle in the proof. \\
\textbf{Step 2: Change of variables.}  Performing a change of variables $s =\tau \omega, \xi = \mathbf{a} \omega$ allows us to estimate $I_2$ as
 \begin{equation} \label{I-split}
        \begin{aligned}
            I_2 &\le  2\int_\mathbb{R} \int_{\mathbb{R}^{d}} \int_\mathbb{R}  | \widehat{u}(s, \xi) \partial_\omega \widehat{\psi}(-\omega) |^2  (1+(\frac{s}{\omega})^2)^p(1+(\frac{\|\xi\|}{\omega})^2)^q  |\omega|^{-(d+1)}~d\omega d\xi ds  \\[1ex] &+ 2\int_\mathbb{R} \int_{\mathbb{R}^{d}} \int_\mathbb{R}  | \frac{1}{\omega} (s\partial_s +\xi \nabla_\xi ) \widehat{u}(s, \xi) \widehat{\psi}(-\omega) |^2  (1+(\frac{s}{\omega})^2)^p(1+(\frac{\|\xi\|}{\omega})^2)^q  |\omega|^{-(d+1)}~d\omega d\xi ds \\[1ex] &=: I_{2A}+I_{2B}
        \end{aligned}
    \end{equation}

\textbf{Step 3: Estimating $I_{2A}$ and $I_{2B}$.} 
Here we show that one can find $M$ large enough (but still finite) so that with the corresponding $\psi$ generated from Proposition \ref{prop:u-rep}, the integrals $I_{2A}, I_{2B}$ are both bounded by constants depending on the $L^2$ norms of $\psi$. Most of the difficulty lies in $I_{2B}$ due to the derivatives that appear. Nonetheless, using properties of the Fourier transform and Fubini to exchange the order of integration, $I_{2B}$ can be estimated as
   \begin{equation}
        \begin{aligned}
           I_{2B} &\le  C_{\psi_3} \int_{\mathbb{R}^{d+1}}  |\widehat{(-it) u}|^2 (s, \xi) (1+ |s|^2)^{p+1}(1+\|\xi\|^{2})^{q}~d\xi ds \\[1ex] &+ C_{\psi_3}\int_{\mathbb{R}^{d+1}}  | \widehat{(-i\mathbf{x}) u}|^2 (s, \xi) (1+ |s|^2)^{p}(1+\|\xi\|^{2})^{q+1}~d\xi ds,
        \end{aligned}
    \end{equation}
    where
    \begin{equation}
        \begin{aligned}
            &C_{\psi_3} = 4\int_\mathbb{R} |\widehat{\psi}(-\omega)|^2 \max(1, |\omega|^{-(2p+2q)}) |\omega|^{-(d+3)}~d\omega.
        \end{aligned}
    \end{equation}
    Then by definition of the space $H^p(\mathbb{R}; H^q(\mathbb{R}^d))$ and the compact support of $u$, we get 
    \begin{equation}
        I_{2B} \le  C_{\psi_3}\rev{K_{T,R}} (\|  u\|_{H^{p+1}_t H^{q}_x}^2 +  \|u\|_{H^{p}_tH^{q+1}_x}^2).
    \end{equation}
    Similar arguments bound $I_1$ and $I_{2A}$ by constants $C_{\psi_1}$ and $C_{\psi_2}$ of the same form as $C_{\psi_3}$ (see Appendix~\ref{sec:proofparseval}). Our final job is to show that we can construct $\psi$ so that each of $C_{\psi_1}, C_{\psi_2}, C_{\psi_3}$ can be suitably bounded. This is done for each constant by separating $|\omega| < 1$ and $|\omega| \ge 1$. On $|\omega| < 1$ we take advantage of Proposition \ref{prop:u-rep} which allows us to choose $\psi$ with enough vanishing moments at $0$ to ensure finiteness. On $|\omega| \ge 1$ we use $\psi \in \mathcal{S}(\mathbb{R})$. This is the main idea. The full details of the proof are deferred to Appendix \ref{sec:proofparseval}.
\end{proof}

\subsection{Proof of Lemma \ref{lemma:parseval}} \label{sec:proofparseval}
\begin{proof}
\textbf{Step 1: Plancherel in $b$. } 
Firstly, let's note that using the Plancherel formula in $b$ we have
\begin{equation} \begin{aligned}
     \int_\mathbb{R} |\mathcal{R}_\psi u(\tau, \mathbf{a}, b)|^2 (1+b^2) ~db &= \int_\mathbb{R} |\mathcal{R}_\psi u(\tau, \mathbf{a}, b)|^2 ~db + \int_\mathbb{R} |b \mathcal{R}_\psi u(\tau, \mathbf{a}, b)|^2 ~db \\[1ex]
     &= \int_\mathbb{R} |\widehat{\mathcal{R}_\psi u}(\tau, \mathbf{a}, \omega)|^2 ~d\omega + \int_\mathbb{R} |\partial_\omega \widehat{\mathcal{R}_\psi u}(\tau, \mathbf{a}, \omega)|^2 ~d\omega.
    \end{aligned}
\end{equation}
    Then
    \begin{equation}  \label{plancherel-b}
    \begin{aligned}
    I &= \int_{\mathbb{R}^{d+2}} | \widehat{\mathcal{R}_\psi u}(\tau, \mathbf{a}, \omega)|^2 (1+|\tau|^{2})^p(1+\|\mathbf{a}\|^{2})^q~d\tau d\mathbf{a} d\omega \\[1ex] &+ \int_{\mathbb{R}^{d+2}} | \partial_\omega \widehat{\mathcal{R}_\psi u}(\tau, \mathbf{a}, \omega)|^2 (1+|\tau|^{2})^p(1+\|\mathbf{a}\|^{2})^q~d\tau d\mathbf{a} d\omega 
    \end{aligned}
    \end{equation}
    We now find an expression for $\widehat{\mathcal{R}_\psi u}$. Using the definition of the Fourier transform and the ridgelet transform,
    \begin{equation} \begin{aligned}
        \widehat{R_\psi u}(\omega) &= \int_\mathbb{R} e^{-i\omega b} ~ \mathcal{R}_\psi u ~db \\[1ex] &= \int_\mathbb{R} e^{-i\omega b} \left( \int_{\mathbb{R}^{d}} \int_{\mathbb{R}} u(t,\mathbf{x}) \psi(\tau t + \mathbf{a}\cdot \mathbf{x} - b) ~dtdx\right)db \\[1ex] &= \int_{\mathbb{R}^{d}}\int_{\mathbb{R}} u(t,\mathbf{x}) \left( \int_\mathbb{R} e^{-i\omega b} \psi(\tau t+ \mathbf{a}\cdot \mathbf{x}-b)~db \right)~dtdx,
        \end{aligned}
    \end{equation} where we have also used Fubini to exchange the order of the integrals. This is valid since we assume $u$ is compactly supported, and therefore $u\psi$ is integrable on $\mathbb{R}^{d+1}$. By a change of variables $p=\tau t+\mathbf{a}\cdot \mathbf{x}-b$ the inner integral is equivalent to $e^{-i\omega(\tau t + \mathbf{a} \cdot \mathbf{x})}\int_\mathbb{R} e^{i\omega p} \psi(p)~dp = e^{-i\omega (\tau t + \mathbf{a} \cdot \mathbf{x})}\widehat{\psi}(-\omega)$. Therefore
    \begin{equation}
        \begin{aligned}
             \widehat{R_\psi u}(\omega) &= \widehat{\psi}(-\omega)\int_{\mathbb{R}^d} u(t,\mathbf{x}) e^{-i\omega(\tau t + \mathbf{a}\cdot \mathbf{x})} ~dtd\mathbf{x}  \\[1ex]
             &= \widehat{u}(\tau \omega, \mathbf{a}\omega) \widehat{\psi}(-\omega) .
        \end{aligned}
    \end{equation} Inserting this into \eqref{plancherel-b}, we have 
    \begin{equation}
        \begin{aligned}
            I &= \int_\mathbb{R} \int_{\mathbb{R}^{d}} \int_\mathbb{R}  |  \widehat{u}(\tau \omega, \mathbf{a}\omega) \widehat{\psi}(-\omega) |^2  (1+|\tau|^{2})^p(1+\|\mathbf{a}\|^{2})^q  ~d\omega d\tau d\mathbf{a} \\[1ex] &+  \int_\mathbb{R} \int_{\mathbb{R}^{d}} \int_\mathbb{R}  |\partial_\omega \left( \widehat{u}(\tau \omega, \mathbf{a}\omega) \widehat{\psi}(-\omega) \right)|^2  (1+|\tau|^{2})^p(1+\|\mathbf{a}\|^{2})^q  ~d\omega d\tau d\mathbf{a} \\[1ex]  &=: I_1 + I_2.
        \end{aligned}
    \end{equation}
    \textbf{Step 2: Change of variables.}
    We now introduce the change of variables $  s = \tau \omega, \xi= \mathbf{a}\omega.$ The corresponding jacobian is $d\mathbf{a}d\tau = |\omega|^{-(d+1)}d\xi ds$ and the operator changes to 
    \begin{equation}
    \begin{aligned}
        \partial_\omega |_{(\tau, \mathbf{a})} &= \partial_\omega|_{(s, \xi)}  + \frac{1}{\omega}(s \partial_s+ \xi \nabla_\xi) |_{(s, \xi)} =: \partial_\omega|_{(s, \xi)} + \mathcal{D}.
    \end{aligned}
    \end{equation}
    Therefore, separating the $\omega$ terms using Fubini,
    \begin{equation}
        \begin{aligned}
            I_1 &=  \int_{\mathbb{R}^{d+2}} |\widehat{u}(s, \xi) |^2 |\widehat{\psi}(-\omega) |^2 (1+|\frac{s}{\omega}|^{2})^p(1+( \frac{\|\mathbf{\xi}\|}{\omega})^{2})^q |\omega|^{-(d+1)} ~d\omega ds d\xi   \\[1ex] &\le C_{\psi_1} \int_{\mathbb{R}^{d+1}} |\widehat{u}(s, \xi) |^2  (1+|s|^{2})^p(1+\|\xi\|^{2})^q  ~ds d\xi,
        \end{aligned}
    \end{equation}
    where
    \begin{equation}
       C_{\psi_1} =  \int_\mathbb{R} |\widehat{\psi}(-\omega)|^2 \max(1, |\omega|^{-2p}) \max(1, |\omega|^{-2q})|\omega|^{-(d+1)} ~d\omega. 
    \end{equation}
    We will show the boundedness of this constant at the end of the proof.  
    For $I_2$,
    \begin{equation} 
        \begin{aligned}
            I_2 &\le  2\int_\mathbb{R} \int_{\mathbb{R}^{d}} \int_\mathbb{R}  | \widehat{u}(s, \xi) \partial_\omega \widehat{\psi}(-\omega) |^2  (1+(\frac{s}{\omega})^2)^p(1+(\frac{\|\xi\|}{\omega})^2)^q  |\omega|^{-(d+1)}~d\omega ds d\xi \\[1ex] &+ 2\int_\mathbb{R} \int_{\mathbb{R}^{d}} \int_\mathbb{R}  | \frac{1}{\omega} (s\partial_s +\xi \nabla_\xi ) \widehat{u}(s, \xi) \widehat{\psi}(-\omega) |^2  (1+(\frac{s}{\omega})^2)^p(1+(\frac{\|\xi\|}{\omega})^2)^q  |\omega|^{-(d+1)}~d\omega ds d\xi \\[1ex] &=:I_{2A} + I_{2B}
        \end{aligned}
    \end{equation} 
    \textbf{Step 3: Estimating $I_{2A}$ and $I_{2B}$.} 
    First note that for any $\mathbf{x} \in \mathbb{R}^d$ and $k >0$,
    \begin{equation} \begin{aligned}
    & (1+ \frac{\|\mathbf{x}\|^{2}}{|\omega|^2})^k \le  \max(1, |\omega|^{-2k})(1+\|\mathbf{x}\|^2)^k.
    \end{aligned}
    \end{equation}
    Therefore, for $I_{2A}$ we can separate the $\omega$ dependence and write 
    \begin{equation}
        \begin{aligned}
            I_{2A} \le C_{\psi_2} \int_{\mathbb{R}^{d+1}} |\widehat{u}(s,\xi)|^2(1+s^2)^p(1+\|\xi\|^2)^q dsd\xi,
        \end{aligned}
    \end{equation}
    where \begin{equation}
          C_{\psi_2} = \rev{2}\int_\mathbb{R} | \partial_\omega \widehat{\psi}(-\omega)|^2 \max(1, |\omega|^{-(2p+2q)}) |\omega|^{-(d+1)}~d\omega.
    \end{equation}
    We will consider the boundedness of this constant at the end. For $I_{2B}$, we can estimate the $(s,\xi)$ part as
  \begin{equation} \begin{aligned}
         \int_{\mathbb{R}^{d+1}} & | \frac{1}{\omega} (s\partial_s +\xi \nabla_\xi ) \widehat{u}(s,\xi)|^2~dsd\xi  \le  2 |\omega|^{-2} \left( \int_{\mathbb{R}^{d+1}}| s\partial_s \widehat{u}(s,\xi)|^2+ \int_{\mathbb{R}^{d+1}} | \xi\nabla_\xi \widehat{u}(s,\xi)|^2 \right) \\[1ex] &= 2 |\omega|^{-2} \left( \int_{\mathbb{R}^{d+1}}| s \widehat{(-it)u}(s,\xi)|^2+ \int_{\mathbb{R}^{d+1}} | \xi \widehat{(-i\mathbf{x})u}(s,\xi)|^2 \right) \\[1ex] &\le 2|\omega|^{-2}\left( \int_{\mathbb{R}^{d+1}}| \widehat{(-it)u}(s,\xi)|^2(1+s^2) + \int_{\mathbb{R}^{d+1}} | \widehat{(-i\mathbf{x})u}(s,\xi)|^2(1+\|\xi\|^2)    \right),
        \end{aligned}
    \end{equation} where we used $s^{2} \le 1+s^2$ to obtain the final line. 
   Thus we have
    \begin{equation}
        \begin{aligned}
           I_{2B} &\le  C_{\psi_3} \int_{\mathbb{R}^{d+1}}  |\widehat{(-it) u}|^2 (s, \xi) (1+ |s|^2)^{p+1}(1+\|\xi\|^{2})^{q}~d\xi ds \\[1ex] &+ C_{\psi_3}\int_{\mathbb{R}^{d+1}}  | \widehat{(-i\mathbf{x}) u}|^2 (s, \xi) (1+ |s|^2)^{p}(1+\|\xi\|^{2})^{q+1}~d\xi ds,
        \end{aligned}
    \end{equation}
    where
    \begin{equation}
        \begin{aligned}
            C_{\psi_3} = \rev{4}\int_\mathbb{R} |\widehat{\psi}(-\omega)|^2 \max(1, |\omega|^{-(2p+2q)}) |\omega|^{-(d+3)}~d\omega.
        \end{aligned}
    \end{equation} We will consider this constant at the end of the proof.
    Next, \rev{we normalise the norm of $H^p(\mathbb{R};H^q(\mathbb{R}^d))$ on the Fourier side as}
    \begin{equation}
        \int_{\mathbb{R}^{d+1}}  |\widehat{u}(s, \xi)|^2  (1+ |s|^2)^p(1+\|\xi\|^{2})^q ~d\xi ds =: \|u\|_{H^p(\mathbb{R} ; H^q(\mathbb{R}^d))}^2\rev{,}
    \end{equation}
    \rev{and note that the factor $(2\pi)^{-1}$ from Plancherel in $b$ compensates the factor $2\pi$ in $\|\widehat{\psi}\|_{L^2}^2 = 2\pi\|\psi\|_{L^2}^2$ and $\|\omega\widehat{\psi}\|_{L^2}^2 = 2\pi\|\psi'\|_{L^2}^2$.}

    This gives us (using the compact support of $u$)
    \begin{equation}
        \begin{aligned}
            I &\le (C_{\psi_1}+C_{\psi_2}) \|u\|_{H^p_t H^q_x}^2 + C_{\psi_3} (\| |t| u\|_{H^{p+1}_t H^{q}_x}^2 +  \||\mathbf{x}|u\|_{H^{p}_tH^{q+1}_x}^2)
             \\[1ex] &\le (C_{\psi_1}+C_{\psi_2}) \|u\|_{H^p_t H^q_x}^2+C_{\psi_3}\rev{K_{T,R}}  \|u\|_{H^{p+1}_tH^{q+1}_x}^2 \\[1ex]
             &\le  (C_{\psi_1}+C_{\psi_2}+C_{\psi_3}\rev{K_{T,R}})  \|u\|_{H^{p+1}_tH^{q+1}_x}^2.
        \end{aligned}
    \end{equation}
        \rev{Here $K_{T,R} :=(2T+\lceil p\rceil+1)^2+(2R\sqrt{d}+\lceil q\rceil+1)^2$, since $|t| \le 2T$ and $\|\mathbf{x}\| \le 2R\sqrt{d}$ on the support of $u$, and each derivative falling on $t$ or $\mathbf{x}$ produces a lower-order term. For non-integer $p$ or $q$, the bound follows by interpolation between the neighbouring integer orders.} It remains to bound the constants $C_{\psi_1}, C_{\psi_2}$ and $C_{\psi_3}$. Starting with the most recent $C_{\psi_3}$, we consider $|\omega|<1$ and $|\omega| \ge 1$ separately. On $|\omega| < 1$, we have
    \begin{equation*}
        C_{\psi_3} =  \rev{4}\int_{-1}^1 |\widehat{\psi}(-\omega)|^2 \max(1, |\omega|^{-(2p+2q)}) |\omega|^{-(d+3)}~d\omega \le \rev{4}\int_{-1}^1 |\widehat{\psi}(-\omega)|^2 |\omega|^{-(2p+2q+d+3)}~d\omega,
    \end{equation*}
    so letting $M = (2p+2q+d+3)/2$ and choosing $\psi$ as per the construction of Proposition \ref{prop:u-rep} (see \eqref{psi-construction}), we have
    \begin{equation}
        |\widehat{\psi}(\omega)| \le \sqrt{2\pi}|\omega|^M
 \text{ for } |\omega| < 1.    \end{equation}
    This implies that the integral is bounded by $\int_{-1}^{1}2\pi d\omega=4\pi$. On $|\omega| \ge 1$ we can bound it by $\|\psi\|_{L^2(\mathbb{R})}^2$ (by Plancherel), and so $C_{\psi_3} \le \rev{16\pi + 4\|\psi\|_{L^2(\mathbb{R})}^2}$.  We can apply the same estimate to $C_{\psi_1}$ to get $C_{\psi_1} \le 4\pi + \|\psi\|_{L^2(\mathbb{R})}^2$. For $C_{\psi_2}$, recall from \eqref{psi-construction} that $ \widehat{\psi}(\omega) = \sqrt{2\pi}~i^{m}\omega|\omega|^{m-1} G(\omega)$, for $m$ which we can choose arbitrarily large. A simple computation gives us on $|\omega| < 1$ that $|\partial_\omega \widehat{\psi}| \le \sqrt{2\pi}(m+1)\omega^{m-1}G(\omega)$. Similarly, on $|\omega| \ge 1$ we get $|\partial_\omega \widehat{\psi}| \le \sqrt{2\pi}\rev{(m+1)|\omega|^{m+1}}G(\omega)$. Therefore on $|\omega| < 1$, by choosing the same $m=M$, 
    \begin{equation}
        \begin{aligned}
          \int_{-1}^1 | \partial_\omega \widehat{\psi}(-\omega)|^2 \max(1, |\omega|^{-(2p+2q)}) |\omega|^{-(d+1)}~d\omega  &\le 2\pi(M+1)^2 \int_{-1}^1 |\omega|^2~d\omega  \\[1ex] &\le 4\pi (M+1)^2.
        \end{aligned}
    \end{equation} For $|\omega| > 1$, notice that $|\partial_\omega \widehat{\psi}| \le \rev{(M+1)}|\omega| |\widehat{\psi}(\omega)|$ and so
    \begin{equation}
        \begin{aligned}
            \int_{\mathbb{R} \backslash B_1(0)} &| \partial_\omega \widehat{\psi}(-\omega)|^2 \max(1, |\omega|^{-(2p+2q)}) |\omega|^{-(d+1)}~d\omega  \le  \rev{(M+1)^2}\int_{\mathbb{R} \backslash B_1(0)} |\omega\widehat{\psi}(\omega)|^2~d\omega.
            \end{aligned}
    \end{equation} and so $C_{\psi_2} \le \rev{2(M+1)^2\left(4\pi + \| \psi'\|_{L^2(\mathbb{R})}^2\right)}$. This leads to \begin{equation}
        I \le  \left[ (4\pi + \|\psi\|_{L^2(\mathbb{R})}^2)\rev{(1+4K_{T,R})}+2\rev{(M+1)^2\left(4\pi + \|\psi '\|_{L^2(\mathbb{R})}^2\right)} \right]\|u\|_{H^{p+1}_tH^{q+1}_x}^2,
    \end{equation}
    where $M = (2p+2q+d+3)/2$.  \rev{Here $\psi$ is the unnormalised function of \eqref{psi-construction}; for the normalised $\psi/K_{\psi,\sigma}$ used in Theorem~\ref{thm1}, $\mathcal{L}_\psi$ is divided by $|K_{\psi,\sigma}|^2$.}

\end{proof}

\subsection{Proof of Theorem \ref{thm1}} \label{sec:proofTHM}

\begin{proof}\textbf{Step 1: extension of $u$ to $\mathbb{R} \times \mathbb{R}^d$.} 
In order to make use of the previous results, we extend $u$ to the full space $\mathbb{R} \times \mathbb{R}^d$. Since $D$ is Lipschitz, there exists a bounded extension operator $E_x : H^{q+s_2}(D) \to H^{q+s_2}(\mathbb{R}^d)$ (see Theorem 5 in Chapter VI of \cite{stein1970singular}). Applying $E_x$ pointwise in time and combining with a standard one-dimensional extension in time, we obtain $\bar{u} \in H^{p+s_1}(\mathbb{R}; H^{q+s_2}(\mathbb{R}^d))$ such that $\bar{u} = u$ on $(0,T) \times D$ and $\|\bar{u}\|_{H^{p+s_1}(\mathbb{R}; H^{q+s_2}(\mathbb{R}^d))} \le C_1 \|u\|_{H^{p+s_1}(0,T; H^{q+s_2}(D))}$. Let $\eta \in C_c^\infty(\mathbb{R})$ and $\chi \in C_c^\infty(\mathbb{R}^d)$ be cut-off functions equal to $1$ on $[0,T]$ and $D$ respectively. Defining $\tilde{u}(t,x) := \eta(t)\chi(x)\bar{u}(t,x)$
we obtain $\tilde u \in H^{p+s_1}_t H^{q+s_2}_x(\mathbb{R}\times\mathbb{R}^d)$ with compact support and $
\|\tilde u\|_{H^{p+s_1}_t H^{q+s_2}_x}^2 \le C_\Omega \|u\|_{H^{p+s_1}_t H^{q+s_2}_x}^2$.
\rev{Note that the constant $C_\Omega$ collects two contributions. The first is the extension constant $C_1$, which depends on $p$, $q$, $d$ and the Lipschitz character of $\partial D$; it is not explicit in \cite{stein1970singular} and grows with the Lipschitz constant of $\partial D$ and with $p, q$. The second is the cost of multiplication by $\eta\chi$, which by the Leibniz rule is controlled by $\|\eta\|_{C^p}\|\chi\|_{C^q}$ times a combinatorial factor depending on $q$ and $d$. A narrower collar between $D$ and $\operatorname{supp}\chi$ increases $\|\chi\|_{C^q}$, while a wider one increases $R$ and hence $\mathcal{L}_\psi$ in \eqref{Lpsi}, so the cut-off cannot be chosen to improve both. In all cases $C_\Omega$ is independent of $N$, so it affects the constant in \eqref{thm-ineq} but not the rate; cases in which its dependence on $d$ is explicit are discussed in Remark~\ref{rmk-growth}.}
\rev{Proposition~\ref{prop:u-rep} applies to $\tilde u$: it is compactly supported, and $\widehat{\tilde u} \in L^1$ by Cauchy--Schwarz, since $\tilde u \in H^{s_1}_tH^{s_2}_x$ with $s_1 > 1/2$, $s_2 > d/2$. Moreover, by Cauchy--Schwarz and Lemma~\ref{lemma:parseval} the integral below converges absolutely, so the limit in the dual ridgelet transform of \cite{sonoda2017neural} can be dropped.}
Then from \eqref{u-representation}, $\tilde{u}$ can be represented as
    \begin{equation} \begin{aligned}\label{thm1:u}
           \tilde{u}(t,x) &= \int_\mathbb{R} \int_{\mathbb{R}^{d}} \int_\mathbb{R} (\mathcal{R}_\psi \tilde{u}) (\tau, \mathbf{a}, b) \sigma (\tau t+ \mathbf{a} \cdot \mathbf{x}-b)~d\tau d\mathbf{a}db \\[1ex]
           &= \int_\mathbb{R} \int_{\mathbb{R}^{d}} \int_\mathbb{R} \frac{(\mathcal{R}_\psi \tilde{u}) (\tau, \mathbf{a}, b)}{\pi(\tau, \mathbf{a}, b)} \sigma (\tau t+ \mathbf{a}\cdot \mathbf{x}-b)\pi(\tau, \mathbf{a}, b)d\tau d\mathbf{a}db,
           \end{aligned}
    \end{equation}
where $\pi : \mathbb{R}^{d+2} \to \mathbb{R}_+$ is the probability density function from \eqref{pi}. \\

\textbf{Step 2: construction of the unbiased estimator $u_N$.} Our neural network approximation of $u$ will be denoted $u_N$, and we define it as an unbiased estimator of $\tilde{u}$, i.e.
\begin{equation}
    u_N(t,\mathbf{x}) := \frac{1}{N} \sum_{i=1}^N  \frac{R_\psi \tilde{u}(\tau_i, \mathbf{a}_i, b_i)}{\pi(\tau_i, \mathbf{a}_i, b_i)} \sigma (\tau_i t + \mathbf{a}_i \cdot \mathbf{x} -b_i) \equiv \frac{1}{N} \sum_{i=1}^N X_i(t,\mathbf{x}),
\end{equation}
where $(\tau_i, \mathbf{a}_i, b_i) \sim \pi$. By construction, it is clear that $\mathbb{E}_\Theta (u_N) = \tilde{u}$. More generally, for any $0 \le \ell \le p$ and multi-index $\beta$ with $|\beta| \le q$, we have 
\begin{equation}
    \mathbb{E}_\Theta(\partial_t^\ell D_\mathbf{x}^\beta (u_N)) = \partial_t^\ell D_\mathbf{x}^\beta (\tilde{u}).
\end{equation} In order to estimate $u - u_N$, we have
\begin{equation} 
    \begin{aligned}
        \mathbb{E}_\Theta \left( \| \partial_t^\ell D_{\mathbf{x}}^\beta (\tilde{u}-u_N) \|_{L^{2}((0,T) \times D)}^2 \right)  &= \mathbb{E}_\Theta \int_{(0,T) \times D} |\partial_t^\ell D_{\mathbf{x}}^\beta (\tilde{u}-u_N)|^2~dxdt \\[1ex]
        &=  \int_{(0,T) \times D} \text{Var}_\Theta \left( \partial_t^\ell D_{\mathbf{x}}^\beta (\tilde{u}-u_N) \right)~dxdt.
    \end{aligned}
\end{equation}
Let $Y_i(t,\mathbf{x)} := X_i(t,\mathbf{x}) - \mathbb{E}_\Theta (X_1(t,\mathbf{x}))$. Then $\E{(Y_i)} = 0$ and $\Var{(Y_i)} = \Var{(X_i)}$. Moreover, we have $u - u_N = \frac{1}{N}\sum_{i=1}^N Y_i(t,\mathbf{x})$ and therefore
$\Var{(u -u_N)}
=\frac{1}{N} \Var{(Y_1(t,\mathbf{x}))} = \frac{1}{N} \Var{(X_1(t,\mathbf{x}))}$. The same argument works if we replace $u - u_N$ with $\partial_t^\ell D_{\mathbf{x}}^\beta (u-u_N)$. As a result, we have (since $u=\tilde{u}$ on $(0,T)\times D$))
\begin{equation} \label{MC-intermediate}
    \begin{aligned}
        \mathbb{E}_\Theta \left( \| \partial_t^\ell D_{\mathbf{x}}^\beta (u-u_N) \|_{L^{2}((0,T) \times D)}^2 \right) &= \mathbb{E}_\Theta \left( \| \partial_t^\ell D_{\mathbf{x}}^\beta (\tilde{u}-u_N) \|_{L^{2}((0,T) \times D)}^2 \right)\\[1ex] & \le \frac{1}{N} \int_{(0,T) \times D}\mathbb{E}_\Theta |\partial_t^\ell D_{\mathbf{x}}^\beta X_i|^2~dxdt.
    \end{aligned}
\end{equation}
\textbf{Step 3: bounding $\mathbb{E}_\Theta (|\partial_t^{\ell} D_{\mathbf{x}}^\beta X_i|^2)$.}
Let $\ell=0, \beta = 0$ to begin with. A direct computation gives 
\begin{equation} \begin{aligned}
     \mathbb{E}_\Theta &(|X_i(t,\mathbf{x})|^2) \\[1ex] &=  C_{\pi} \int_{\mathbb{R}^{d+2}} |R_\psi \tilde{u}(\tau, \mathbf{a}, b)|^2 |\sigma(\tau t+\mathbf{a}\mathbf{x}-b)|^2 (1+ \tau^2)^{\lambda_\tau}(1+\|\mathbf{a}\|^{2})^{\lambda_a}(1 + b^2)~d\tau d\mathbf{a}db \\[1ex] &\le  C_\pi\|\sigma\|_{\infty}^2 \mathcal{L}_\psi\|\tilde{u}\|_{H^{s_1}(\mathbb{R}; H^{s_2}(\mathbb{R}^{d}))}^2,
     \end{aligned}
\end{equation}
where, since $\lambda_\tau>1/2$ and $\lambda_a >d/2$, we have $s_1 > 3/2, s_2 > (d+2)/2$ using Lemma \ref{lemma:parseval}. Therefore, we have \begin{equation} \begin{aligned}
     \mathbb{E}_\Theta \left( \| u-u_N \|_{L^{2}((0,T) \times D)}^2 \right) &\le \frac{T|D| \|\sigma\|_{\infty}^2 \mathcal{L}_\psi}{N}\|\tilde{u}\|_{H^{s_1}(\mathbb{R}; H^{s_2}(\mathbb{R}^{d}))}^2 .
     \end{aligned}
\end{equation}
The same argument works if we take a general $0 \le \ell \le p$ and a non-zero multi-index $\beta = (\beta_1, ..., \beta_d)$ with $|\beta| \le q$. In this case we have
\begin{equation}
    \begin{aligned}
         \mathbb{E}_\Theta &(|\partial_t^\ell D_\mathbf{x}^\beta X_i(t,\mathbf{x})|^2) \\[1ex]&\le  C_{\pi} \int_{\mathbb{R}^{d+2}} |R_\psi \tilde{u}(\tau, \mathbf{a}, b)|^2 |\sigma^{(p+q)}|^2 |\tau|^{2p} \|\mathbf{a}\|^{2q}(1+ \tau^2)^{\lambda_\tau}(1+\|\mathbf{a}\|^{2})^{\lambda_a}(1 + b^2)~d\tau d\mathbf{a}db \\[1ex]
        &\le    C_{\pi} \int_{\mathbb{R}^{d+2}} |R_\psi \tilde{u}(\tau, \mathbf{a}, b)|^2 |\sigma^{(p+q)}|^2 (1+ \tau^2)^{p+\lambda_\tau}(1+\|\mathbf{a}\|^{2})^{q+\lambda_a}(1 + b^2)~d\tau d\mathbf{a}db.
 \end{aligned}
\end{equation}
Then using \eqref{parseval}, we get
\begin{equation}
    \begin{aligned}
         \mathbb{E}_\Theta &(|\partial_t^\ell D_\mathbf{x}^\beta X_i(t,\mathbf{x})|^2)\le C_\pi \|\sigma^{(p+q)}\|_\infty^2 \mathcal{L}_\psi \|\tilde{u}\|_{H^{p+s_1}(\mathbb{R}; H^{q+s_2}(\mathbb{R}^d))}^2.
    \end{aligned}
\end{equation}
 By \eqref{MC-intermediate} this implies (using $\|\tilde{u}\|_{H^{p+s_1}_tH^{q+s_2}_x}^2 \le C_\Omega \|u\|_{H^{p+s_1}_tH^{q+s_2}_x}^2$)
\begin{equation} \label{HpHq-MC}
    \begin{aligned}
        \mathbb{E}_\Theta \left( \| \partial_t^\ell D_{\mathbf{x}}^\beta (u-u_N) \|_{L^{2}((0,T) \times D)}^2 \right) &\le \frac{C_\Omega T|D|C_\pi \|\sigma^{(p+q)}\|_{\infty}^2 \mathcal{L}_\psi}{N}   \|u\|_{H^{p+s_1}(0,T; H^{q+s_2}(D))}
^2,
    \end{aligned}
\end{equation}
where again $s_1 > 3/2$ and $s_2 > (d+2)/2$.
\rev{Summing up this estimate over the \rev{$(p+1)\binom{q+d}{d}$ derivatives $\partial_t^\ell D_{\mathbf{x}}^\beta$ with $\ell \le p$, $|\beta| \le q$} leads to the claimed result.}
\end{proof}

\begin{remark}[Growth factor for the leading coefficient] \label{rmk-growth}
We can give a more precise description of the constant in the right-hand side of \eqref{thm-ineq} for the case where $\sigma \in \mathcal{T}_k$ (e.g. $\tanh$; see Definition \ref{def:Tk})\rev{, $|D| \le 1$, and $u$ is already defined on $\mathbb{R} \times \mathbb{R}^d$ with support in $[-2T,2T] \times [-2R,2R]^d$. In this case no extension is needed, so $C_\Omega=1$.} In such a case, we can show that if we let $\lambda_a = \frac{d}{2} + \alpha$ \rev{for any fixed $\alpha > 0$}, then the coefficient in the right-hand side of \eqref{thm-ineq} will grow at most polynomially in $d$. We will denote by $C$ a positive constant independent of dimension and by $\Gamma(\cdot)$ the Gamma function.  Then note that \begin{equation}  \begin{aligned}
    C_\pi \le C \int_{\mathbb{R}^d} (1 + \|\mathbf{a}\|^2)^{-\lambda_a} ~d\mathbf{a} .
\end{aligned}
\end{equation}
Going to spherical coordinates and letting $\lambda_a = \frac{d}{2}+\alpha$, we find
\begin{equation}
    \begin{aligned}
        \rev{\int_{\mathbb{R}^d} (1 + \|\mathbf{a}\|^2)^{-\lambda_a} ~d\mathbf{a}} = \frac{2 \pi^{d/2}}{\Gamma(d/2)} \int_0^\infty r^{d-1} (1+r^2)^{-\lambda_a}~dr = \pi^{d/2} \frac{\Gamma(\alpha)}{\Gamma(d/2+\alpha)}.
    \end{aligned}
\end{equation}
Now, to bound \rev{$\mathcal{L}_\psi$, we go back to the proof of Lemma~\ref{lemma:parseval}, which shows that \eqref{parseval} holds with $\mathcal{L}_\psi$ replaced by $\mathcal{L}_\psi' := C_{\psi_1} + C_{\psi_2} + K_{T,R}\, C_{\psi_3}$. The expression \eqref{Lpsi} is obtained from $\mathcal{L}_\psi'$ by bounding $|\omega|^{-(d+1)} \le 1$ on $|\omega| \ge 1$, which is too crude to track $d$, so we bound $\mathcal{L}_\psi'$ directly. In the proof of Theorem~\ref{thm1}, Lemma~\ref{lemma:parseval} is applied with the exponents $p+\lambda_\tau$ and $q+\lambda_a$ in place of $p$ and $q$, so the weight in $C_{\psi_1}, C_{\psi_2}, C_{\psi_3}$ is $\max(1, |\omega|^{-2Q})$ with
\begin{equation*}
    Q := p+q+\lambda_\tau+\lambda_a = \frac{d}{2} + c_0, \qquad c_0 := p+q+\lambda_\tau+\alpha.
\end{equation*}}%
We will use the construction of $\psi$ from the proof of Proposition~\ref{prop:u-rep}. \rev{Before normalisation, this gives $\bar{\psi}$ with $|\widehat{\bar{\psi}}(\omega)| = \sqrt{2\pi}\,|\omega|^{\nu}G(\omega)$ and $\nu := d+1+2n+\beta$ (with $\beta = 1$ for $\tanh$ and $\beta = 0$ for $\cos$). Since $|\widehat{\bar{\psi}}(\zeta)||\zeta|^{-(d+1)} = \sqrt{2\pi}\,|\zeta|^{2n+\beta}G(\zeta)$, we have $K_{\bar{\psi},\sigma} = (2\pi)^{d}c_\sigma$ with $c_\sigma \neq 0$ independent of $d$. The normalised function $\psi = \bar{\psi}/K_{\bar{\psi},\sigma}$ therefore satisfies
\begin{equation} \label{rmk-psi-bounds}
    \begin{gathered}
    |\widehat{\psi}(\omega)| = \kappa |\omega|^{\nu} G(\omega), \qquad \kappa := \frac{\sqrt{2\pi}}{|K_{\bar{\psi},\sigma}|} \le C(2\pi)^{-d}, \\
    |\partial_\omega \widehat{\psi}(\omega)| \le \kappa \left(\nu|\omega|^{\nu-1} + |\omega|^{\nu+1}\right) G(\omega).
    \end{gathered}
\end{equation}
On $|\omega| < 1$ the weight is $|\omega|^{-2Q}$, and by \eqref{rmk-psi-bounds} the integrands of $C_{\psi_1}, C_{\psi_2}, C_{\psi_3}$ are at most $C\kappa^2(\nu+1)^2|\omega|^{e}$, where
\begin{equation*}
    e = 2\nu - 2Q - (d+3) = 4n + 2\beta - 2c_0 - 1.
\end{equation*}
The dimension cancels, so we can fix $n$ independently of $d$ such that $e \ge 0$, and the contribution of $|\omega|<1$ is at most $C\kappa^2(\nu+1)^2$. On $|\omega| \ge 1$ the weight is $1$, and we keep the factor $|\omega|^{-(d+1)}$:}
\begin{equation}
    \begin{aligned}
        \rev{\int_{|\omega| \ge 1} |\widehat{\psi}(\omega)|^2 |\omega|^{-(d+1)}~d\omega} &\rev{\le \kappa^2 \int_{|\omega|>1} |\omega|^{2\nu-d-1}e^{-\omega^2}~d\omega} \\[1ex] &= \rev{\kappa^2} \int_{1}^\infty u^{\rev{\nu-d/2-1}} e^{-u} ~du \\[1ex] &\le \rev{\kappa^2\, \Gamma(\nu - d/2)},
    \end{aligned}
\end{equation}
where we have used the substitution $u = \omega^2$. For \rev{$\partial_\omega\widehat{\psi}$}, we can repeat the same process \rev{with \eqref{rmk-psi-bounds}} to find
\begin{equation}
    \rev{\int_{|\omega| \ge 1} |\partial_\omega\widehat{\psi}(\omega)|^2 |\omega|^{-(d+1)}~d\omega \le 2\kappa^2(\nu^2+1)\,\Gamma(\nu - d/2 + 1),}
\end{equation}
\rev{and the weight $|\omega|^{-(d+3)}$ in $C_{\psi_3}$ is smaller than $|\omega|^{-(d+1)}$. Since $\nu - d/2 + 1 = \frac{d}{2} + a$ with $a := 2+2n+\beta$, and $K_{T,R} \le C(d+1)^2$ (here $\lceil q\rceil$ in $K_{T,R}$ becomes $\lceil q+\lambda_a\rceil \le q+\lambda_a+1$), we obtain
\begin{equation*}
    \mathcal{L}_\psi' \le C(d+1)^2(\nu^2+1)\,\kappa^2\, \Gamma\left(\tfrac{d}{2}+a\right) \le C(d+1)^4 (2\pi)^{-2d}
\, \Gamma\left(\tfrac{d}{2}+a\right).
\end{equation*}}%
Now we consider the product \rev{$C_\pi \mathcal{L}_\psi'$}. Notice that the leading term is given by
\begin{equation}
    \begin{aligned}
     F(d) := \rev{\frac{\pi^{d/2}}{(2\pi)^{2d}} \, \frac{\Gamma(\frac{d}{2}+a)}{\Gamma(\frac{d}{2}+\alpha)}}.
    \end{aligned}
\end{equation}
We now prove that $F(d)$ grows at most polynomially in $d$. \rev{The first factor equals $\left(\sqrt{\pi}/(4\pi^2)\right)^d < 1$. For the second, we use the inequality
\begin{equation} \label{ineq-gamma}
    \frac{\Gamma(x+a)}{\Gamma(x+b)} \le (x+a)^{a-b}, \qquad a \ge b, \quad x+b > 0,
\end{equation}
which follows from $\Gamma(y+1) = y\Gamma(y)$ applied to the integer part of $a-b$ and the log-convexity of $\Gamma$ for the fractional part. Since $e \ge 0$ implies $a \ge \alpha$, \eqref{ineq-gamma} with $x = d/2$ and $b = \alpha$ gives $F(d) \le (\frac{d}{2}+a)^{a-\alpha}$, and hence $C_\pi \mathcal{L}_\psi' \le C(d+1)^{4+a-\alpha}$.}
Therefore, \rev{for any fixed $\alpha > 0$,} the bound \eqref{thm-ineq} holds with a constant that grows at most polynomially in $d$\rev{, since $\|\sigma^{(p+q)}\|_\infty$ does not depend on $d$ and $T|D|(p+1)\binom{q+d}{d} \le T(p+1)(d+1)^q$}.
\end{remark}
\rev{\begin{remark}[Scope of the regularity assumptions] \label{rmk:scope}
    Theorem~\ref{thm1} requires $u \in H^{p+s_1}_t H^{q+s_2}_x$ with $s_1 > 3/2$ and $s_2 > (d+2)/2$, so the class of admissible targets tightens as the spatial dimension grows. For the equations treated in Section~\ref{sec:PDE} this regularity is supplied by the equations themselves rather than assumed: the porous medium equation regularises instantaneously away from the free boundary under \eqref{u0-PME}, and one-dimensional compressible Navier--Stokes with $H^k$ data, \rev{$k \geq 9$}, yields \eqref{nse-class}. The assumptions are therefore natural for parabolic problems with smooth data, and our requirement is weaker than that of the closest comparable result \cite[Thm~3.9]{de2025approximation}, which needs $s \geq (d+9)/2$ and provides only fixed-time estimates. They nonetheless exclude several regimes of practical interest: solutions near a free boundary, where the profile loses differentiability (this is the case for the full-domain Barenblatt experiment of Appendix~\ref{apdx-PME}, as noted in Remark~\ref{rmk:barenblatt-smoothness}); shocks and low-regularity weak solutions; and turbulent regimes for Navier--Stokes. Within our proof strategy the requirement cannot be relaxed without either shrinking the admissible class or changing the integral representation, since the correspondence between the polynomially weighted integral of Lemma~\ref{lemma:parseval} and the Sobolev norm is what controls the variance of the estimator. Two directions seem worth pursuing. The first is to trade Sobolev regularity against decay conditions on $\hat{u}$, as in \cite{neufeld2023universal}, although such information is rarely available for solutions of non-linear PDEs. The second, which we regard as more promising, is to note that $s_2 > (d+2)/2$ enters through the normalisability of $(1+\|\mathbf{a}\|^2)^{-\lambda_a}$ in \eqref{pi}, and hence through the sampling measure rather than the ridgelet representation itself; replacing the isotropic weight by one adapted to the directions in which $u$ is rough may permit the hypothesis to be stated on an anisotropic or Besov scale, rather than uniformly across all $d$ spatial directions.
\end{remark}}
\subsection{Proof of Corollary \ref{cor:PME}}
\label{sec:proofPME}
\begin{proof}
    The bound of \eqref{PME-bound} follows from an application of Theorem \ref{thm1}, so we now focus on bounding the PDE residual. \rev{To simplify notation, in this proof we write $(0,T)$ for the interval $(t_0,T)$, on which $u$ is smooth.} We take $m=2$ for simplicity, but the computation is easily generalised to any $m\ge 1$. Using $\partial_tu- \Delta(u^2)=0$, we have
    \begin{equation}
    \begin{aligned}
       \mathcal{J}_{\text{PDE}}(u_N)  &= \int_{(0,T) \times D} (\partial_t u_N - \Delta (u^2_N))^2~dtdx \\[1ex] &= \int_{(0,T) \times D} (\partial_t (u_N-u)  - \Delta (u^2_N - u^2))^2~dtdx \\[1ex]
       &\le 2\|u_N - u\|_{H^1_t L^2_x}^2 + 2\int_{(0,T)\times D} |\Delta (u^2_N - u^2)|^2~dtdx.
    \end{aligned}
\end{equation}
Now, Markov's inequality states that $\mathbb{P}(X> \eta) \leq \mathbb{E}[X] \frac{1}{\eta}$ for $X$ non-negative, $\eta >0$. In particular, for any $\delta \in (0,1)$,  from the bound in Theorem~\ref{thm1} we obtain with the choices $X= \| u-u_N \|_{H^{p}(0,T; H^q(D))}^2 $ and $\eta = \frac{1}{\delta} \frac{C_{\Omega} C_\pi \|\sigma^{(p+q)}\|_\infty^2 \rev{T|D|(p+1)\binom{q+d}{d}} \mathcal{L}_\psi}{N} \cdot \|u\|_{H^{p+s_1}(0,T; H^{q+s_2}(D))}^2$ that
      \begin{equation} 
   \mathbb{P}(X> \eta) \leq   \mathbb{E}_\Theta \left( \| u-u_N \|_{H^{p}(0,T; H^q(D))}^2 \right) \frac{1}{\eta} \leq \delta.  
    \end{equation}
This implies that $ \mathbb{P}(X\leq \eta)\geq 1-\delta$. That is,  with probability at least $1-\delta$ it holds that $X \leq \eta = \frac{1}{\delta} \frac{C_{\Omega} C_\pi \|\sigma^{(p+q)}\|_\infty^2 \rev{T|D|(p+1)\binom{q+d}{d}} \mathcal{L}_\psi}{N} \cdot \|u\|_{H^{p+s_1}(0,T; H^{q+s_2}(D))}^2$. 
\rev{We apply this once, with $X = \|u_N-u\|_{H^1_tH^2_x}^2$ (i.e.\ $p=1$, $q=2$), which dominates both $\|u_N-u\|_{H^1_tL^2_x}^2$ and $\|u_N-u\|_{L^2_tH^2_x}^2$.}
Then we have with probability $1-\delta$ that \begin{equation}
    \|u_N-u\|_{H^1_t L^2_x}^2 \le \frac{\mathcal{M}_\psi}{N\delta} \|u\|_{H^{1+s_1}_t H^{2+s_2}_x}^2,
\end{equation}
where $s_1>3/2, s_2 > (d+2)/2$.
Now we note that $\Delta(u^2_N-u^2)  = (u_N+u)\Delta e_N + 2\nabla(u_N+u) \nabla e_N + e_N\Delta (u_N+u)$, where $e_N:=u_N-u$. Combining this with $(a+b)^2 \le 2(a^2+b^2)$ and Holder's inequality,
\begin{equation}
    \begin{aligned}
          \int_{(0,T)\times D} |\Delta (u^2_N - u^2)|^2~dtdx &\le 8\|u_N+u\|_{L^\infty_{t,x}}^2 \|\Delta(u_N-u)\|_{L^2_{t,x}}^2 \\[1ex] &+8 \|\nabla(u_N+u)\|_{L^\infty_{t,x}}^2 \|\nabla(u_N-u)\|_{L^2_{t,x}}^2 \\[1ex] &+ 8 \|\Delta(u_N+u)\|_{L^\infty_{t,x}}^2 \|u_N-u\|_{L^2_{t,x}}^2.
    \end{aligned}
\end{equation}

We also have
\begin{equation}
\begin{aligned}
     \| u\|_{L^\infty_{t,x}} + \| \nabla u\|_{L^\infty_{t,x}} + \| \Delta u\|_{L^\infty_{t,x}} \le C_{emb}\|u\|_{L^\infty_t H^{2+k}_x},
\end{aligned}
\end{equation} for $k >d/2$, where $C_{emb}$ is the constant arising from the Sobolev embedding $H^{k}_x \hookrightarrow L^\infty_{x}$.
Then on the same event, letting $L:=\| u_N\|_{L^\infty_{t,x}} + \| \nabla u_N\|_{L^\infty_{t,x}} + \| \Delta u_N\|_{L^\infty_{t,x}}$, we have with probability $1-\delta$ that
\begin{equation}
    \begin{aligned}
         \int_{(0,T)\times D} |\Delta (u^2_N - u^2)|^2~dtdx &\le 16(L^2+ C_{emb}^2\|u\|_{L^\infty_t H^{2+k}_x}^2)\| u_N-u\|_{L^2_t H^2_x}^2 \\[1ex]
          &\le \frac{16(L^2+ C_{emb}^2\|u\|_{L^\infty_t H^{2+k}_x}^2) \mathcal{M}_\psi }{N\delta} \|u\|_{H^{1+s_1}_t H^{2+s_2}_x}^2.
    \end{aligned}
\end{equation} 
All in all, we get
\begin{equation} \label{residual-pme-aux}
    \begin{aligned}
           \mathcal{J}_{\text{PDE}}(u_N)  &\le  \frac{(32(L^2+C_{emb}^2\|u\|_{L^{\infty}_{t}H^{2+k}_x}^2)+2) \mathcal{M}_\psi }{N\delta} \rev{\|u\|_{H^{1+s_1}_t H^{2+s_2}_x}^2},
    \end{aligned}
\end{equation}
with probability $1-\delta$. Taking $s_1 = 2$ and $s_2 = k+1$ with $k > d/2$, so that $s_2 > (d+2)/2$, we can estimate the leading norm by $\|u\|_{H^{3}_t H^{3+k}_x}^2$. For the general case $m \ge 2$, the argument can be adapted using a mean-value formula to re-write $u_N^m - u^m$ in terms of $e_N$. In this case, we get an extra constant in the coefficient of \eqref{residual-pme-aux}, $C_m$, which will be polynomial in $m$. This leads to the claimed result \eqref{pme-residual-1}. In summary, taking initial data $u_0$ satisfying \eqref{u0-PME}, we can find a RaNN which approximates the corresponding solution to the PDE such that the expected loss of the PDE residual is inversely proportional to the number of random features, and the rate is independent of dimension.
\end{proof}

\subsection{Proof of Corollary~\ref{cor:NS}}
\label{sec:proofNS}
\begin{proof}
    Since we are dealing with a system of equations, we need to illustrate how the result of Theorem \ref{thm1} applies. Firstly, given a couple $(\rho, u)$ belonging to class \eqref{nse-class}, the representation result of Proposition \ref{prop:u-rep} tells us that we can find $\psi_\rho, \psi_u \in \mathcal{S}(\mathbb{R})$ such that we have the representations (assuming $\rho, u$ have been smoothly extended to the full space):
    \begin{equation}\begin{aligned}
        \rho(t,x) &= \int_\mathbb{R} \int_{\mathbb{R}^{d}} \int_\mathbb{R} (A_\rho) (\tau, \mathbf{a}, b) \sigma (\tau t+ \mathbf{a} \cdot \mathbf{x}-b)~d\tau d\mathbf{a}db, \\[1ex] u(t,x) &=       \int_\mathbb{R} \int_{\mathbb{R}^{d}} \int_\mathbb{R} (A_u) (\tau, \mathbf{a}, b) \sigma (\tau t+ \mathbf{a} \cdot \mathbf{x}-b)~d\tau d\mathbf{a}db,
    \end{aligned}
    \end{equation}
    where $A_\rho := \mathcal{R}_{\psi_\rho} \rho$ and $A_u := \mathcal{R}_{\psi_u} u$.
    Then as in Theorem \ref{thm1} we can introduce two unbiased estimators
    \begin{equation}
        \begin{aligned}
            \rho_N(t,x) &=  \frac{1}{N} \sum_{i=1}^N  \frac{A_\rho (\tau_i, \mathbf{a}_i, b_i)}{\pi(\tau_i, \mathbf{a}_i, b_i)} \sigma (\tau_i t + \mathbf{a}_i \cdot \mathbf{x} -b_i), \\[1ex]
            u_N(t,x) &=  \frac{1}{N} \sum_{i=1}^N  \frac{A_u(\tau_i, \mathbf{a}_i, b_i)}{\pi(\tau_i, \mathbf{a}_i, b_i)} \sigma (\tau_i t + \mathbf{a}_i \cdot \mathbf{x} -b_i) .
        \end{aligned}
    \end{equation}
    The result of Theorem \ref{thm1} gives us the following bounds immediately:
    \begin{equation}
        \begin{aligned}
            \|\rho_N - \rho\|_{L^2_t H^q_x}^2 &\le \frac{C_{\Omega} C_\pi \|\sigma^{(q)}\|_\infty^2 \rev{T|D|(q+1)
} \mathcal{L}_\psi}{N} \cdot \|\rho\|_{H^{s_1}_t H^{q+s_2}_x}^2, \\[1ex] \|u_N - u\|_{L^2_t H^q_x}^2 &\le \frac{C_{\Omega} C_\pi \|\sigma^{(q)}\|_\infty^2 \rev{T|D|(q+1)} \mathcal{L}_\psi}{N} \cdot \|u\|_{H^{s_1}_tH^{q+s_2}_x}^2,
        \end{aligned}
    \end{equation}
    for $s_1,s_2>3/2$ since we are in dimension $d=1$.
    From now on we label
    \begin{equation} \label{nse-Mpsi}
             \rev{   \mathcal{M}_\psi := C_{\Omega} C_\pi \|\sigma^{(q+1)}\|_\infty^2\, 2T|D|(q+1)\, \mathcal{L}_\psi }
    \end{equation}
    The pair $v_N(t,x):=(\rho_N(t,x), u_N(t,x))$ can be interpreted as the outputs of a single random neural network of width $2N$, from which the first claim follows. We now bound the loss functions $\mathcal{J}_{PDE}^1$ and $\mathcal{J}_{PDE}^2$. For $\mathcal{J}_{PDE}^1$ we note the expression $\rho_N u_N - \rho u= \rho_N(u_N-u) + u(\rho_N-\rho)$ to get 
    \begin{equation}
        \partial_x (\rho_N u_N-\rho u) = \partial_x \rho_N (u_N-u) + \rho_N \partial_x (u_N-u) + \partial_x u (\rho_N - \rho) + u \partial_x (\rho_N-\rho).\end{equation} Therefore,
    \begin{equation} 
        \begin{aligned}
            \mathcal{J}_{PDE}^1&(\mathbf{v}_N) = \int_{(0,T)\times D} |\partial_t (\rho_N -\rho)+ \partial_x (\rho_N u_N -\rho u)|^2 \\[1ex] &\le 2\|\partial_t (\rho_N - \rho)\|_{L^2_{t,x}}^2 +8\|\partial_x \rho_N \|_{L^\infty_{t,x}}^2\|u_N-u\|_{L^2_{t,x}}^2 + 8\|\rho_N\|_{L^\infty_{t,x}}^2\|\partial_x (u_N-u)\|_{L^2_{t,x}}^2 \\[1ex] &+ 8\|\partial_xu\|_{L^\infty_{t,x}}^2\|\rho_N-\rho\|_{L^2_{t,x}}^2+ 8\|u\|_{L^\infty_{t,x}}^2\|\partial_x(\rho_N-\rho)\|_{L^2_{t,x}}^2 \\[1ex] &\le \left[8 (\|\partial_x \rho_N \|_{L^\infty_{t,x}}^2+\|\rho_N\|_{L^\infty_{t,x}}^2+\|\partial_xu\|_{L^\infty_{t,x}}^2 + \|u\|_{L^\infty_{t,x}}^2)+2\right] \\[1ex] &\qquad \times (\|\rho_N-\rho\|_{H^1_{t,x}}^2 + \|u_N-u\|_{H^{1}_{t,x}}^2)
\\[1ex] &=: \mathcal{B_N}(\|\rho_N-\rho\|_{H^1_{t,x}}^2 + \|u_N-u\|_{H^{1}_{t,x}}^2).
        \end{aligned}
    \end{equation}
    Note that $\|\partial_x u\|_{L^\infty_{t,x}}$ is finite since \rev{$u \in C(0,T; H^9(D))$}. An upper bound for the coefficient  $\mathcal{B}_N$ is $8(\|\rho_N\|_{W^{1,\infty}_{t,x}}^2+\|u\|_{W^{1,\infty}_{t,x}}^2)+2$. Applying Markov's inequality once, as in the proof of Corollary~\ref{cor:PME}, to $X := \|\rho_N-\rho\|^2_{H^1_tH^1_x} + \|u_N-u\|^2_{H^1_tH^2_x}$, we have with probability $1-\delta$ that $X \le \frac{\mathcal{M}_\psi}{N\delta}\big(\|\rho\|^2_{H^{1+s_1}_tH^{1+s_2}_x} + \|u\|^2_{H^{1+s_1}_tH^{2+s_2}_x}\big)$. All estimates below hold on this event. In particular,
    \begin{equation}
        \begin{aligned}
            \mathcal{J}_{PDE}^1(\mathbf{v}_N) \le \frac{(8(L+\|u\|_{W^{1,\infty}_{t,x}}^2)+2) \mathcal{M}_\psi}{N\delta}
 ( \|\rho\|^2_{H^{1+s_1}_t H^{1+s_2}_x} + \|u\|^2_{H^{1+s_1}_t H^{1+s_2}_x}),
        \end{aligned}
    \end{equation} where $s_1, s_2 > 3/2$,
    \begin{equation}\label{nse-weights-assumption}L:=\|\rho_N\|_{W^{1,\infty}_{t,x}}^2+ \|u_N\|_{W^{1,\infty}_{t,x}}^2  +  \|\rho\|_{W^{1,\infty}_{t,x}}^2 +  \|u\|_{W^{1,\infty}_{t,x}}^2,
    \end{equation} which is guaranteed to be finite since the neural network $(\rho_N,u_N)$ is smooth.
    For the momentum equation, we use $\rho_N u_N^2 - \rho u^2 = \rho_N u_N (u_N-u) + u  (\rho_N u_N-\rho u)$, and so
    \begin{equation} \begin{aligned}
        \partial_x (\rho_Nu_N^2) - \partial_x (\rho u^2) &= \partial_x (u_N \rho_N)(u_N-u) + u_N \rho_N \partial_x (u_N-u) +  (\rho_N u_N-\rho u) \partial_x u \\[1ex] &+u  \partial_x(\rho_N u_N-\rho u).
        \end{aligned}
    \end{equation}
    Therefore, in a similar way to the continuity equation, with probability $1-\delta$,
    \begin{equation}
        \begin{aligned}
              \mathcal{J}_{PDE}^2(\mathbf{v}_N) &= \int_{(0,T)\times D} |\partial_t (\rho_N u_N -\rho u)+ \partial_x (\rho_N u_N^2 -\rho u^2) - \mu\partial_x^2(u_N-u)|^2 dxdt\\[1ex] &\le \mathcal{C}_L  (\|\rho_N -\rho\|^2_{H^{1}_{t,x}} + (2\mu+1) \|u_N-u\|^2_{H^{1}_t H^{2}_x}),
        \end{aligned}
    \end{equation}
    where
    \begin{equation}
        \begin{aligned}
            \mathcal{C}_L = ~&32\|u_N\|_{W^{1,\infty}_{t,x}}^2 + 64\|u \|_{W^{1,\infty}_{t,x}}^2 + 80\|\rho\|_{W^{1,\infty}_{t,x}}^2 + 16\|\rho_N\|_{W^{1,\infty}_{t,x}}^2\|u_N\|_{W^{1,\infty}_{t,x}}^2 \\[1ex] &+ 64\|u\|_{W^{1,\infty}_{t,x}}^2\|u_N\|_{W^{1,\infty}_{t,x}}^2
        \end{aligned}
    \end{equation}   The assumption \eqref{nse-weights-assumption} ensures that $\mathcal{C}_L$ can be bounded as
    \begin{equation}
       C_L \le  (16L^2 +112L+2\mu)+ 64(L+1)\|u\|_{W^{1,\infty}_{t,x}}^2.
    \end{equation}
    Note that $\|\partial_t u\|_{L^\infty_{t,x}}$ is finite; this can be seen by rewriting the momentum equation as $\partial_tu = -u\partial_xu + \mu \rho^{-1}\partial_x^2u$, and using the regularity $u \in C(0,T; H^k(D))$ for \rev{$k \ge 9$}, along with the Sobolev embedding $H^1 \hookrightarrow L^\infty$.  Then we have with probability $1-\delta$ that
    \begin{equation*}
        \begin{aligned}
           & \mathcal{J}_{PDE}^2(\mathbf{v}_N) \le \\[1ex] &\frac{  \left[ 16L^2 +112L+2\mu+ 64(L+1)\|u\|_{W^{1,\infty}_{t,x}}^2 \right]\mathcal{M}_\psi }{N\delta} ( \|\rho\|^2_{H^{1+s_1}_t H^{1+s_2}_x} +  (2\mu+1) \|u\|^2_{H^{1+s_1}_t H^{2+s_2}_x}).
        \end{aligned}
    \end{equation*}
    Noting that $s_1, s_2 > 3/2 $, we get
    \begin{equation} \label{nse-final}
        \begin{aligned} &
             (\mathcal{J}_{PDE}^1 + \mathcal{J}_{PDE}^2)(\mathbf{v}_N) \le \\[1ex] & \frac{\left[ 16L^2 +128L+2\mu+2+ 64(L+1)\|u\|_{W^{1,\infty}_{t,x}}^2 \right]\mathcal{M}_\psi}{N\delta} ( \|\rho\|^2_{H^{3}_t H^{3}_x} +  (2\mu+1) \|u\|^2_{H^{3}_t H^{4}_x}),
        \end{aligned}
    \end{equation} which is guaranteed to be finite if we take $(\rho_0, u_0) \in (H^k(D))^2$ for \rev{$k \ge 9$}. This can be seen by noting that the continuity equation can be re-written as $\partial_t \rho = -u\partial_x\rho - \rho \partial_x u$, and the momentum equation as $\partial_tu = -u\partial_xu + \mu \rho^{-1}\partial_x^2u$. Then using the regularity \eqref{nse-class}, one can bound the higher order Sobolev norms on the right-hand side of \eqref{nse-final}\rev{: each time derivative of $u$ costs two spatial derivatives and each time derivative of $\rho$ costs one, so for instance $\partial_t^3 u \in L^2(0,T; H^{k-5}(D))$, which requires $k \ge 9$ for $\|u\|_{H^3_t H^4_x}$ to be finite}.
\end{proof}

\section{Experiments} \label{sec:experiments}
In this section we give further details on the numerical experiments carried out in Section \ref{sec:experiments-main}. All experiments were performed using a NVIDIA RTX 3080 GPU with 10GB VRAM. The code is available at \url{https://github.com/M6LI/RaNN-Expressivity-NonLinearPDEs}.
\subsection{Porous Medium Equation} \label{apdx-PME}
Recall that we consider the Barenblatt-Kompaneets-Zeldovich solution
\begin{equation} 
    u(t,x) = \frac{1}{t^\alpha} \left( b - \frac{m-1}{2m} \beta \frac{\|x\|^2}{t^{2\beta}} \right)^{\frac{1}{m-1}}_+,
\end{equation}
where $\|\cdot\|$ is the $\ell^2$ norm, $(\cdot)_+$ is the positive part and
$\alpha = \frac{d}{d(m-1)+2}$ and $\beta = \frac{1}{d(m-1)+2}$. We set $m=2$.

This solution is compactly supported but not differentiable at the edges of the support, which causes difficulty for numerical schemes. Since this solution is generated by a Dirac delta-valued initial data, we fix a small time $t_0 >0$ and take the initial data to be $u_0(x) := u(t_0, x)$ and solve on the shifted time domain $(t_0, T+t_0)$. 



We aim to investigate whether a convergence rate of $N^{-1/2}$ can be observed in practice. To this end, we train a RaNN to approximate the Barenblatt profile \eqref{barenblatt} in dimensions $d=1,...,5$. In each dimension, we take a set of widths $N \in \{N_1, ..., N_{k} \}$, train the network for each width and plot the relative error of the final network against the true solution. The RaNN includes a Fourier feature layer, where frequencies $\omega_j$ are sampled from $\mathcal{N}(0,10^2)$. For each dimension $d$ and width $N$, we sample $M = 10N$ points (to ensure the problem remains well-posed) with a mixed strategy: $50\%$ of the points are sampled uniformly and $50\%$ are sampled uniformly on $[0.2, 0.8]^d$, which is a box focused on the support of the solution. Then we find weights $\mathbf{W} =\{ W_i \}_{i=1}^N$ which minimise the Ridge regression loss

\begin{equation} \mathcal{L}(\mathbf{W}) = \frac{1}{M} \sum_{i=1}^M \| \hat{\mathbf{u}}(t_i,\mathbf{x}_i) - \mathbf{y}_i \|_2^2 + \lambda \| \mathbf{W} \|_{2}^{2}, \qquad \lambda = 10^{-5}. \end{equation}

We compute the closed-form solution $\hat{\mathbf{W}}$ directly (using a Cholesky decomposition), and then evaluate the relative error. We train the RaNN five times for each width. The mean relative $L^2$ errors are plotted against the widths in Figure \ref{fig:errors-PME}. The errors are also plotted on log-log scales in Figure \ref{fig:errors-PME-log}. The key observation here is that the RaNN error points lie close to the $C/\sqrt{N}$ curve, which supports the upper bound of Theorem \ref{thm1}. Note that in $d=1,2,3$ we take $t_0 = 0.1$ (i.e. the initial data is set to $u(t_0, x)$, where $u$ is the Barenblatt profile from \eqref{barenblatt}) and in $d=4,5$ this is reduced to $t_0 = 0.01$ to keep the problem computationally manageable.
\begin{remark}[Smoothness of Barenblatt profile] \label{rmk:barenblatt-smoothness}
    Note that the lack of smoothness of the Barenblatt profile at the edges of the support places this experiment outside of the assumptions in Corollary \ref{cor:PME}. Nonetheless, if we restrict the solution to a compact subset strictly contained within the support, the result does apply. The full-domain experiment can therefore be viewed as a setting which slightly exceeds the strict theoretical assumptions. 
\end{remark}
\begin{remark}[Sampling distributions used in the experiments] \label{rmk:sampling}
    In \rev{the experiments of Section~\ref{sec:experiments-main} and Figures~\ref{fig:errors-PME} and~\ref{fig:errors-PME-log}}, the weights of the randomised neural networks are sampled from a Gaussian distribution rather than the distribution $\pi$ \eqref{pi} used in Theorem \ref{thm1}. This choice was made in order to test the robustness of the approximation rate beyond the specific (and somewhat unusual) $\pi$ used in Theorem \ref{thm1}. \rev{Results for $\pi$ and a similar Student-$t$ distribution are reported in Table~\ref{tab:pi-slopes} and Figures~\ref{fig:errors-PME-pi} and~\ref{fig:errors-PME-t}.} Empirically, we observe consistent $N^{-1/2}$ scaling behaviour which supports the relevance of our theory when sampling from a more practical Gaussian distribution. \rev{We note that the proof of Theorem~\ref{thm1} does rely on the heavy-tailed form of $\pi$: Lemma~\ref{lemma:parseval} controls a polynomially weighted integral of the ridgelet transform by a Sobolev norm of $u$, and for a Gaussian $\pi$ this weight would become Gaussian, which can only be controlled if the Fourier transform of $u$ decays faster than a Gaussian. Extending the result to Gaussian weights would therefore require a different integral representation of $u$, whereas an extension to other heavy-tailed distributions appears possible.}
\end{remark}

\begin{figure}[h]
    \centering
    \begin{subfigure}{0.45\textwidth}
        \includegraphics[width=\linewidth]{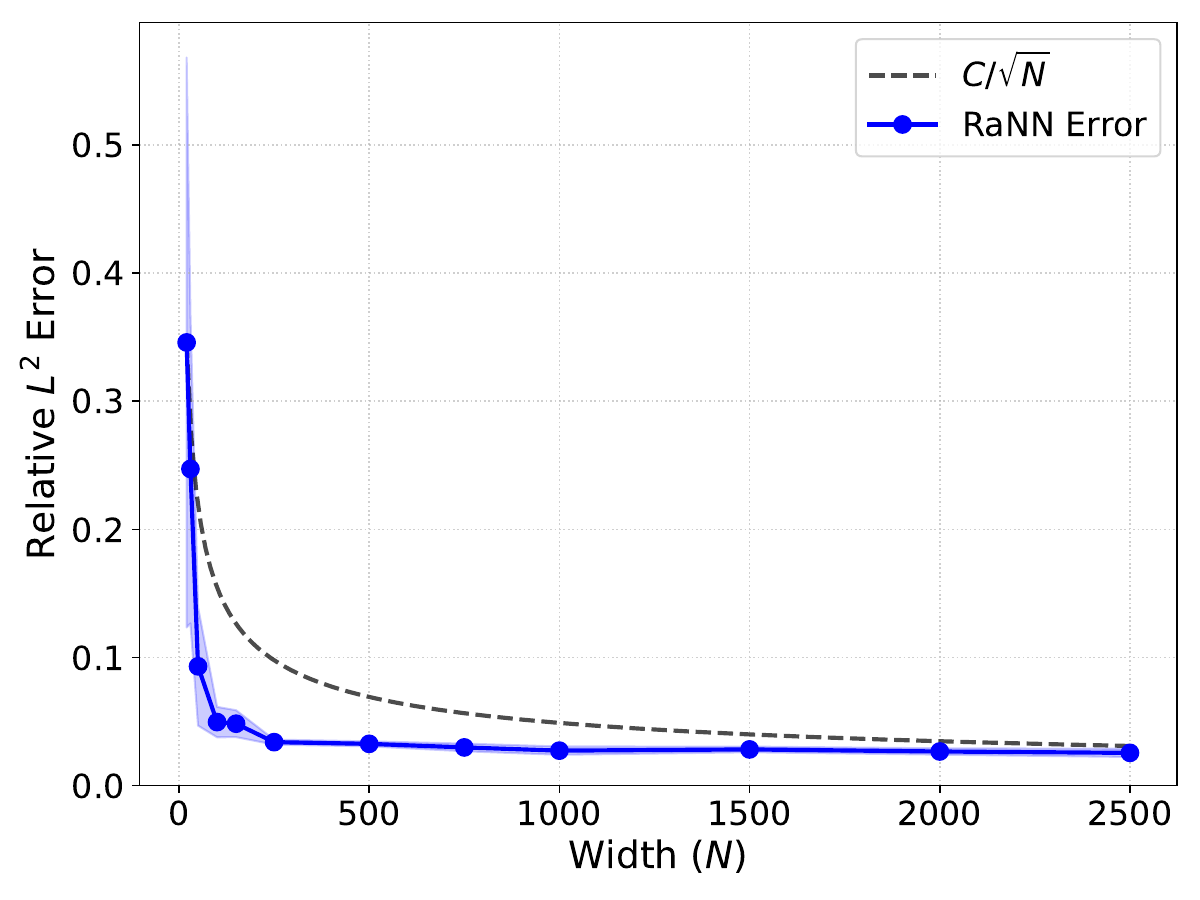}
        \caption{$d = 1$}
    \end{subfigure}
    \begin{subfigure}{0.45\textwidth}
        \includegraphics[width=\linewidth]{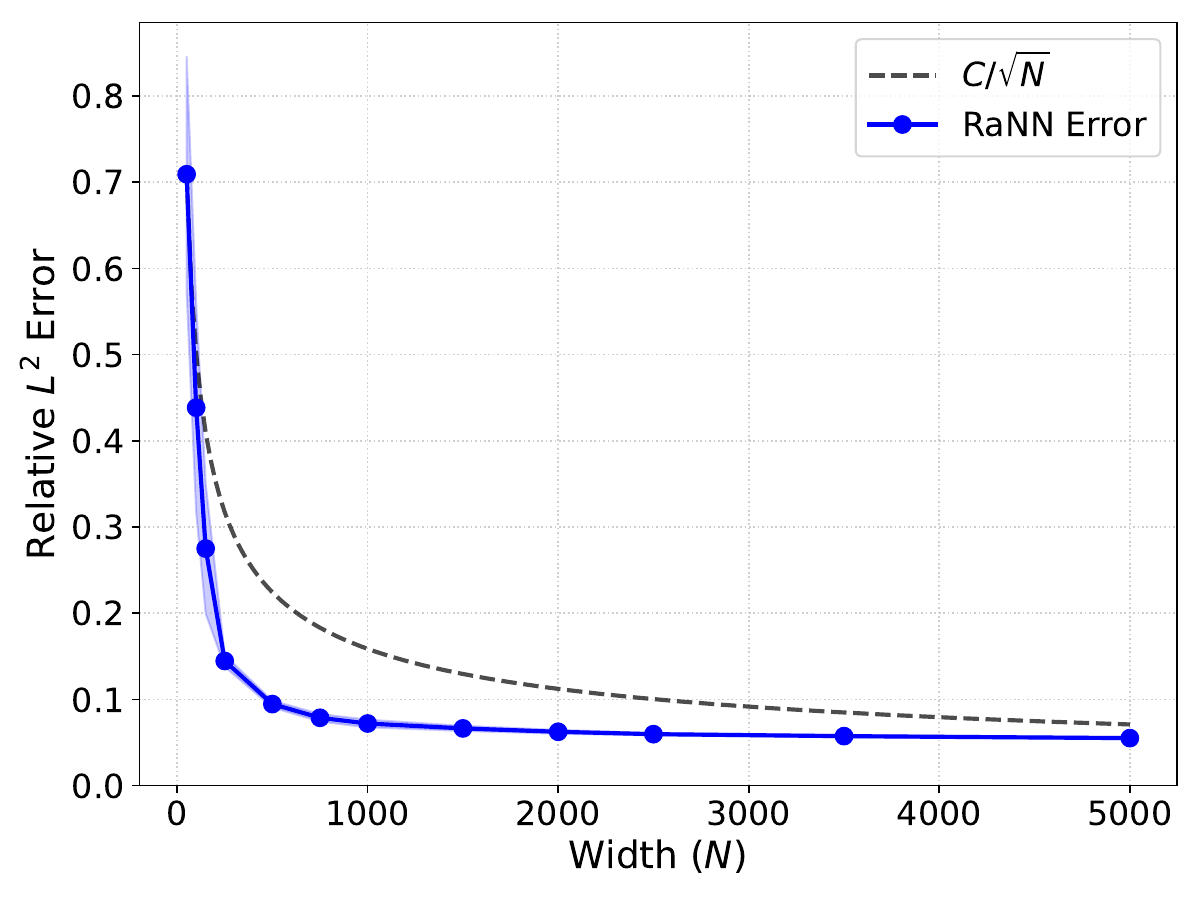}
        \caption{$d = 2$}
    \end{subfigure}
    \begin{subfigure}{0.45\textwidth}
        \includegraphics[width=\linewidth]{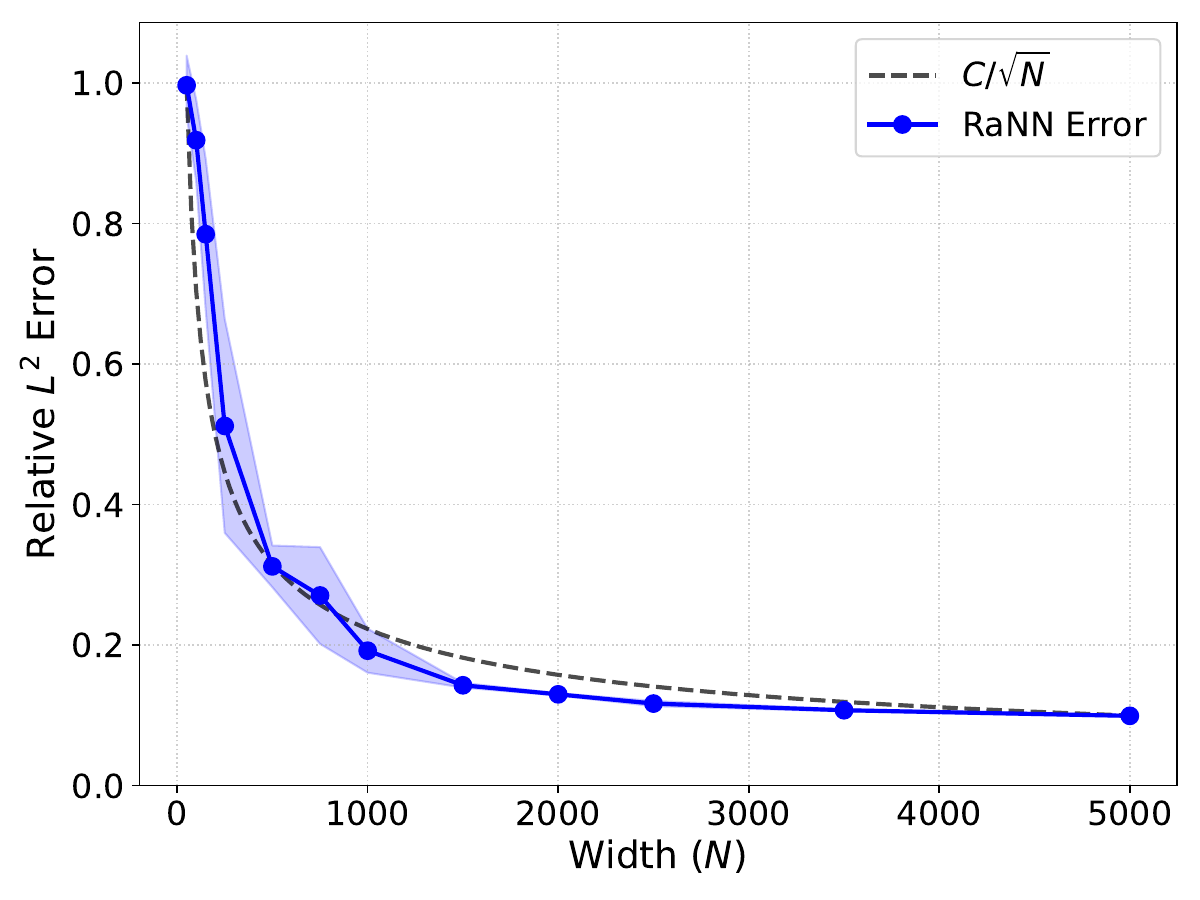}
        \caption{$d = 3$}
    \end{subfigure}

    \begin{subfigure}{0.45\textwidth}
        \includegraphics[width=\linewidth]{figures/d4.pdf}
        \caption{$d = 4$}
    \end{subfigure}
    \begin{subfigure}{0.45\textwidth}
        \includegraphics[width=\linewidth]{figures/d5.pdf}
        \caption{$d = 5$}
    \end{subfigure}
    \begin{subfigure}{0.45\textwidth}
        \centering
        \caption*{} 
    \end{subfigure}

    \caption{Approximation error of RaNNs of varying width  for solving PMEs in dimensions $d = 1,\ldots,5$. The shaded band indicates the region within one standard deviation of the mean relative $L^2$ error.}
    \label{fig:errors-PME}
\end{figure}

\begin{figure}[h]
    \centering
    \begin{subfigure}{0.45\textwidth}
        \includegraphics[width=\linewidth]{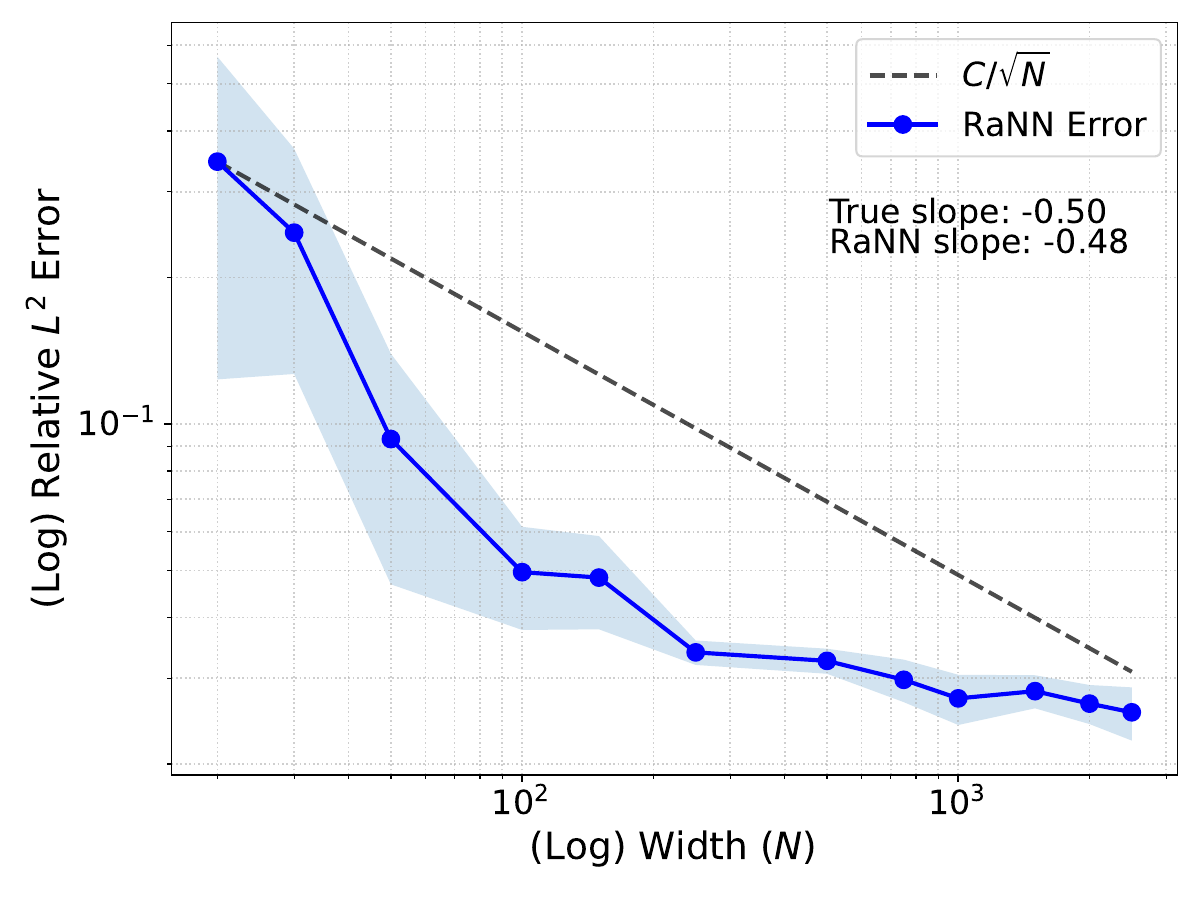}
        \caption{$d = 1$}
    \end{subfigure}
    \begin{subfigure}{0.45\textwidth}
        \includegraphics[width=\linewidth]{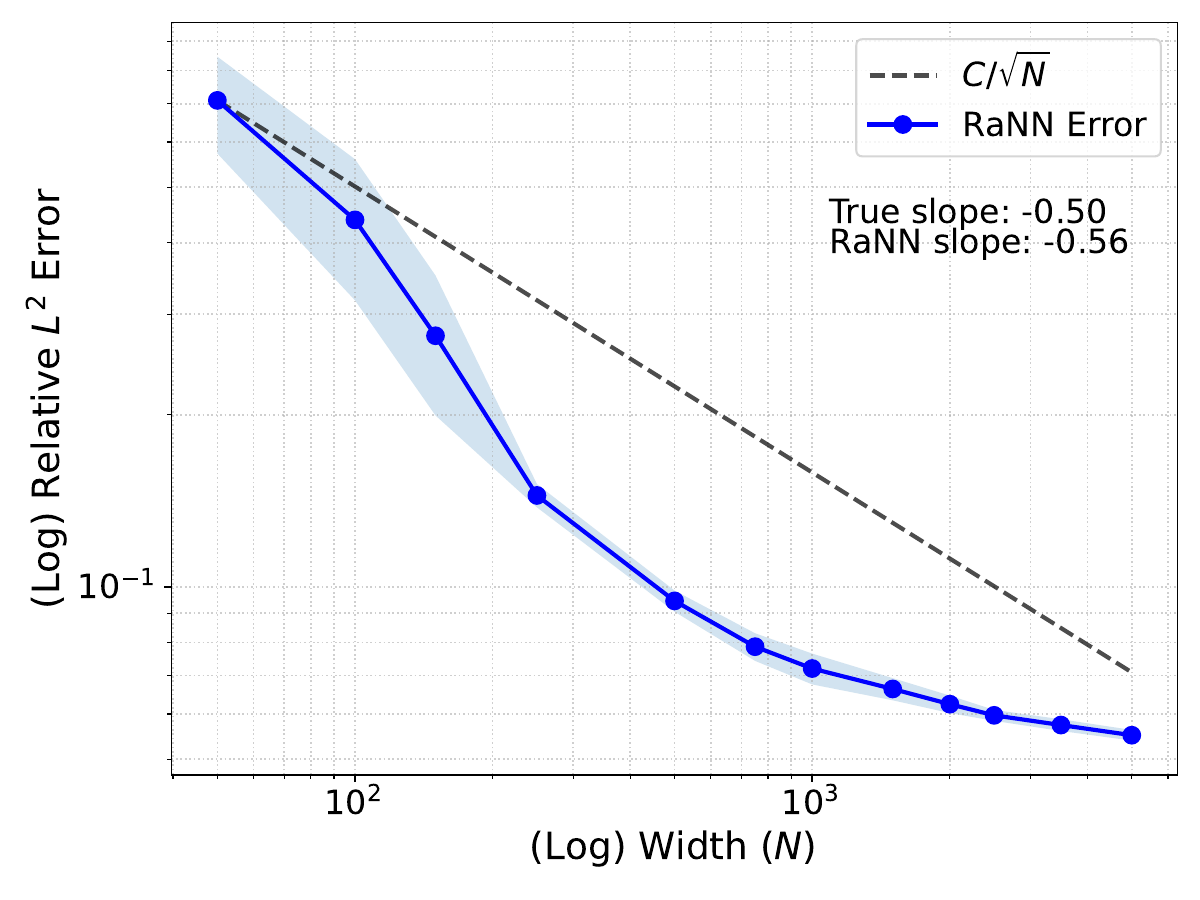}
        \caption{$d = 2$}
    \end{subfigure}
    \begin{subfigure}{0.45\textwidth}
        \includegraphics[width=\linewidth]{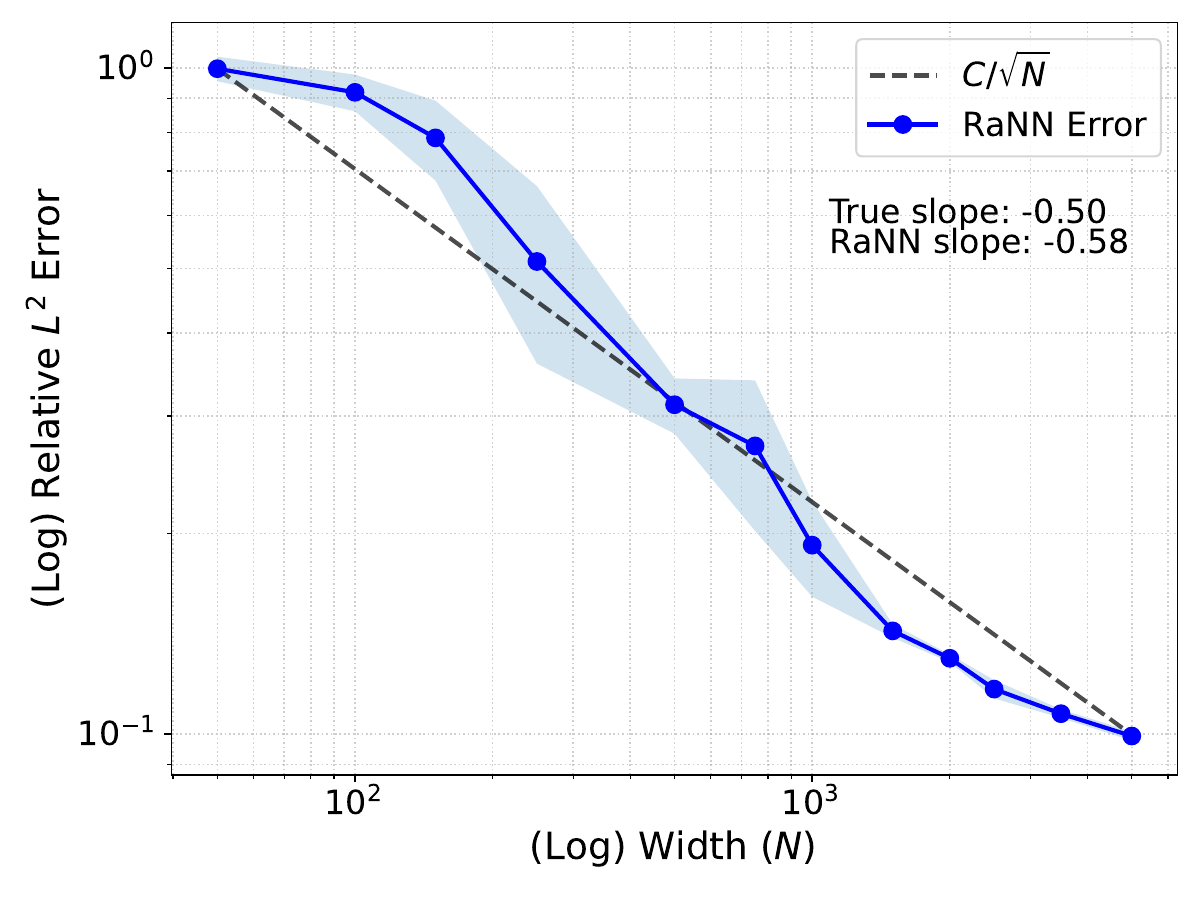}
        \caption{$d = 3$}
    \end{subfigure}

    \begin{subfigure}{0.45\textwidth}
        \includegraphics[width=\linewidth]{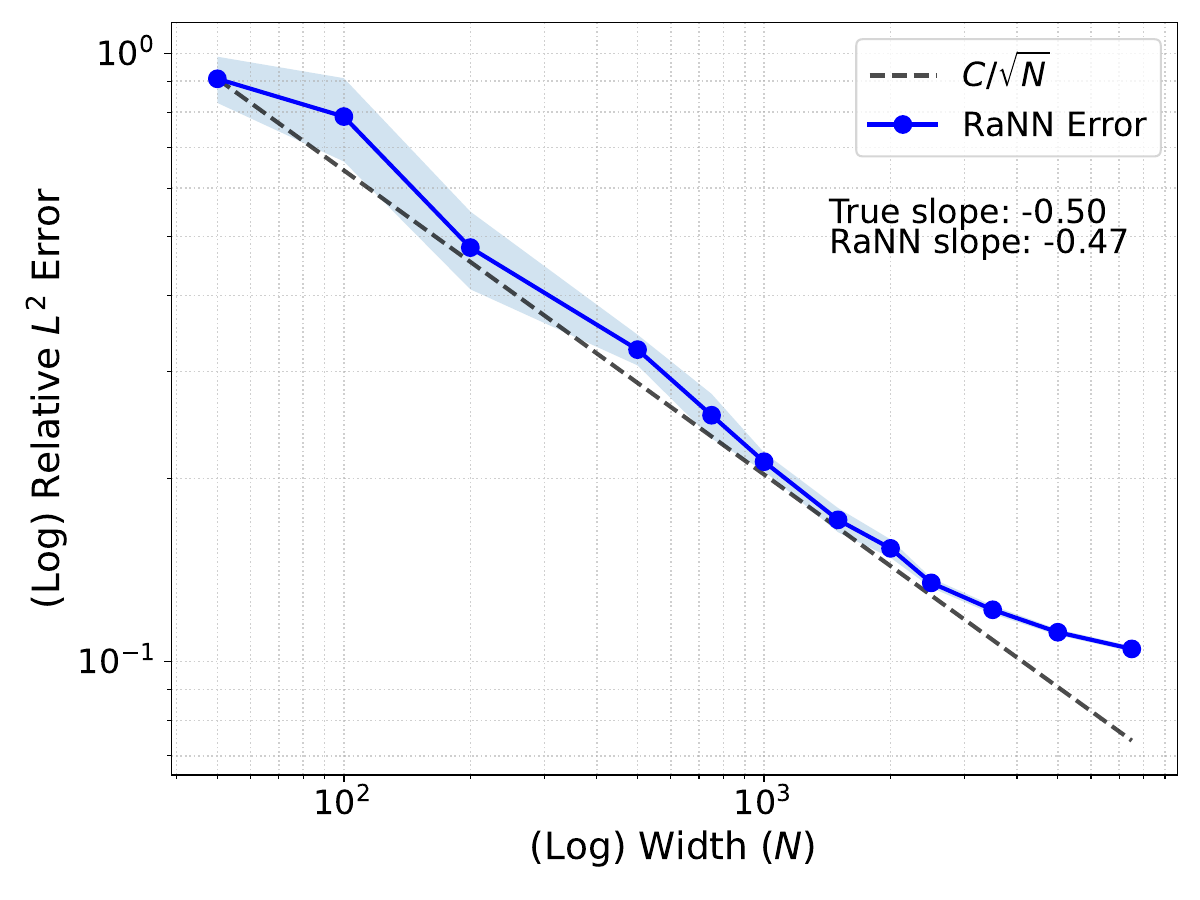}
        \caption{$d = 4$}
    \end{subfigure}
    \begin{subfigure}{0.45\textwidth}
        \includegraphics[width=\linewidth]{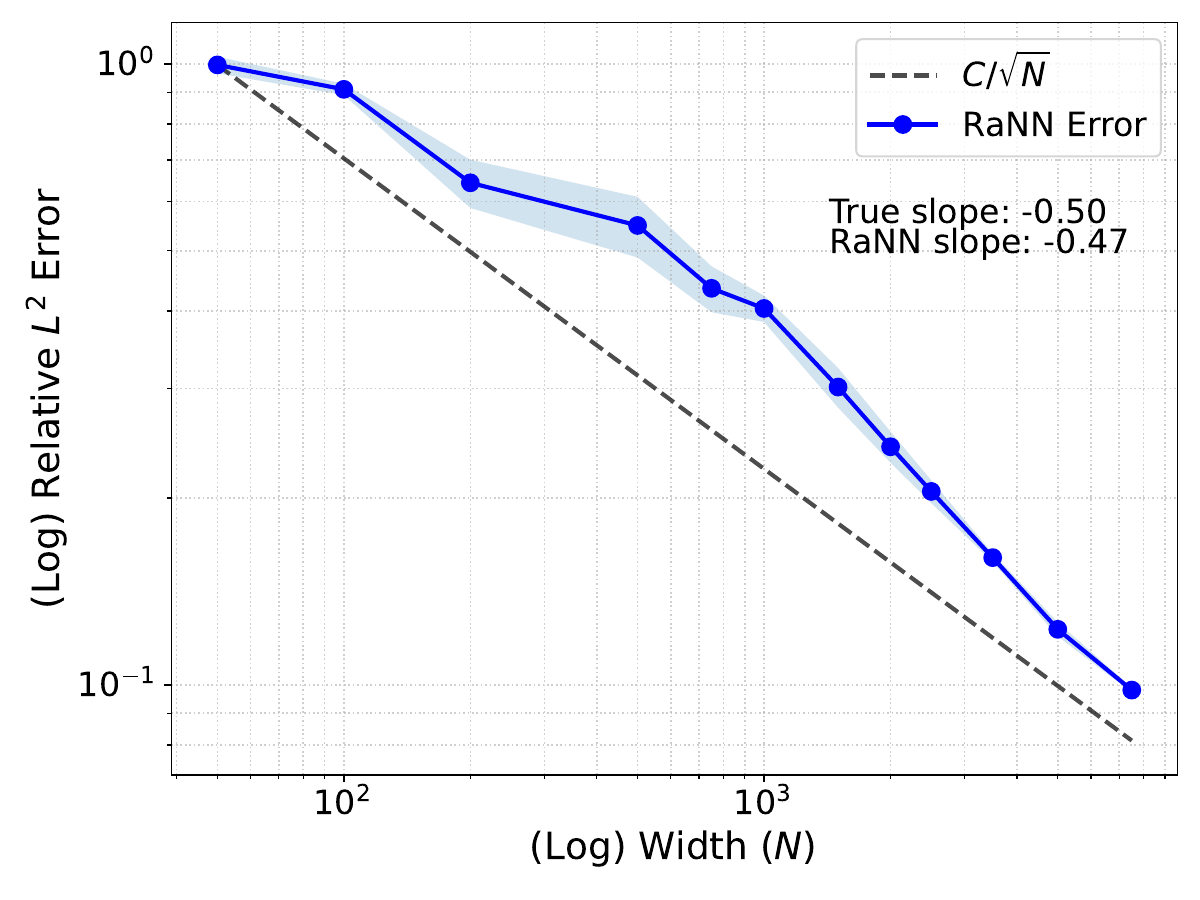}
        \caption{$d = 5$}
    \end{subfigure}
    \begin{subfigure}{0.45\textwidth}
        \centering
        \caption*{} 
    \end{subfigure}

    \caption{Log-log plot of the relative $L^2$ error versus the width $N$. The reference scaling $C/\sqrt{N}$ and the measured RaNN slope are shown. The shaded band indicates the region within one standard deviation of the mean relative $L^2$ error.}
    \label{fig:errors-PME-log}
\end{figure}

\rev{
\paragraph{Heavy-tailed weight distributions.} We have repeated the experiments above with the Gaussian law of the hidden weights replaced by two heavy-tailed laws. The first is the distribution $\pi$ in \eqref{pi} itself, with $\lambda_\tau = 3$ and $\lambda_a = (d+5)/2$: the frequencies $\tau_i$ and $\mathbf{a}_i$ are drawn independently from a Student-$t$ and a $d$-dimensional multivariate Student-$t$ distribution with $5$ degrees of freedom, and the biases $b_i$ from a standard Cauchy distribution. Since $\sigma = \cos$ is $2\pi$-periodic, only $b_i$ modulo $2\pi$ enters the network. The frequencies are applied to rescaled inputs $(s_t t, s \mathbf{x})$ for constants $s_t, s > 0$, which amounts to applying $\pi$ on a rescaled domain and only affects the constant in \eqref{thm-ineq}. For the second law, $(\tau_i, \mathbf{a}_i)$ are drawn jointly from a $(d+1)$-dimensional multivariate Student-$t$ distribution with $5$ degrees of freedom, whose density is proportional to $(1 + \tau^2 + \|\mathbf{a}\|^2)^{-(d+6)/2}$ (with the same rescaling of the inputs), and the biases are drawn uniformly from $[0, 2\pi]$. Under this variant, $\tau_i$ and $\mathbf{a}_i$ have the same marginal distributions as under $\pi$ but are not independent. The remaining set-up is as above, except that we use the ridge parameter $\lambda = 10^{-8}$, three independent draws per width, and the time shifts $t_0 = 6.8\times10^{-3}, 1.9\times10^{-2}, 2.8\times10^{-2}$ for $d=1,2,3$, chosen such that the amplitude of the solution decays by the same factor $2.5$ over the time window in each dimension, and $t_0 = 10^{-2}$ for $d=4,5$. Both laws use the same configuration and random seeds. The fitted slopes are reported in Table~\ref{tab:pi-slopes} and the error curves in Figures~\ref{fig:errors-PME-pi} and~\ref{fig:errors-PME-t}. The two laws give very similar results: the slopes differ by at most $0.04$ and the errors at the largest width by at most $7\%$. For both laws, four of the five $95\%$ confidence intervals contain $-1/2$; at $d=4$ the interval lies below $-1/2$, i.e.\ the error decays slightly faster than $N^{-1/2}$, which is consistent with Theorem~\ref{thm1} since it is an upper bound.

\begin{table}[h]
    \centering
    \small
    \caption{Fitted log-log slopes of the relative $L^2$ error for the PME with hidden weights drawn from $\pi$ and from the multivariate Student-$t$ variant. Slopes are fitted to the mean errors; the $95\%$ confidence intervals are computed from slopes fitted to each of the $3$ draws.}
    \label{tab:pi-slopes}
    \begin{tabular}{llll}
    \toprule
    $d$ & widths $N$ & $\pi$ & Student-$t$ variant \\
    \midrule
    $1$ & $10, 14, 20, 50, 100, 150, 250, 500, 750, 1000, 1500, 2000, 2500$ & $-0.49 \pm 0.16$ & $-0.48 \pm 0.15$ \\
    $2$ & $20, 50, 100, 150, 250, 500, 750, 1000, 1500, 2000, 2500$ & $-0.53 \pm 0.09$ & $-0.53 \pm 0.14$ \\
    $3$ & $20, 50, 100, 150, 250, 500, 750, 1000, 1500, 2000, 2500$ & $-0.58 \pm 0.11$ & $-0.57 \pm 0.12$ \\
    $4$ & $250, 500, 750, 1000, 1500, 2000, 2500, 5000$ & $-0.60 \pm 0.04$ & $-0.57 \pm 0.04$ \\
    $5$ & $250, 500, 750, 1000, 1500, 2000, 2500, 5000$ & $-0.52 \pm 0.08$ & $-0.50 \pm 0.05$ \\
    \bottomrule
    \end{tabular}
\end{table}

\begin{figure}[h]
    \centering
    \begin{subfigure}{0.45\textwidth}
        \includegraphics[width=\linewidth]{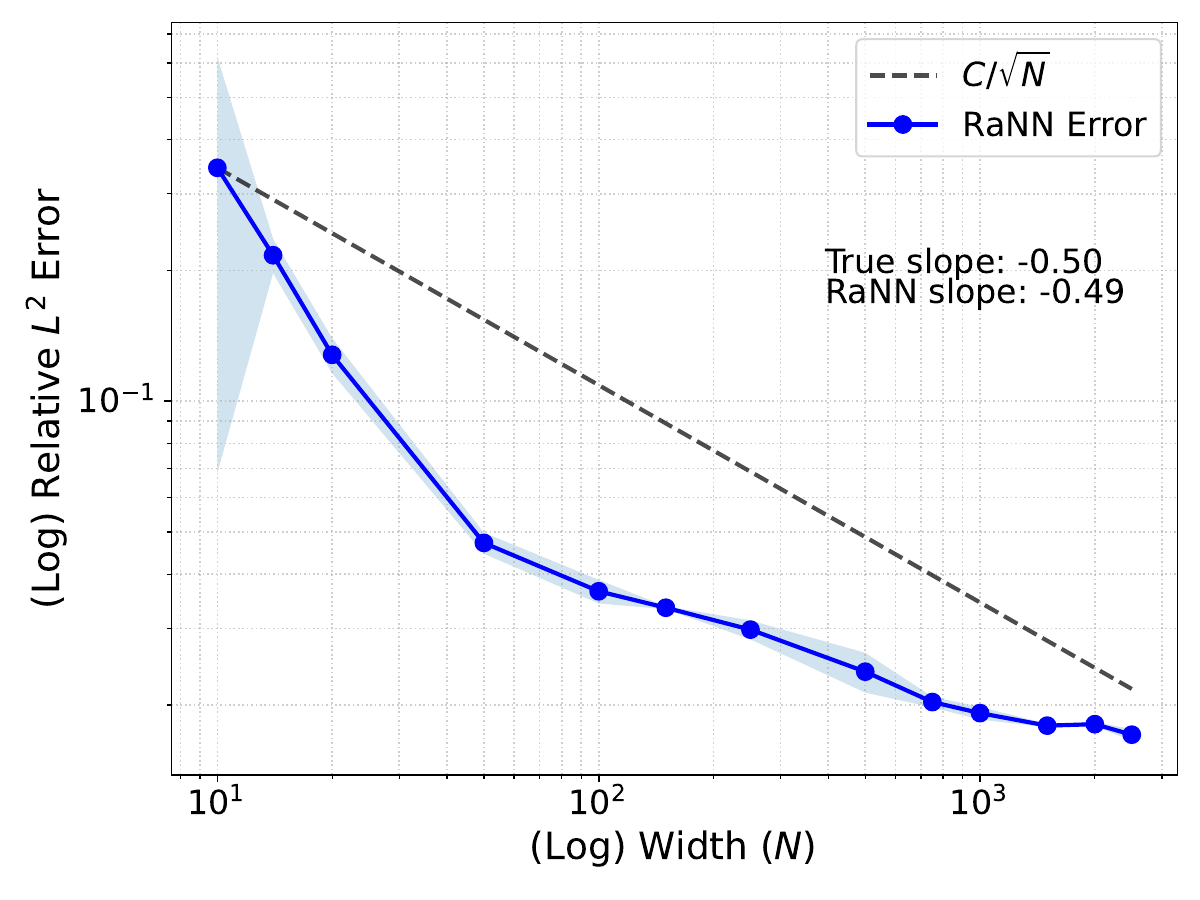}
        \caption{$d = 1$}
    \end{subfigure}
    \begin{subfigure}{0.45\textwidth}
        \includegraphics[width=\linewidth]{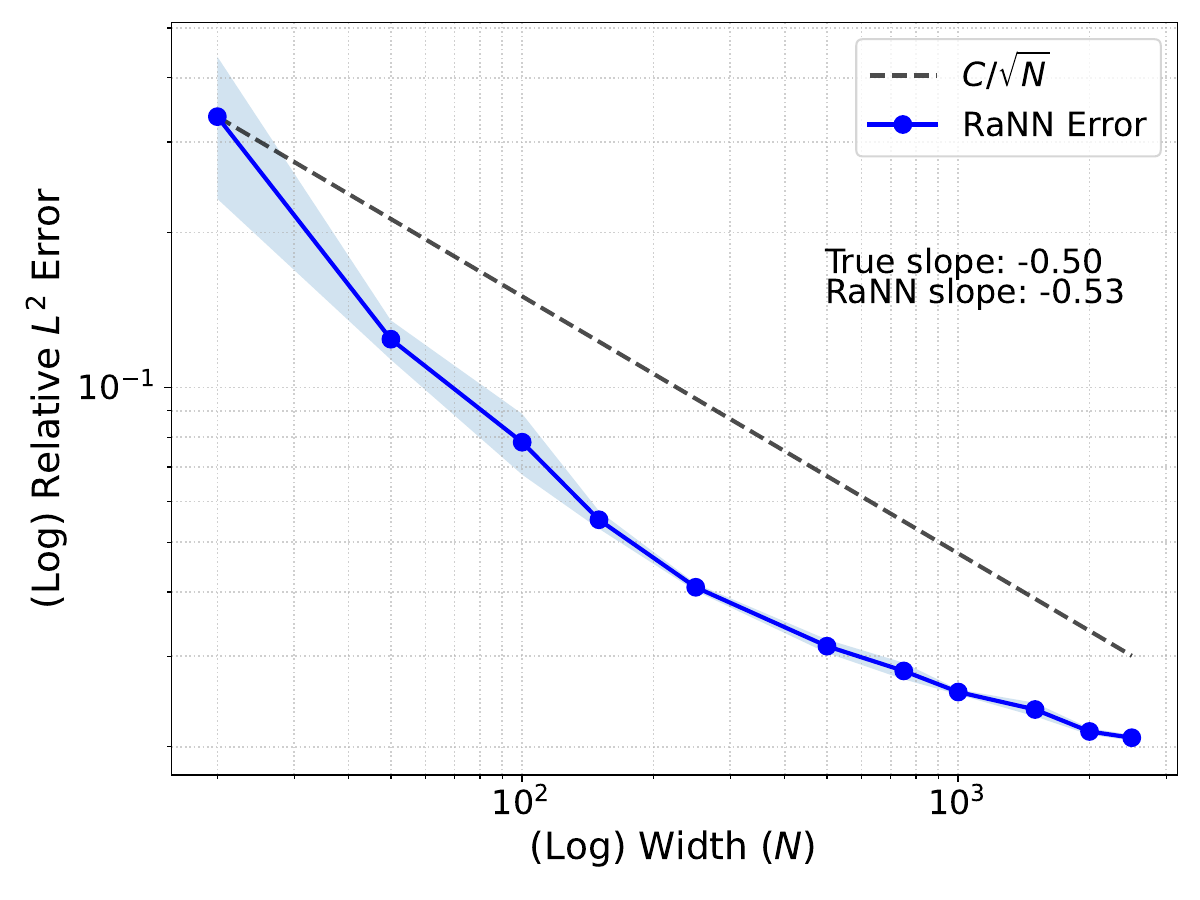}
        \caption{$d = 2$}
    \end{subfigure}
    \begin{subfigure}{0.45\textwidth}
        \includegraphics[width=\linewidth]{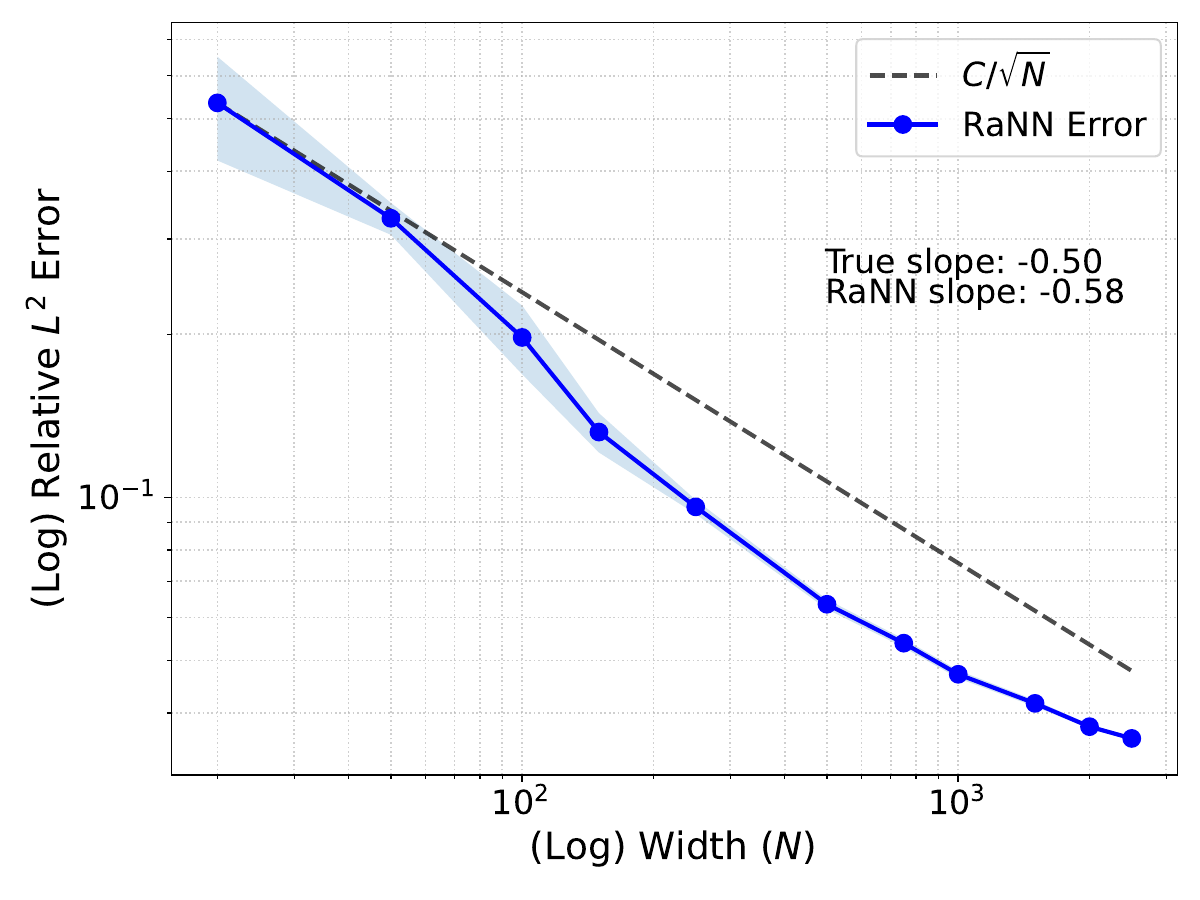}
        \caption{$d = 3$}
    \end{subfigure}
    \begin{subfigure}{0.45\textwidth}
        \includegraphics[width=\linewidth]{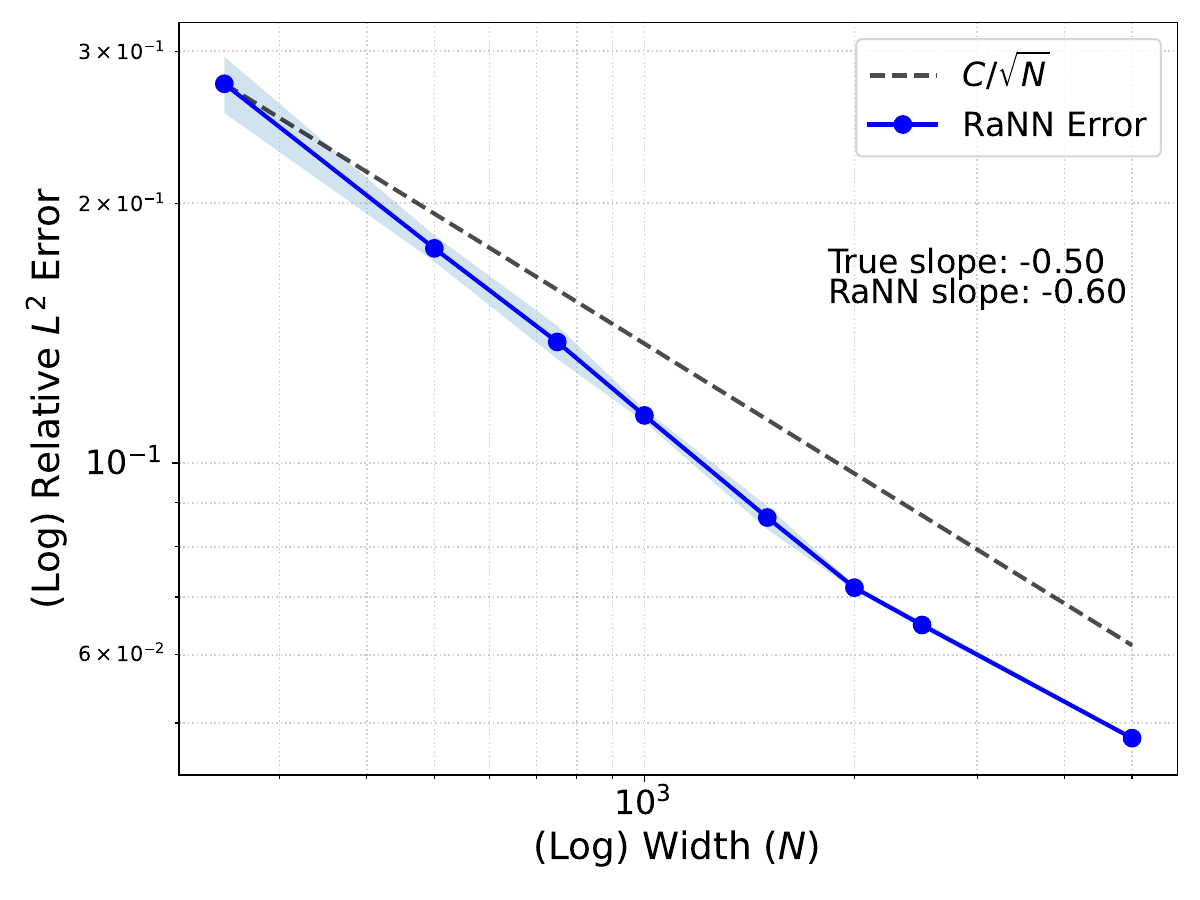}
        \caption{$d = 4$}
    \end{subfigure}
    \begin{subfigure}{0.45\textwidth}
        \includegraphics[width=\linewidth]{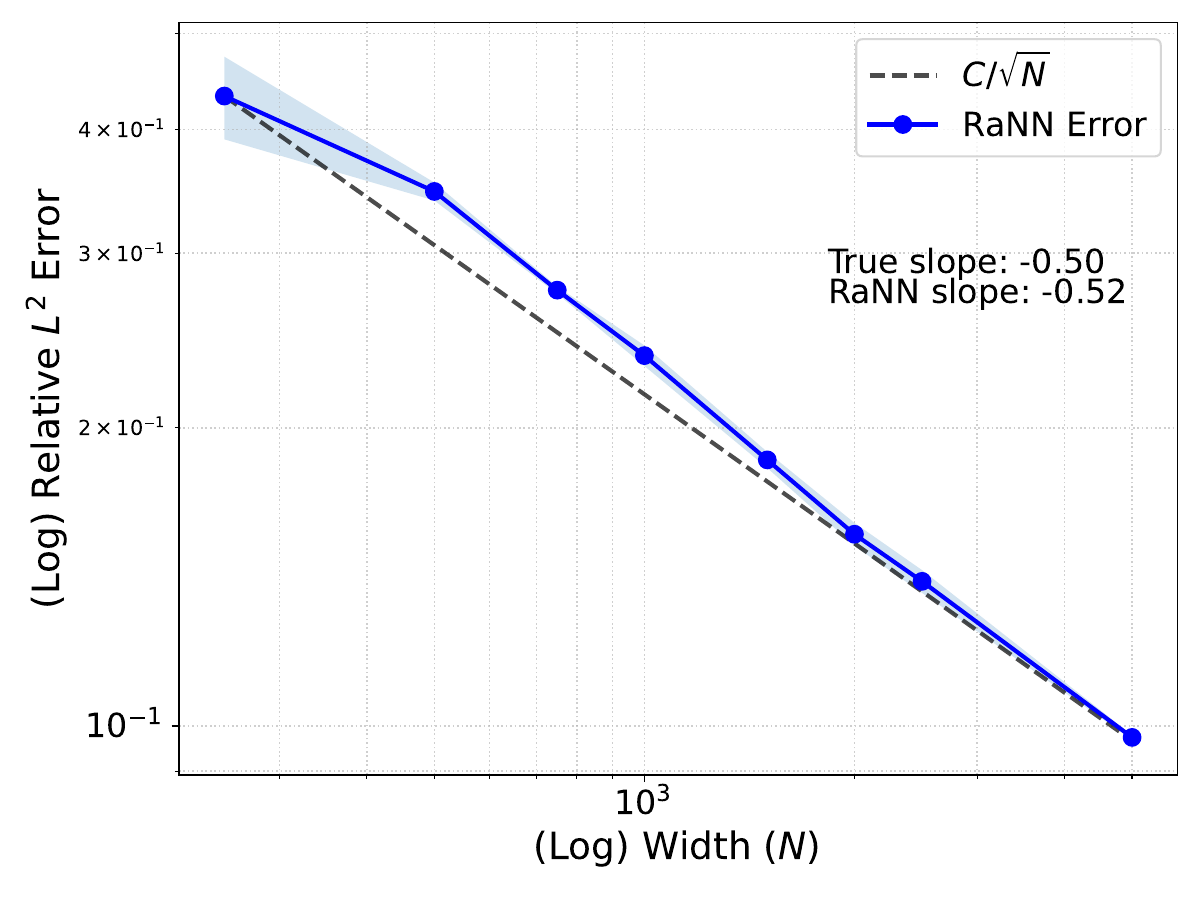}
        \caption{$d = 5$}
    \end{subfigure}
    \caption{Log-log plot of the relative $L^2$ error versus the width $N$ for the PME with hidden weights drawn from $\pi$. The shaded band indicates the region within one standard deviation of the mean over $3$ draws.}
    \label{fig:errors-PME-pi}
\end{figure}

\begin{figure}[h]
    \centering
    \begin{subfigure}{0.45\textwidth}
        \includegraphics[width=\linewidth]{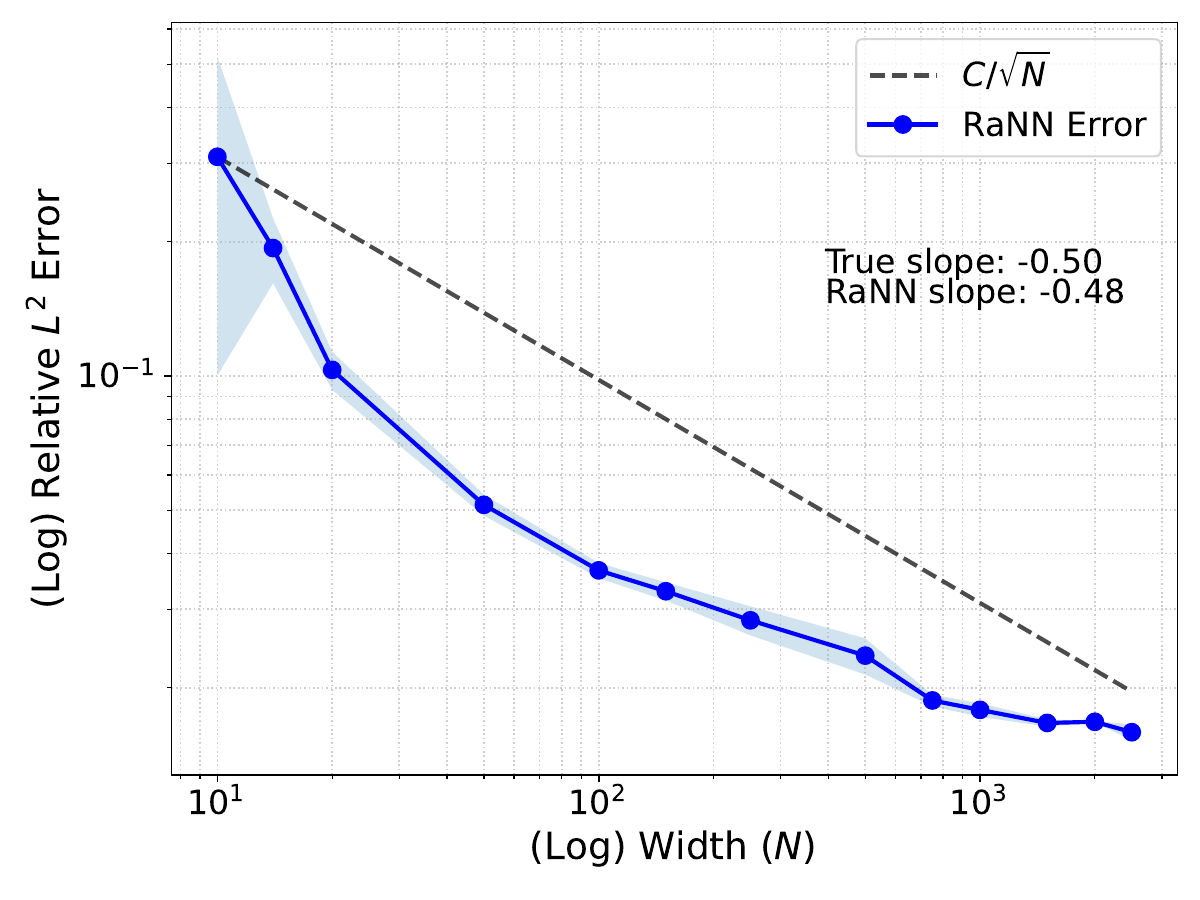}
        \caption{$d = 1$}
    \end{subfigure}
    \begin{subfigure}{0.45\textwidth}
        \includegraphics[width=\linewidth]{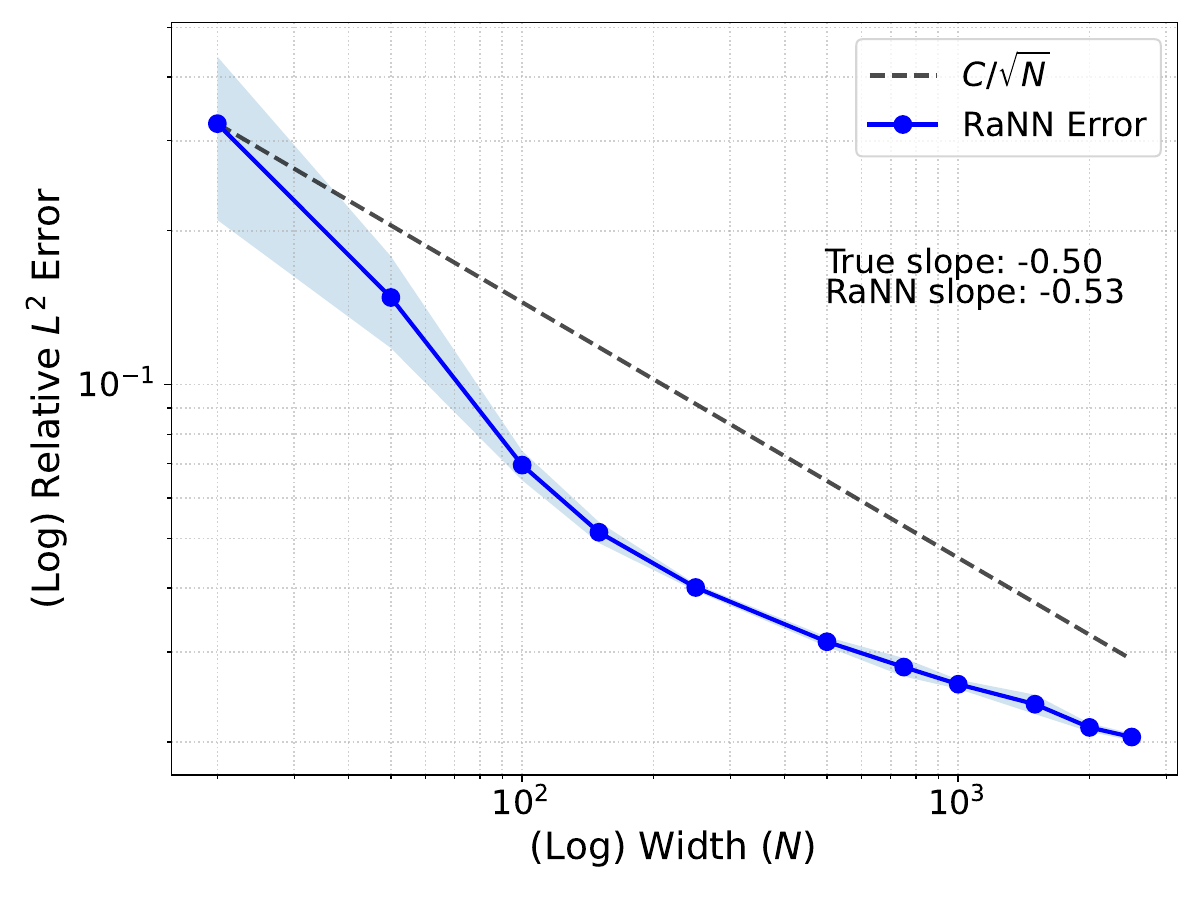}
        \caption{$d = 2$}
    \end{subfigure}
    \begin{subfigure}{0.45\textwidth}
        \includegraphics[width=\linewidth]{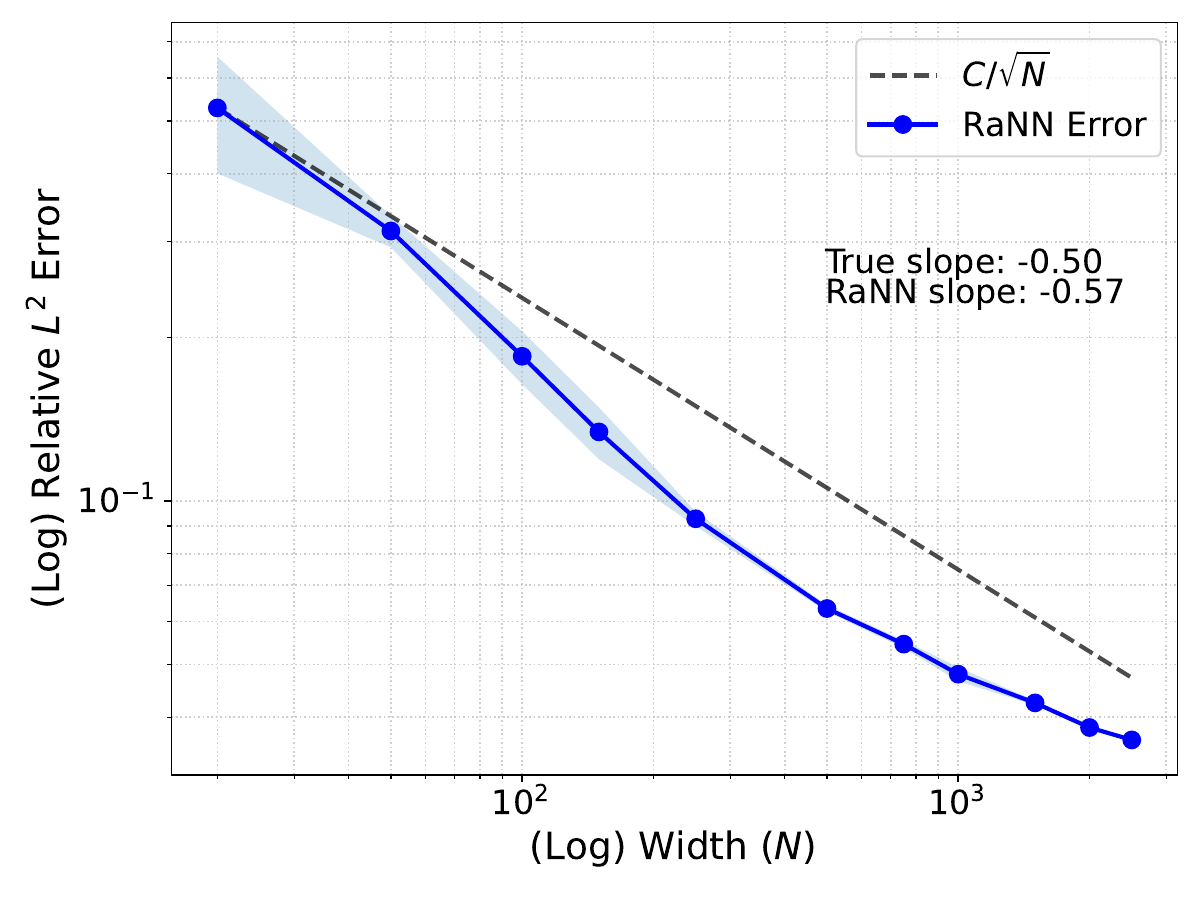}
        \caption{$d = 3$}
    \end{subfigure}
    \begin{subfigure}{0.45\textwidth}
        \includegraphics[width=\linewidth]{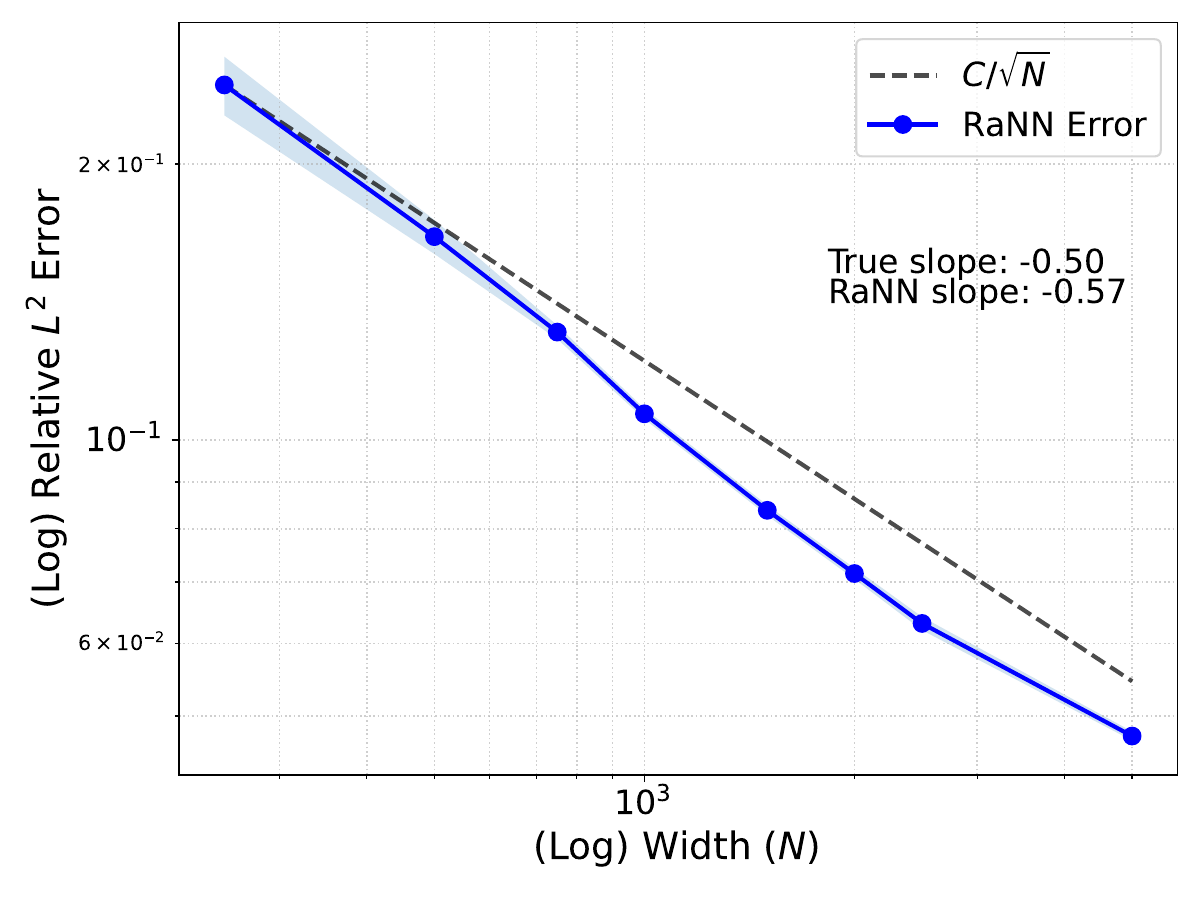}
        \caption{$d = 4$}
    \end{subfigure}
    \begin{subfigure}{0.45\textwidth}
        \includegraphics[width=\linewidth]{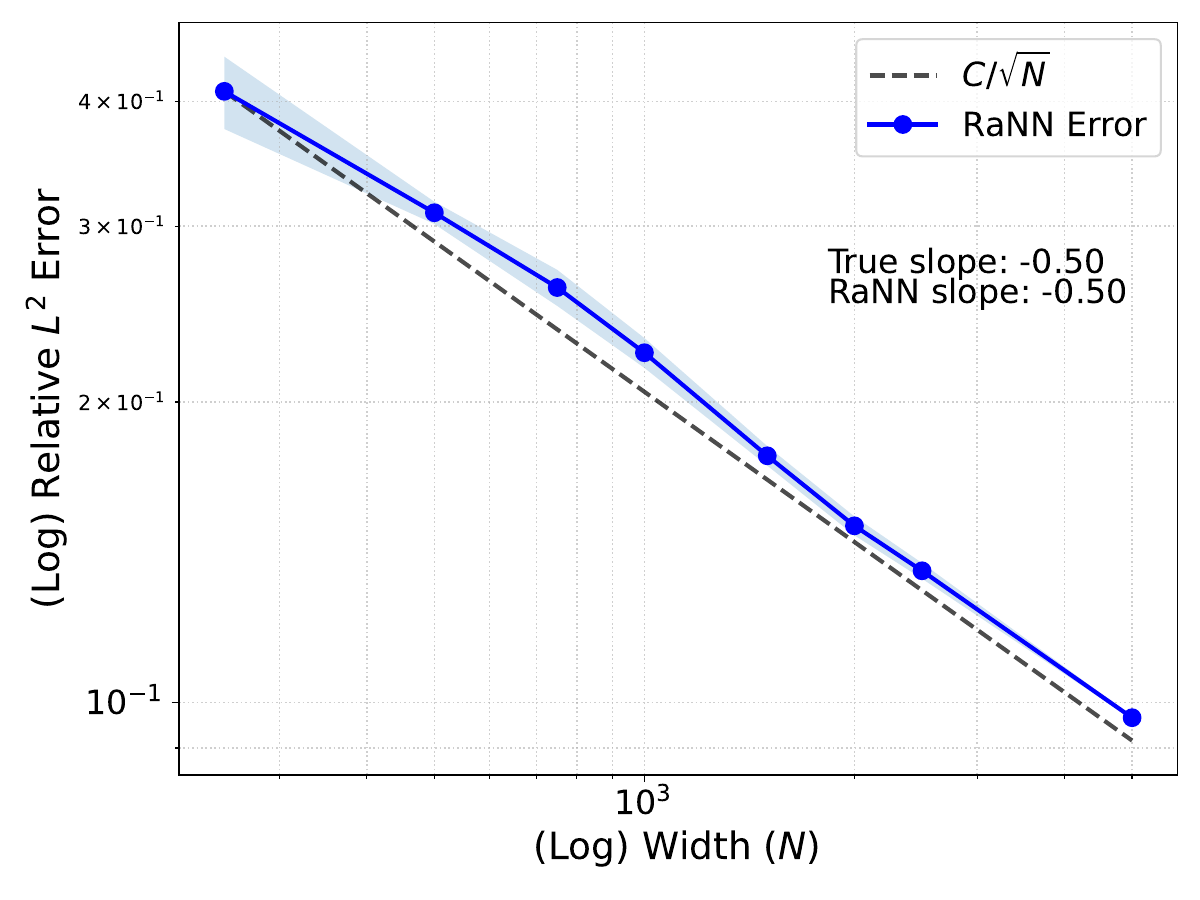}
        \caption{$d = 5$}
    \end{subfigure}
    \caption{Log-log plot of the relative $L^2$ error versus the width $N$ for the PME with hidden weights drawn from the multivariate Student-$t$ variant of $\pi$. The shaded band indicates the region within one standard deviation of the mean over $3$ draws.}
    \label{fig:errors-PME-t}
\end{figure}
}

\rev{
\subsection{Training on the PDE residual loss} \label{apdx:residual}
The experiments of Section~\ref{sec:experiments-main} fit the read-out weights by regression onto the exact solution, which isolates the approximation error. Here we instead train RaNNs only by minimising the PINN loss, and compare them with fully trainable PINNs. We consider the PME with $m=2$ in $d=1,2$. Since the residual involves second derivatives, both models are trained and evaluated on the smooth interior of the support, $\{(t,x) : t\in[0,T],\ \|x-x_0\| \le 0.8\,R(t)\}$, with $T=0.1$, where $R(t)$ is the radius of the support of \eqref{barenblatt}, and the exact solution is imposed as data on the parabolic boundary of this region (the initial slice and the lateral surface). There the Barenblatt solution is bounded away from zero, the equation is uniformly parabolic and the solution is smooth, so the hypotheses of Corollary~\ref{cor:PME} hold; no values of the solution in the interior are used for training. Both models minimise
\begin{equation*}
    \mathcal{J} = \mathcal{J}_{\mathrm{res}} + 10\,(\mathcal{J}_{\mathrm{ic}} + \mathcal{J}_{\mathrm{bc}}),
\end{equation*}
where $\mathcal{J}_{\mathrm{res}}$ is the discretisation of \eqref{pme-JPDE} on $20000$ interior collocation points and $\mathcal{J}_{\mathrm{ic}}, \mathcal{J}_{\mathrm{bc}}$ are the squared errors on $4000$ points of the initial slice and of the lateral boundary respectively, each term normalised by the corresponding norm of the exact solution.

The two models are compared at matched numbers $N$ of trainable parameters. The RaNN has $N$ read-out weights over $N/2$ random features whose hidden weights are \rev{drawn either from $\pi$ \eqref{pi} or from its multivariate Student-$t$ variant (Appendix~\ref{apdx-PME}),} and kept fixed. The fully trainable network is a $\tanh$ MLP with two hidden layers whose total number of parameters is the largest one not exceeding $N$. Both are trained with the same procedure: Adam ($1000$ steps, learning rate $10^{-3}$) followed by L-BFGS (up to $3000$ iterations). We report the RMS PDE residual and the relative $L^2$ error on $20000$ held-out interior points, averaged over $3$ independent draws of the random features (respectively initialisations of the MLP), in Table~\ref{tab:residual}.
\begin{table}[h]
    \centering
    \small
    \caption{RaNN versus fully trainable PINN, both trained on the PINN loss for the PME on the interior of the support. Slopes are fitted to the mean values; the $95\%$ confidence intervals are computed from slopes fitted to each of the $3$ draws. \rev{RaNN ($\pi$) and RaNN ($t$) denote features drawn from $\pi$ and from its multivariate Student-$t$ variant.}}
    \label{tab:residual}
    \setlength{\tabcolsep}{3pt}
    \begin{tabular}{llllllll}
    \toprule
    & & $N=80$ & $160$ & $320$ & $640$ & $1280$ & slope \\
    \midrule
    \multicolumn{8}{l}{\emph{RMS PDE residual}} \\
    $d=1$ & \rev{RaNN ($\pi$)} & \rev{$5.97\times10^{-1}$} & \rev{$4.32\times10^{-1}$} & \rev{$2.07\times10^{-1}$} & \rev{$1.34\times10^{-1}$} & \rev{$7.69\times10^{-2}$} & \rev{$-0.76 \pm 0.07$} \\
          & RaNN \rev{($t$)} & $6.57\times10^{-1}$ & $3.78\times10^{-1}$ & $2.18\times10^{-1}$ & $1.26\times10^{-1}$ & $8.05\times10^{-2}$ & $-0.76 \pm 0.08$ \\
          & PINN & $6.82\times10^{-2}$ & $2.22\times10^{-2}$ & $2.36\times10^{-2}$ & $2.76\times10^{-2}$ & $1.36\times10^{-2}$ & $-0.43 \pm 0.78$ \\
    $d=2$ & \rev{RaNN ($\pi$)} & \rev{$9.83\times10^{-1}$} & \rev{$2.98\times10^{-1}$} & \rev{$6.65\times10^{-2}$} & \rev{$3.26\times10^{-2}$} & \rev{$1.82\times10^{-2}$} & \rev{$-1.47 \pm 0.29$} \\
          & RaNN \rev{($t$)} & $9.04\times10^{-1}$ & $2.59\times10^{-1}$ & $7.44\times10^{-2}$ & $3.22\times10^{-2}$ & $1.64\times10^{-2}$ & $-1.46 \pm 0.12$ \\
          & PINN & $4.21\times10^{-2}$ & $3.60\times10^{-2}$ & $1.61\times10^{-2}$ & $1.36\times10^{-2}$ & $1.07\times10^{-2}$ & $-0.53 \pm 0.69$ \\
    \midrule
    \multicolumn{8}{l}{\emph{Relative $L^2$ error}} \\
    $d=1$ & \rev{RaNN ($\pi$)} & \rev{$6.89\times10^{-3}$} & \rev{$4.25\times10^{-3}$} & \rev{$1.75\times10^{-3}$} & \rev{$1.36\times10^{-3}$} & \rev{$7.07\times10^{-4}$} & \rev{$-0.82 \pm 0.14$} \\
          & RaNN \rev{($t$)} & $7.78\times10^{-3}$ & $3.25\times10^{-3}$ & $1.93\times10^{-3}$ & $1.00\times10^{-3}$ & $6.67\times10^{-4}$ & $-0.88 \pm 0.05$ \\
          & PINN & $1.45\times10^{-3}$ & $4.69\times10^{-4}$ & $4.75\times10^{-4}$ & $5.53\times10^{-4}$ & $2.64\times10^{-4}$ & $-0.47 \pm 0.79$ \\
    $d=2$ & \rev{RaNN ($\pi$)} & \rev{$1.25\times10^{-1}$} & \rev{$1.94\times10^{-2}$} & \rev{$3.15\times10^{-3}$} & \rev{$1.44\times10^{-3}$} & \rev{$7.59\times10^{-4}$} & \rev{$-1.85 \pm 0.37$} \\
          & RaNN \rev{($t$)} & $9.90\times10^{-2}$ & $1.52\times10^{-2}$ & $3.53\times10^{-3}$ & $1.43\times10^{-3}$ & $6.84\times10^{-4}$ & $-1.78 \pm 0.23$ \\
          & PINN & $2.56\times10^{-3}$ & $2.32\times10^{-3}$ & $1.13\times10^{-3}$ & $8.40\times10^{-4}$ & $6.27\times10^{-4}$ & $-0.55 \pm 0.80$ \\
    \bottomrule
    \end{tabular}
\end{table}

In both dimensions the RaNN residual decays at least as fast as $N^{-1/2}$, which is what Corollary~\ref{cor:PME} asserts. A slope steeper than $-1/2$ is consistent with the corollary, which is an upper bound. The fully trainable network attains a smaller residual at every width, but the gap narrows as $N$ grows: at $N=1280$ the RaNN residual is larger by a factor of $5.9$ in $d=1$ and $1.5$ in $d=2$ \rev{(Student-$t$ features; $5.7$ and $1.7$ with $\pi$)}, and in $d=2$ the relative $L^2$ errors are almost equal. The rate measured for the RaNN is also considerably more reproducible across draws than that of the fully trainable network ($\pm 0.08$ and $\pm 0.12$ \rev{with Student-$t$ features, $\pm 0.07$ and $\pm 0.29$ with $\pi$,} against $\pm 0.78$ and $\pm 0.69$ for the residual). \rev{The two feature laws give essentially the same rates.} This is a consequence of the a-priori construction: with the features frozen the model is linear in its trainable parameters, so the optimisation problem is better behaved than that of the fully trainable network (see also Appendix~\ref{apdx:solvers}). We stress that, unlike the experiments of Section~\ref{sec:experiments-main}, the errors reported here include optimisation and collocation errors in addition to the approximation error bounded by Theorem~\ref{thm1}.
}
\subsection{Navier-Stokes} \label{apdx-NSE}
We now turn to the compressible Navier-Stokes system, given by \eqref{NSE}. We restrict our attention to the one-dimensional case, where \eqref{NSE} is reduced to (in Lagrangian mass coordinates)
\begin{equation} \label{lagrangian}
\begin{cases}
    \partial_t v - \partial_x u = 0, \\[1ex]
    \partial_t u + \partial_x p_\epsilon (v) - \mu \partial_x \left( \frac{1}{v}\partial_x u \right) = 0.
\end{cases}
\end{equation} Here, $v = 1/\rho$ represents the specific volume and $u$ the velocity. We consider the singular pressure in the work of \cite{dalibard2020existence}.
\begin{equation}
    p_\epsilon(v) = \frac{\epsilon}{(v-1)^\gamma}, \qquad \gamma > 0.
\end{equation}
In order to evaluate the performance of the randomised PINN method, we need a baseline solution analogous to the Barenblatt profile (\eqref{barenblatt}) which we used for the Porous Medium Equation. For this purpose, we consider the travelling shock-wave solutions to system \eqref{lagrangian}, which were studied by \cite{dalibard2020existence}. Shock wave solutions to compressible Navier-Stokes models have also been studied in other works (\cite{mascia2004stability,dalibard2021local,humpherys2010stability,vasseur2016nonlinear}). The travelling wave solutions to system \eqref{lagrangian} can be obtained by taking the ansatz $(v,u)(t,x) = (\mathfrak{v}, \mathfrak{u})(x- st)$, where $s$ is the shock speed. This reduces the PDE to an ODE for $\mathfrak{v}$:
\begin{equation}
    \mathfrak{v} ' = \frac{\mathfrak{v}}{\mu s} (s^2 (v_{-} - \mathfrak{v}) + p_\epsilon (v_{-}) - p_{\epsilon}(\mathfrak{v})). \label{ODE}
\end{equation}

The velocity $\mathfrak{u}$ can then be obtained from the relationship $\mathfrak{u} = s(v_+ - \mathfrak{v})$, which follows from the conservation of mass.

We consider the domain $(0,T) \times (-5,5)$ with $T=1.0$, $\mu =10,\epsilon = 10^{-3}, \gamma = 2$. The shock profile connects a far-field state $v_-$ to the far-field state $v_+$. We fix $v_+ = 1.5$, while $v_{-} < v_+$ and the shock speed $s$ are derived from the Rankine-Hugoniot jump condition (see Proposition 1.1 of \cite{dalibard2020existence}).

To obtain the baseline solution $(\mathfrak{v}, \mathfrak{u})$, we numerically integrate the ODE using the scipy.integrate.solve\_ivp solver on an interval $\xi \in[-15,15]$ with $5000$ points. The velocity profile $\mathfrak{u}(\xi)$ is then obtained from $\mathfrak{u} = s(v_+ - \mathfrak{v})$. The travelling wave solution can be seen below in Figure \ref{fig:TW}.
\begin{figure}
    \centering
    \includegraphics[width=\linewidth]{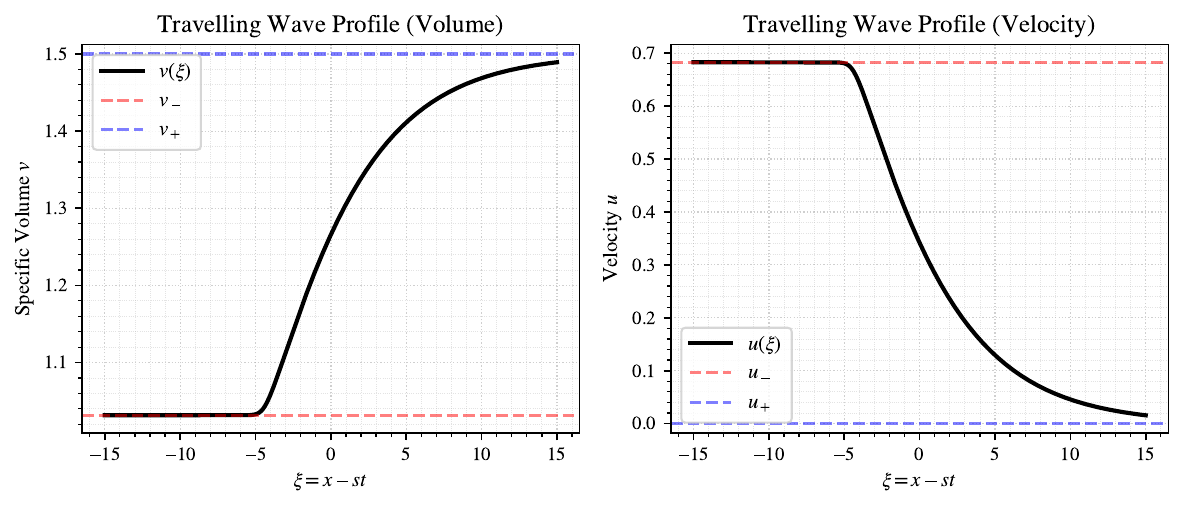}
    \caption{The travelling wave solution $(\mathfrak{v}, \mathfrak{u})$ to $\eqref{lagrangian}$, obtained by solving the ODE \eqref{ODE}.}
    \label{fig:TW}
\end{figure}

In an effort to minimise optimisation error and adhere to the setting of Theorem \ref{thm1} as closely as possible, we choose to use a supervised learning approach when finding a Randomised Neural Network. We include Fourier features (see Appendix \ref{RFNN} for details) where frequencies $\omega_j$ are drawn from $\mathcal{N}(0, 3.5^2)$. Note that we choose to use a smaller variance than for the PME because the solution is of lower frequency than the Barenblatt profile.

We sweep across a range of widths $N \in \{50, ..., 250\}$ and aim to minimise the mean-squared error between the network  and the solution to the ODE. For each width, we train the model using Ridge regression on a dataset where the sample size is $M=2000N$ ($M$ grows with $N$ to avoid an underdetermined problem, ensuring the optimisation problem remains stable). For any given training point $(t_i, x_i)$, the evaluation of the baseline solution $(\mathfrak{v}, \mathfrak{u})$ is obtained by linear interpolation on the ODE grid. We sample uniformly in time but use a mixture of uniform and importance sampling in space; $50\%$ of points are sampled uniformly on $[-5,5]$ whereas $50\%$ of points are sampled from a normal distribution $N(x_0, 1)$ around the shock location $x_0(t) = x_0 + st$.  The frequencies for the Fourier features $\omega_j$ are sampled from $\mathcal{N}(0, 3.5^2)$. We find the smaller variance of $3.5^2$ to be effective for the simpler behaviour of a travelling wave solution.

The network is trained to minimise the $L^2$-regularised MSE (Ridge regression loss) :
\begin{equation} \mathcal{L}(\mathbf{W}) = \frac{1}{M} \sum_{i=1}^M \| \hat{\mathbf{v}}(t_i,\mathbf{x}_i) - \mathbf{y}_i \|_2^2 + \lambda \| \mathbf{W} \|_{2}^{2}, \qquad \lambda = 10^{-10}, \end{equation}
where $\hat{\mathbf{v}}$ is the network output and $\mathbf{W}$ is the vector of output weights. The minimiser $\hat{\mathbf{W}}$ has a closed-form which can be explicitly calculated and used to generate the final network. With the final network, we compute the $L^2$ errors relative to the baseline solution using a set of $20,000$ (pre-generated) points. Each width $N$ is tested five times and the mean relative error is recorded for each $N$. These errors are plotted against $N$ in Figure \ref{fig:NS}. The errors are plotted on a log-log scale in Figure \ref{fig:NS-log}.
\begin{figure}
    \centering
    \includegraphics[width=0.5\linewidth]{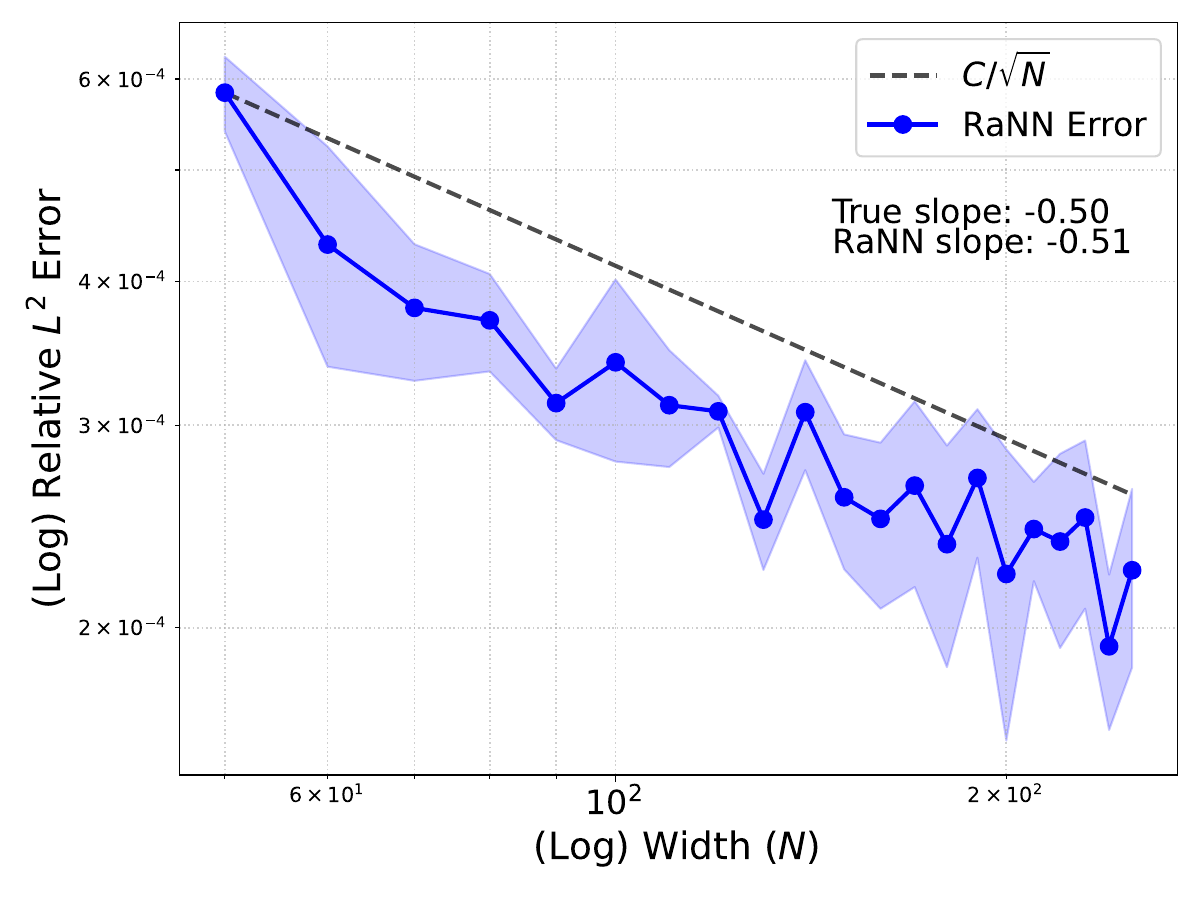}
    \caption{Approximation error of RaNNs of varying width $N$ for solving the compressible Navier-Stokes system on $(0,1) \times [-5,5]$ in logarithmic scale.}
    \label{fig:NS-log}
\end{figure}
\section{Random Feature Neural Networks} \label{RFNN}
It is known that classical PINNs suffer from a spectral bias phenomenon \cite{rahaman2019spectral}, which essentially means that the network is biased towards learning lower frequency functions (see\cite{wang2022respecting, wang2023expert, xu2019frequency}). This can be troublesome, particularly for non-linear PDEs whose solutions are often highly chaotic. \cite{tancik2020fourier} suggested to use Fourier features to overcome this. 

To carry out the simulations in Section \ref{sec:experiments-main}, we integrated Fourier feature embeddings into the RaNN network. We now describe the architecture of a network with Fourier feature embeddings. Instead of choosing a smooth activation such as $\tanh$, we take $\sigma(z) = \cos(z)$ and include both cosine and sine activations for symmetry. Concretely, let $\{\tau_i, \mathbf{a}_i\}_{i=1}^N \subset \mathbb{R}^{1+d}$ and $\{b_i\}_{i=1}^N \subset [0,2\pi]$ be frozen random samples, and define
\begin{equation}
\phi_i(t,x) = \cos(\tau_it + \mathbf{a}_i \cdot x + b_i), \qquad 
\psi_i(t,x) = \sin(\tau_it + \mathbf{a}_i \cdot x + b_i).
\end{equation}
We then construct the feature vector
\begin{equation}
\Phi(t,x) = \frac{1}{\sqrt{N}}\big(\phi_1(t,x),\dots,\phi_N(t,x),\psi_1(t,x),\dots,\psi_N(t,x)\big),
\end{equation}
and take
\begin{equation}\label{uWfourier}
u_W(t,x) = \beta + \sum_{i=1}^{N} \big(\alpha_i \phi_i(t,x) + \gamma_i \psi_i(t,x)\big),
\end{equation}
where the coefficients $\{\alpha_i,\gamma_i\}_{i=1}^N$ and bias $\beta$ are the trainable parameters. The prefactor $N^{-1/2}$ normalises the variance of the features, and does not affect the approximation class.

We can express the sum of sine and cosine functions as a single shifted cosine with amplitude $W_i$, giving us the form
\begin{equation}
u_W(t,x) = \beta + \sum_{i=1}^N W_i \cos(\tau_i t + \mathbf{a}_i \cdot \mathbf{x} + \tilde{b}_i),
\end{equation}
where $W_i := \sqrt{\alpha_i^2 + \gamma_i^2}$, $\theta_i := \operatorname{atan2}(\gamma_i, \alpha_i)$ and $\tilde{b}_i := b_i - \theta_i$. This shows that the RaNN used in our experiments is of the same general form as \eqref{RANN}, with smooth activation $\sigma=\cos$ (the sign of the bias is immaterial since its distribution is symmetric). Recall from Remark \ref{rmk:tanh} that $\sigma = \cos$ is an admissible choice, meaning that the approximation result of Theorem \ref{thm1} directly applies to networks of the form \eqref{uWfourier}.

\rev{
\section{Using the bounds inside non-linear solvers} \label{apdx:solvers}
Theorem~\ref{thm1} bounds the approximation error of the drawn class of RaNNs and does not analyse a particular algorithm for determining the read-out weights. When no reference solution is available, the read-out weights are determined from the equation itself, for instance by least squares on the PDE residual as in the physics-informed extreme learning machine \cite{dwivedi2020physics} and the random feature method \cite{chen2022bridgingtraditionalmachinelearningbased}, by reducing training to the solution of an ODE \cite{datar2024fast}, or within the random deep splitting scheme \cite{neufeld2025full}. Our bounds supply the approximation-error term in the error analysis of each of these methods. We outline two ways of handling the non-linearity of the PDE, and the choice of random weights across iterations.

\paragraph{Direct residual minimisation.} With the hidden weights frozen, a RaNN is linear in its read-out weights $W$. For a linear PDE the residual loss is therefore a convex quadratic in $W$ with a closed-form minimiser. For the PME with $m=2$ the residual is quadratic in the network, so the loss is a polynomial least-squares problem in $W$; it is not convex, but it can be solved by Gauss--Newton methods. Neither structure is available for a fully trainable network, which is non-convex through its parametrisation regardless of the form of the residual.

\paragraph{Picard linearisation.} Alternatively, the non-linear problem can be replaced by a sequence of linear ones. For the PME with $m=2$, given a RaNN $u^{(n)}$ one solves
\begin{equation*}
  \partial_t u^{(n+1)} - \nabla\cdot\big(2u^{(n)}\nabla u^{(n+1)}\big) = 0, \qquad u^{(n+1)}(t_0,\cdot)=u_0,
\end{equation*}
for a RaNN $u^{(n+1)}$. Since the coefficient $u^{(n)}$ is known, each subproblem is linear in the read-out weights of $u^{(n+1)}$ and can be solved by regression.

\paragraph{Random weights across iterations.} For a single regression, Theorem~\ref{thm1} together with Markov's inequality (as in the proof of Corollary~\ref{cor:PME}) gives the error bound with probability at least $1-\delta$ over the random weights. Over $K$ iterations one needs the bound to hold at all $K$ steps simultaneously: applying it at each step with $\delta$ replaced by $\delta/K$, a union bound gives probability at least $1-\delta$. Since the bound is proportional to $1/\delta$, the rate in $N$ is unchanged and only the constant grows, by a factor $K$. This accounting is needed both when a single draw of random weights is reused at every step and when fresh weights are drawn at each step. Reusing one draw allows the feature matrices to be computed once, while redrawing avoids being tied to a draw that is poorly suited to the solution, at the cost of recomputing the feature matrices at each step. Our bounds do not control how the error made at one step propagates to the next; this would require a stability estimate for the linearised scheme and is left for future work.
}

\newpage
\section*{NeurIPS Paper Checklist}

\begin{enumerate}

\item {\bf Claims}
    \item[] Question: Do the main claims made in the abstract and introduction accurately reflect the paper's contributions and scope?
    \item[] Answer: \answerYes{} 
    \item[] Justification: The abstract and introduction claim that we obtain a rigorous approximation bound for sufficiently smooth functions in time-dependent Sobolev spaces by randomised neural networks (RaNNs) with rate $N^{-1/2}$, where $N$ is the width of the RaNN. In Theorem \ref{thm1} and particularly equation \eqref{thm-ineq} we see that the factor of $N$ appearing on the RHS is $N^{-1}$, which when square rooted (because the square of the norm appears on the LHS) gives $N^{-1/2}$, verifying the claim that we prove an approximation result with rate $N^{-1/2}$. Furthermore, since solutions to non-linear PDEs live in time-dependent Sobolev spaces, this approximation result \eqref{thm-ineq} can be applied to such solutions and in particular can be used to obtain bounds on PINN residuals. In Section 4 it is applied to two PDEs where we obtain approximation bounds on the solutions as well as the PINN residuals and carry out experiments to verify the rate.

\item {\bf Limitations}
    \item[] Question: Does the paper discuss the limitations of the work performed by the authors?
    \item[] Answer: \answerYes{}  
    \item[] Justification: We note the limitations through two remarks and the conclusion. In Remark \ref{rmk:smoothness} we note that the space of admissible functions does depend on dimension. In Remark \ref{rmk-growth} we mention that a limitation of our work is that the growth of the coefficient in \eqref{thm-ineq} is difficult to understand, but we obtain a bound that grows at most polynomially in $d$ in a simpler case. In the conclusion we highlight that future work may investigate less smooth solutions (i.e. reducing the assumptions of Theorem \ref{thm1}), using a different distribution for the weights, improving the constant of proportionality and investigating whether the spectral bias that PINNs face carries over to RaNNs.

\item {\bf Theory assumptions and proofs}
    \item[] Question: For each theoretical result, does the paper provide the full set of assumptions and a complete (and correct) proof?
    \item[] Answer: \answerYes{} 
    \item[] Justification: Each result contains all of the necessary assumptions. The full proofs are also given in the appendices.

    \item {\bf Experimental result reproducibility}
    \item[] Question: Does the paper fully disclose all the information needed to reproduce the main experimental results of the paper to the extent that it affects the main claims and/or conclusions of the paper (regardless of whether the code and data are provided or not)?
    \item[] Answer: \answerYes{} 
    \item[] Justification: The methodology is fully specified in Appendix B.1.- B.3. Each of the hyperparameters is given, the loss function and the form of the solution targeted. The code attached to the submission is stand-alone and allows the experiments to be reproduced. The code is publicly available at \url{https://github.com/M6LI/RaNN-Expressivity-NonLinearPDEs}, with instructions in its README. The experiments do not use external datasets.  

\item {\bf Open access to data and code}
    \item[] Question: Does the paper provide open access to the data and code, with sufficient instructions to faithfully reproduce the main experimental results, as described in supplemental material?
    \item[] Answer: \answerYes{} 
    \item[] Justification: The paper references the publicly available repository https://github.com/M6LI/RaNN-Expressivity-NonLinearPDEs where the code can be found. The code is also attached to the submission.

\item {\bf Experimental setting/details}
    \item[] Question: Does the paper specify all the training and test details (e.g., data splits, hyperparameters, how they were chosen, type of optimizer) necessary to understand the results?
    \item[] Answer: \answerYes{} 
    \item[] Justification: The training and test details are fully specified in Appendix B.1. - B.3. The full algorithm details, including each of the hyperparameters is given

\item {\bf Experiment statistical significance}
    \item[] Question: Does the paper report error bars suitably and correctly defined or other appropriate information about the statistical significance of the experiments?
    \item[] Answer: \answerYes{} 
    \item[] Justification: The plots (e.g. Figure 1) include a shaded band representing a region within one standard deviation of the mean relative $L^2$ error. Error bars are also in Tables 1-2.

\item {\bf Experiments compute resources}
    \item[] Question: For each experiment, does the paper provide sufficient information on the computer resources (type of compute workers, memory, time of execution) needed to reproduce the experiments?
    \item[] Answer: \answerYes{} 
    \item[] Justification: We mention in Appendix B that all experiments were run on a NVIDIA RTX 3080 GPU with 10GB VRAM.
    
\item {\bf Code of ethics}
    \item[] Question: Does the research conducted in the paper conform, in every respect, with the NeurIPS Code of Ethics \url{https://neurips.cc/public/EthicsGuidelines}?
    \item[] Answer: \answerYes{} 
    \item[] Justification: We have adhered to the NeurIPS code of ethics in its entirety.

\item {\bf Broader impacts}
    \item[] Question: Does the paper discuss both potential positive societal impacts and negative societal impacts of the work performed?
    \item[] Answer: \answerNA{} 
    \item[] Justification: There is no societal impact of the work performed.
    
\item {\bf Safeguards}
    \item[] Question: Does the paper describe safeguards that have been put in place for responsible release of data or models that have a high risk for misuse (e.g., pre-trained language models, image generators, or scraped datasets)?
    \item[] Answer: \answerNA{} 
    \item[] Justification: The paper poses no such risks.

\item {\bf Licenses for existing assets}
    \item[] Question: Are the creators or original owners of assets (e.g., code, data, models), used in the paper, properly credited and are the license and terms of use explicitly mentioned and properly respected?
    \item[] Answer: \answerNA{} 
    \item[] Justification: The paper does not use existing assets.

\item {\bf New assets}
    \item[] Question: Are new assets introduced in the paper well documented and is the documentation provided alongside the assets?
    \item[] Answer: \answerNA{} 
    \item[] Justification: The paper does not release new assets.

\item {\bf Crowdsourcing and research with human subjects}
    \item[] Question: For crowdsourcing experiments and research with human subjects, does the paper include the full text of instructions given to participants and screenshots, if applicable, as well as details about compensation (if any)? 
    \item[] Answer: \answerNA{} 
    \item[] Justification: The paper does not involve crowdsourcing nor research with human subjects.

\item {\bf Institutional review board (IRB) approvals or equivalent for research with human subjects}
    \item[] Question: Does the paper describe potential risks incurred by study participants, whether such risks were disclosed to the subjects, and whether Institutional Review Board (IRB) approvals (or an equivalent approval/review based on the requirements of your country or institution) were obtained?
    \item[] Answer: \answerNA{} 
    \item[] Justification: The paper does not involve crowdsourcing nor research with human subjects.

\item {\bf Declaration of LLM usage}
    \item[] Question: Does the paper describe the usage of LLMs if it is an important, original, or non-standard component of the core methods in this research? Note that if the LLM is used only for writing, editing, or formatting purposes and does \emph{not} impact the core methodology, scientific rigor, or originality of the research, declaration is not required.
    \item[] Answer: \answerNA{} 
    \item[] Justification: No LLMs used.

\end{enumerate}

\end{document}